\documentclass[preprint]{imsart}

\RequirePackage[OT1]{fontenc}

\arxiv{arXiv:1501.04751}

\usepackage[english]{babel}
 \usepackage{amsrefs}
\usepackage{amsmath,amsfonts,amsthm,amscd,amssymb,mathrsfs}
\usepackage[colorlinks=true,citecolor=blue,urlcolor=blue]{hyperref}
\usepackage{lipsum}
\usepackage{amsmath}
\usepackage{amsfonts}
\usepackage{amsthm}
\usepackage{amssymb}
\usepackage{cases}
\usepackage{bbm}
\usepackage{graphicx}
\usepackage{subfig}
\usepackage{float}
\usepackage{empheq}
\usepackage{mathrsfs}
\usepackage{enumerate}
\usepackage{appendix}
\usepackage{times}
\usepackage{microtype}
\usepackage{mathrsfs}
\usepackage{tikz}

\usepackage{tikz}
\usetikzlibrary{arrows,chains,matrix,positioning,scopes}

\usetikzlibrary{shapes.misc}
\usetikzlibrary{shapes.symbols}
\tikzset{
	xi/.style={circle,fill=blue!10,draw=black,inner sep=0pt,minimum size=1.2mm},
	xib/.style={circle,fill=blue!10,draw=black,inner sep=0pt,minimum size=1.6mm},
	not/.style={circle,fill=black,draw=black,inner sep=0pt,minimum size=0.5mm},
	>=stealth,
	}
\makeatletter
\def\DeclareSymbol#1#2#3{\expandafter\gdef\csname MH@symb@#1\endcsname{\tikz[baseline=#2,scale=0.15]{#3}}}
\def\<#1>{\csname MH@symb@#1\endcsname}

\makeatother

\DeclareSymbol{tree1}{0.5}{\draw (-1,1) node[not] {};}
\DeclareSymbol{tree12}{0.5}{\draw (-0.5,1) node[not] {} -- (0,0.25) node[not]{}; \draw (0.5,1) node[not] {} -- (0,0.25);}
\DeclareSymbol{tree122}{0.5}{\draw (-0.5,0.75) node[not] {} -- (0,0)node[not] {}; \draw (0.5,0.75) -- (0,0); \draw (0,1.5) node[not] {} -- (0.5,0.75)node[not] {}; \draw (1,1.5) node[not] {} -- (0.5,0.75);}
\DeclareSymbol{tree11}{0.5}{\draw (0,1.5) node[not] {} -- (0,0.5); }
\DeclareSymbol{tree21}{0.5}{\draw (0,1.5) node[not] {} -- (0,0.5); \draw (0.5,0.5) node[not] {};}
\DeclareSymbol{tree124}{0.5}{\draw (-0.5,0.75) node[not] {}-- (0,0)node[not] {}; \draw (0.5,0.75)node[not] {} -- (0,0); \draw (0.25,1.5) node[not] {} -- (0.5,0.75); \draw (1,1.5) node[not] {} -- (0.5,0.75); \draw (-1,1.5) node[not] {} -- (-0.5,0.75); \draw (-0.25,1.5) node[not] {} -- (-0.5,0.75);}
\DeclareSymbol{tree1222}{0.5}{\draw (-0.5,0.75) node[not] {} -- (0,0) node[not] {}; \draw (0.5,0.75) -- (0,0); \draw (0,1.5) node[not] {} -- (0.5,0.75)node[not] {}; \draw (1,1.5) -- (0.5,0.75); \draw (0.5,2.25) node[not] {} -- (1,1.5)node[not] {}; \draw (1.5,2.25) node[not] {} -- (1,1.5);}

\startlocaldefs

\newcommand{\eqdef}{\stackrel{\mathclap{\mbox{\tiny def}}}{=}}

\DeclareMathOperator{\prob}{{\mathbb P}}

\makeatletter
\newcommand{\eqcolon}{\mathrel{\mathord{=}\raise.2\p@\hbox{:}}}
\newcommand{\coloneq}{\mathrel{\raise.2\p@\hbox{:}\mathord{=}}}
\makeatother

\newcommand{\dd}{\mathrm{d}}

\newcommand{\1}{\mathbbm{1}}
\newcommand{\E}{\mathbbm{E}}
\newcommand{\loc}{\mathrm{loc}}

\newcommand{\Q}{\mathbb Q}
\newcommand{\R}{\mathbb R}

\newcommand{\Z}{\mathbb Z}
\newcommand{\XX}{\mathbb{X}}

\newcommand{\TT}{\mathbb{T}}

\newcommand{\D}{\mathcal{D}}
\newcommand{\CF}{\mathcal{F}}
\newcommand{\CG}{\mathcal{G}}
\newcommand{\CH}{\mathcal{H}}
\newcommand{\I}{\mathcal I}
\newcommand{\CJ}{\mathcal{J}}
\newcommand{\CK}{\mathcal{K}}
\newcommand{\CS}{\mathcal{S}}

\newcommand{\CV}{\mathcal{V}}
\newcommand{\XR}{\mathcal{X}}

\newcommand{\cb}{\mathscr{B}}

\newcommand{\CC}{\mathscr{C}}

\newcommand{\DD}{\mathscr{D}}

\newcommand{\cF}{\mathscr F}
\newcommand{\cg}{\mathscr G}
\newcommand{\ch}{\mathscr H}

\newcommand{\cs}{\mathscr S}

\newcommand{\cX}{\mathscr{X}}

\newcommand{\eps}{\varepsilon}
\newcommand{\der}{\delta}

\theoremstyle{plain}
\newtheorem{lemma}{Lemma}[section]
\newtheorem{theorem}[lemma]{Theorem}
\newtheorem{corollary}[lemma]{Corollary}
\newtheorem{proposition}[lemma]{Proposition}

\theoremstyle{definition}
\newtheorem{definition}[lemma]{Definition}
\newtheorem{remark}[lemma]{Remark}

\newtheorem*{notation*}{Notation}
\newtheorem*{plan*}{Plan of the paper}
\newtheorem*{ackno*}{Acknowledgements}

\numberwithin{equation}{section}

\endlocaldefs

\begin{document}

\begin{frontmatter}
\title{Multidimensional SDEs with singular drift \\and universal construction of the polymer measure with white noise potential}
%\runtitle{A Sample Document}
%\thankstext{T1}{Footnote to the title with the ``thankstext'' command.}

\begin{aug}
\author{\fnms{Giuseppe} \snm{Cannizzaro}\thanksref{t1}\ead[label=e1]{giuse.cannizzaro@gmail.com}}
\and
\author{\fnms{Khalil} \snm{Chouk}\thanksref{}\ead[label=e2]{khalilchouk@gmail.com}}
%\and
%\author{\fnms{Third} \snm{Author}\thanksref{t1,m2}
%\ead[label=u1,url]{http://www.foo.com}}

\thankstext{t1}{This work was carried out while G.C. was a PhD student at Technische Universit\"at Berlin}
%\thankstext{t2}{First supporter of the project}
%\thankstext{t3}{Second supporter of the project}
%\runauthor{F. Author et al.}

\affiliation{University of Warwick\thanksmark{m1} and Technische Universit\"at Berlin\thanksmark{m2}}

%\address{ Gibbet Hill Rd, Coventry CV4 7AL, UK\\
%Str. des 17. Juni 136, Berlin 10623, DE\\
%\printead{e1}\\
%\phantom{E-mail:\ }\printead*{e2}}

%\address{Address of the Third author\\
%Usually a few lines long\\
%Usually a few lines long\\
%\printead{e3}\\
%\printead{u1}}
\end{aug}

\begin{abstract}
We study existence and uniqueness of solution for stochastic differential equations with distributional drift by giving a meaning to the Stroock-Varadhan martingale problem associated to such equations. The approach we exploit is the one of paracontrolled distributions introduced in~\cite{GIP15}. As a result we make sense of the three dimensional polymer measure with white noise potential.
\end{abstract}

\begin{keyword}[class=MSC]
\kwd[Primary ]{60K35}
\kwd{60K35}
\kwd[; secondary ]{60K35}
\end{keyword}

\begin{keyword}
\kwd{SDEs with Rough Drift}
\kwd{Singular SPDEs}
\kwd{Paracontrolled Calculus}
\kwd{Polymer Measure}
\end{keyword}

\end{frontmatter}

%\author{Giuseppe Cannizzaro\footnote{TU Berlin, Institut f\"ur Mathmatik, Strasse 17. des Juni 136, D-10623. \textit{giuse.cannizzaro@gmail.com}} , Khalil Chouk\footnote{TU Berlin, Institut f\"ur Mathmatik, Strasse 17. des Juni 136, D-10623. \textit{khalilchouk@gmail.com}}}

{\tableofcontents}

\section{Introduction}

The aim of the present paper is to give a meaning to Stochastic Differential Equations (SDEs) of the form
\begin{equation}\label{e:SDE}
\dd X_t=V(t,X_t)\dd t+\dd B_t,\qquad X_0=x
\end{equation}
where $B$ is a $d$-dimensional Brownian motion, $x$ a point in $\R^d$ and $V$ is a function of time taking values in the space of distributions $\cs'(\R^d,\R^d)$. 
Of course, as it is written,~\eqref{e:SDE} does not make any sense unless we impose certain restrictions concerning the regularity or integrability (or both) of the drift $V$. 

The case of $V$ being a smooth enough vector-field has been deeply investigated and is nowadays well-understood. Upon assuming $V\in L_{\loc}^p((0,+\infty)\times \R^d)$ for $p>d+2$, it is still possible to obtain local pathwise existence and uniqueness as shown in~\cite{KR05}. 
When $V$ is an effective distribution, the majority of results deals with the time-homogenous situation (i.e. $V$ is taken to be independent of time), see for example~\cite{BC01,FRW03,FRW04}, and existence and uniqueness can be determined either in the weak or strong sense, depending on the interplay between its regularity and integrability. 

When $V\in C([0,T],\cs'(\R^d,\R^d))$ with a non-trivial dependence on time, the picture becomes even more blurred, since it is already unclear how to define a convenient notion of solution. 
Nevertheless some advances have been recently made in~\cite{FFF}, where the authors investigate the case of a time dependent distributional drift taking values in a class of Sobolev spaces with negative derivation order on $\R^d$. 
\newline

Our attempt is to generalize the work of F. Delarue and  R. Diel. In~\cite{FR}, they construct solutions to SDEs with $V(t,\cdot)=\partial_xY(t,\cdot)$ and $Y$ a  $(1/3+\eps)$-H\"older function in space on some interval $I\subseteq \R$, by formulating a Stroock-Varadhan martingale problem for~\eqref{e:SDE}. What we aim at is to go beyond the one dimensional case and consider a distributional drift on $\R^d$ for $d\geq1$. More precisely we study the case of $V\in C([0,T], \CC^\beta(\R^d,\R^d))$ for $\beta<0$, where $\CC^\beta(\R^d,\R^d)$ is the Besov-H\"older space of distributions on $\R^d$ (see~\eqref{eq:Besov-space} for the exact definition). 

In the same spirit as~\cite{FR}, we prove well-posedness for the martingale problem corresponding to the generator $\mathscr G^V$ of the diffusion~\eqref{e:SDE}, which is given by
\begin{equation}
\cg^V=\partial_t+\frac{1}{2}\Delta+V\cdot\nabla.
\end{equation}
In general, one would want to say that a probability measure $\prob $  on $\Omega=C([0,T],\mathbb R^d)$, endowed with the usual Borel $\sigma$-algebra $\mathscr{B}(C([0,T],\R^d))$, solves the martingale problem related to $\mathscr G^V$ starting at $x$, if the canonical process $X$, $X_t(\omega)=\omega(t)$, satisfies 
\begin{enumerate}
\item $\prob(X_0=x)=1$,
\item  for any $T^\star\leq T$ and  $\varphi\in\D$, where $\D$ is a set of functions on $[0,T^\star]\times \R^d$, the process 
\begin{equation}\label{eq:sv-martingale1}
\Big\{\varphi(t,X_t)-\int_0^t(\mathscr G^V\varphi)(s,X_s)\dd s\Big\}_{0\leq t\leq T^\star}
\end{equation}
is a square integrable martingale with respect to $\prob$.
\end{enumerate}
The problem here lies in the fact that if we choose $\D$ simply as the space of smooth functions and $V\in C([0,T], \CC^\beta(\R^d,\R^d))$, with $\beta<0$, then $\mathscr G^V\varphi$ is not a function anymore but a distribution (with the same regularity as $V$) and, once again, it is not clear what meaning to attribute to $(\mathscr G^V\varphi)(s,X_s)$. 
The point here is that we need to determine a suitable domain $\D$ for which $\mathscr G^V\varphi$ is a continuous function of time, bounded in space. 
In other words, we need to solve the following Partial Differential Equation (PDE), that we will refer to as the {\it generator equation},
\begin{equation}\label{eq:generator11}
\mathscr G^V\varphi=f,\qquad\varphi(T,\cdot)=\varphi^T
\end{equation}
for $f\in C([0,T],L^\infty(\R^d))$ and a sufficiently large class of terminal conditions $\varphi^T$. 
Once this is done, we can replace the assertion~\eqref{eq:sv-martingale1} with the requirement that the process 
\begin{equation}\label{e:MartingaleNew}
\Big\{\varphi(t,X_t)-\int_0^tf(s,X_s)\dd s\Big\}_t
\end{equation}
is a square integrable martingale for every $f\in C([0,T],L^\infty(\R^d))$ and $\varphi$ the solution of~\eqref{eq:generator11}. 

However, PDEs of the type~\eqref{eq:generator11}, assuming $\beta\in(-\frac{2}{3},0)$, cannot be classically handled since the presumed solution is not expected to be smooth enough to allow to define the ill-posed term $V\cdot\nabla \varphi$. 
To bypass it, F. Delarue and R. Diel in~\cite{FR} adopt the technique exploited by M.Hairer in~\cite{hairer_solving_2013} and, more precisely, they make use of Lyons Rough Path theory to interpret the ill-defined product as a rough integral.

Despite the possibility of overcoming the well-posedness issues, rough path theory has the dramatic disadvantage of being crucially attached to the one parameter setting so that there is simply no hope to go beyond the one-dimensional case with those techniques.

This is precisely the point in which the paracontrolled distributions approach, developed in~\cite{GIP15} (or alternatively the Theory of Regularity Structures~\cite{hairer_theory_2013}), comes into play. In this context though, the possibility of solving equations that are not classically well-posed comes at a ``price". 
More specifically, in case $\beta\in(-\frac{2}{3},-\frac{1}{2}]$, we are not allowed to take {\it any} $V\in C([0,T],\CC^\beta(\R^d,\R^d))$ but only those that can be {\it enhanced} to a rough distribution, $\CV$ (see Definition~\ref{def:rough}). In other words, we need to be able to build in some way, starting from $V$, an additional object satisfying suitable regularity requirements but depending {\it only} on $V$ itself. 

We refrain from detailing the construction here and we limit ourselves to loosely state the result. 

\begin{theorem}\label{t:GenEq}
Let $\beta\in(-\frac{2}{3},0)$, $\gamma\in(0,\beta+2)$ and $V\in C([0,T],\CC^\beta(\R^d,\R^d))$. If  $\beta\in(-\frac{2}{3},-\frac{1}{2}]$, assume further that $V$ can be enhanced to a rough distribution $\CV$. Then, there exists a non-trivial Banach space, $\D\subseteq C([0,T], \CC^\gamma(\R^d))$, such that for any $\varphi^T\in\CC^\gamma(\R^d)$ and $f\in C([0,T],L^\infty(\R^d))$,~\eqref{eq:generator11} admits a unique solution in $\D$. Moreover, the map assigning to $\varphi^T$, $f$ and $\CV$ the solution to the generator equation is jointly locally Lipschitz continuous. 
\end{theorem}

If we now formulate the Stroock-Varadhan martingale problem for the SDE~\eqref{e:SDE}, by requiring point 1. stated before and~\eqref{e:MartingaleNew} to be a square integrable martingale for every $f\in C([0,T],L^\infty(\R^d))$, with $\varphi\in\D$ and $\D$ the Banach space determined in the previous theorem, then we can indeed prove its well-posedness. 

\begin{theorem}\label{t:MartProb}
Let $\beta\in(-\frac{2}{3},0)$ and $V\in C([0,T],\CC^\beta(\R^d,\R^d))$. If  $\beta\in(-\frac{2}{3},-\frac{1}{2}]$, assume further that $V$ can be enhanced to a rough distribution $\CV$. 
Then, there exists a unique probability measure $\prob$ which solves the martingale problem with generator $\mathscr G^V$ starting at $x$ (as described above), for every $x\in\R^d$. 
%Moreover, the canonical process under $\prob$, $X_t(\omega)=\omega(t)$, is strong Markov.
\end{theorem}

The natural question at this point is if and when it is possible to build, given $V\in C([0,T],\CC^\beta(\R^d,\R^d))$, its enhancement $\CV$. 
The examples are various (for $d=1$, the ones described in~\cite[Section 5]{FR} would do) but probably one of the most interesting cases is the one that allows to construct the $2$ and $3$ dimensional polymer measure with white noise potential. 

The Polymer measure with white noise potential is a singular measure on the space of continuous functions that is formally given by
\begin{equation}\label{eq:Polymer1}
\mathbb Q_T(\dd\omega)=Z_0^{-1}\exp\left(\int_0^T\xi(\omega_s)\dd s\right)\mathbb W_T(\dd\omega)
\end{equation}
where $\mathbb W$ is the Wiener measure on $C([0,T],\R^d)$, $d=2,3$, $\xi$ a spatial white noise on the d-dimensional torus $\mathbb T^d$ independent of $\mathbb W$, and $Z_0$ is an infinite renormalization constant. 

As it is written, the expression in~\eqref{eq:Polymer1} is of course senseless since we are exponentiating the integral in time of a white noise, which is a distribution, over a Brownian path and dividing then by an infinite constant, all operations that require to be given a meaning to. 

Even if seemingly unrelated, we will see that, {\it if} it were well-posed, under the polymer measure the canonical process, $X_t(\omega)=\omega_t$ has the same law as the solution to the SDE given by
\begin{equation}\label{eq:SDEPol}
\dd X_t=\nabla h(T-t,X_t)\dd t + \dd B_t
\end{equation}
where $B$ is a brownian motion with respect to $\mathbb W$ and $h$, the solution to the KPZ-type equation
\begin{equation}\label{e:KPZType1}
\partial_t h=\frac{1}{2}\Delta h+\frac{1}{2}|\nabla h|^2+\xi,\,\qquad h(0,\cdot)=0
\end{equation}
in which $\xi$ is the same space white noise as the one appearing in~\eqref{eq:Polymer1}. 
Summarizing, if we are able to describe the law of~\eqref{eq:SDEPol} then we can also give a quenched description of the infinitesimal dynamics of the polymer itself, in other words, make sense of it. 

It is not difficult to guess, from the KPZ-type equation above, that $\nabla h$ has regularity slightly less than $0$ in dimension $2$ and slightly less than $-\frac{1}{2}$ in dimension $3$ thus, in principle, falling into the scope of Theorem~\ref{t:MartProb}. 
But of course to be able to apply it, we will need to prove well-posedness of~\eqref{e:KPZType1}, which is non-trivial given the singularity of the noise, and, for this, we will exploit once more the paracontrolled distribution approach. 

Once local existence and uniqueness for the previous Stochastic Partial Differential Equation (SPDE) is established and one has shown that, in $d=3$, $V(t,\cdot)\eqdef \nabla h(T-t,\cdot)$ can be enhanced to a rough distribution, we obtain the following result. 

\begin{theorem}\label{t:ConstrPol}
Let $\xi_\eps$ be a mollified version of the noise and $\mathbb Q_T^\eps$ the polymer measure defined in~\eqref{eq:Polymer1} with $\xi_\eps$ replacing $\xi$. Then, there exists a measure $\mathbb Q_T$ and $T^\star=T^\star(\xi)>0$, independent of the choice of the mollifier, such that for all $T< T^\star$, $\mathbb Q^\eps_T\Longrightarrow\mathbb Q_T$.
\end{theorem}

The last part of our work will consist in determining some of the properties of the Polymer Measure built in the previous theorem. 
At first notice that, the construction above is {\it local} in the sense that we can prove that the measure formally given in~\eqref{eq:Polymer1} exists only up to a possibly finite explosion time $T^\star$, depending, in principle, on the features of the noise. 
We want to show that such an explosion does not occur. Our proof relies on the strict positivity of the solution to the Parabolic Anderson Equation (PAM), formally given by
\[
\partial_t u=\frac{1}{2}\Delta u +u \xi
\]
with initial condition identically equal to 1 and we provide a novel proof of this in Section~\ref{section:pam} valid for both $d=2$ and $3$. 

At last, looking at the way in which the Polymer measure~\eqref{eq:Polymer1} is written, it might seem that $\mathbb Q_T$ is absolutely continuous with respect to the Wiener one. This is definitely not the case. 
In principle, since $\mathbb Q_T$ is the measure describing the law of the solution to~\eqref{eq:SDEPol}, looking at the SDE one guesses (correctly) that the drift cannot be of Cameron-Martin type. 

The actual proof does not make use of the previous heuristics but instead focuses on the renormalization properties of~\eqref{e:KPZType1} so that in the end we have the following statement.

\begin{theorem}\label{thm:PropPol}
In the assumptions of Theorem~\ref{t:ConstrPol}, let $T^\star$ and $\mathbb Q_T$ be as stated above. 
Then, in both dimensions $d=2$ and $3$, $T^\star$ can be chosen to be $+\infty$ and the measure $\mathbb Q_T$ is singular with respect to the Wiener one.  
\end{theorem}

As a last remark, we point out that the construction of the Polymer measure we carried out before is rather universal in the sense that it does not rely on the specific features of the noise. Indeed, given that we are able to prove well-posedness of an equation of the type~\eqref{e:KPZType1} driven by a generic noise $\tilde \xi$ then the same arguments apply. 

The same holds true for the proof of the singularity. For the continuous directed random polymer, i.e. the one formally given by the expression~\eqref{eq:Polymer1}, but with a space-time white noise in spatial dimension 1, an analogous result was obtained in~\cite{AKQ14}. 
Our proof follows a completely different approach, which in turn can be straightforwardly adapted to recover their result.

\begin{plan*} 
In Section~\ref{sec:FunctionSpaces}, we introduce Besov spaces and the main elements of paracontrolled calculus that will be needed in the rest of the paper. Section~\ref{sec:generatorequation} is dedicated to the generator equation. 
We prove that it admits a unique solution and that the flow is a locally Lipschitz map. As anticipated in the introduction, this is then crucial for Section~\ref{sec:MartProblem}, in which we define the Martingale problem associated to the SDE~\eqref{e:SDE} and prove its well-posedness. 
The last three sections are devoted to the Polymer measure: its construction (Section~\ref{sec:Polymer}), the KPZ-type equation thanks to which it is possible (Section~\ref{section:KPZ}) and its properties (Section~\ref{sec:PropPol}). 
\end{plan*}

\begin{ackno*}
We want to thank Professors P.K.Friz and M.Gubinelli for the numerous discussions and fruitful advices. While this work was developed GC was funded by the RTG 1845 and Berlin Mathematical School. KC is supported by  European Research Council under
the European Union's Seventh Framework Programme (FP7/2007-2013) / ERC grant
agreement nr. 258237.
\end{ackno*}

\begin{notation*}
We collect here some notations we will use throughout the paper. In the present work we will always consider functions/distributions on $\R^d$, $d\geq 1$ arbitrary but fixed, with values in $\R^n$, so, in order to lighten the notations, if $B(\R^d,\R^n)$ is a space of functions/distributions from $\R^d$ to $\R^n$, then we will denote by
\[
B_{\R^n}\eqdef B(\R^d,\R^n)\,,\qquad\text{and}\qquad B\eqdef B(\R^d,\R)
\]
Let $\delta\geq 0$, $\eta\in\R$, $T>0$ and $\bar T\in[0,T)$. Let $(D,\,\|\cdot\|_D)$ be a Banach space and $\zeta,\,\bar \zeta:[T-\bar T,T]\to B$ be two functions. We will say that $\zeta\in C^\delta_{\eta,\bar T, T}D$ and $\bar\zeta\in C_{\eta,\bar T,T} D$ if $\|\zeta\|_{C^{\delta}_{\eta,\bar T,T}D}<\infty$ and $\|\bar\zeta\|_{C_{\eta,\bar T, T}D}<\infty$ respectively, where
\small\begin{align*}
\|\zeta\|_{C^{\delta}_{\eta,\bar T, T}D}&\eqdef \sup_{s<t\in(T-\bar T,T]}(T-t)^{\frac{\eta}{2}}\frac{\|f(t)-f(s)\|_D}{|t-s|^\delta}\,,\\
\|\bar\zeta\|_{C_{\eta,\bar T, T}D}&\eqdef \sup_{t\in(T-\bar T,T]}(T-t)^{\frac{\eta}{2}}\|f(t)\|_D\,.
\end{align*}\normalsize
In case the norm on $\zeta$ does not depend on $\eta$, i.e. $\eta=0$, or $\bar T=T$, we will simply remove the corresponding subscript. 
\newline

We will say that $a\lesssim b$ if there exists a constant $C>0$ such that $a\leq C b$. 
\end{notation*}

\section{Besov Spaces and Paracontrolled Calculus}\label{sec:FunctionSpaces}

In this first paragraph we want to introduce the definition of the function spaces we will be using throughout the rest of the work and recall the main ingredients of the paracontrolled calculus\footnote{For a thorough introduction on Besov Spaces see~\cite{BCD11}, or~\cite{GIP15} for the main definitions and properties we will use from now on.} . 

Let $\chi,\, \varrho\in \DD$ be non-negative radial functions such that
\begin{enumerate}
	\item The support of $\chi$ is contained in a ball and the support of $\varrho$ is contained in an annulus;
	\item $\chi(\xi)+\sum_{j\ge0}\varrho(2^{-j}\xi)=1$ for all $\xi\in\R^d$;
	\item $\text{supp}(\chi) \cap \text{supp}(\varrho(2^{-j}.)) = \emptyset$ for $i \geq 1$ and $\text{supp}(\varrho(2^{-i}.)) \cap \text{supp}(\varrho(2^{-j}.)) = \emptyset$ when $|i-j| > 1$.
\end{enumerate}
$(\chi,\varrho)$ satisfying the above properties are said to form a dyadic partition of unity. For the existence of a dyadic partition of unity see \cite[Proposition 2.10]{BCD11}.  

Let now $\cF$ denote the Fourier transform and $(\chi,\varrho)$ be a dyadic partition of unity. Then, the Littlewood-Paley blocks are defined as
\[
\Delta_{-1} u = \cF^{-1}(\chi \cF u),\qquad  \Delta_j u = \cF^{-1}(\varrho_j(\cdot)\cF u)\quad \mathrm{for}\ j \ge 0
\]
where $\varrho_j(\cdot)\eqdef \varrho(2^{-j}\cdot)$ and, for $\alpha\in\R$, $p,\,q\in[1,+\infty]$, the Besov space $B_{p,q}^\alpha(\R^d,\R^n)$ is
\small\begin{equation}\label{eq:Besov-space}
B_{p,q}^{\alpha}(\R^d,\R^n)=\Big\{u\in \cs'(\R^d,\R^n); \quad \|u\|^q_{B_{p,q}^{\alpha}}=\sum_{j\geq-1}2^{jq\alpha}\|\Delta_ju\|^q_{L^p(\R^d,\R^n	)}<+\infty\Big\}.
\end{equation}\normalsize
We will often deal with the special case $p=q=\infty$, so we set $\CC^{\alpha}(\R^d,\R^n)\eqdef B_{\infty,\infty}^{\alpha}(\R^d,\R^n)$ and denote by $\|u\|_{\alpha}=\|u\|_{B_{\infty,\infty}^{\alpha}}$ its norm. 
Such a notation is also justified by the fact that, for non-integer $\alpha>0$, $\CC^\alpha(\R^d,\R^n)$ coincides with the usual space of $\alpha$-H\"older continuous functions. 

In order to manipulate stochastic terms and exploit properties of the elements in Wiener chaos, we will bound their norm in Besov spaces with finite $p=q$ and then get back to the space $\CC^\alpha$. 
To do so, the following Besov embedding will prove to be fundamental.
\begin{proposition}\label{proposition:Bes-emb} 
Let $1\leq p_1\leq p_2\leq +\infty$ and $1\leq q_1\leq q_2\leq +\infty$. For all $s\in \R$ the space $B_{p_1,q_1}^{s}$ is continuously embedded in $B_{p_2,q_2}^{s-d(\frac{1}{p_1}-\frac{1}{p_2})}$, in particular we have $\|u\|_{\alpha-\frac{d}{p}}\lesssim\|u\|_{B_{p,p}^{\alpha}}$.
 \end{proposition}

\subsection{Operations with Besov-H\"older distributions}

Let $f,\,g$ be two distributions in $\cs'(\R^d)$. Upon using the Littlewood-Paley decomposition of $f$ and $g$, we can formally write their product as  
$$
fg=f\prec g+f\circ g+f\succ g
$$
where the first and the last summand at the right hand side are called {\it paraproducts} while the second {\it resonant} term, and they are respectively defined by
$$
f\prec g=g\succ f=\sum_{j\geq-1}\sum_{i<j-1}\Delta_if\Delta_jg\quad\text{and}\quad f\circ g=\sum_{j\geq-1}\sum_{|i-j|\leq 1}\Delta_if\Delta_jg.
$$
With these notations at hand, we can state the following proposition.

\begin{proposition}[Bony's Estimates,~\cite{BCD11}]\label{prop:bony}   
Let $\alpha$, $\beta\in\R$. Let $f\in\CC^\alpha$ and $g\in\CC^\beta$, 
\begin{itemize}
\item if $\alpha\geq0$, then $f\prec g \in\CC^\beta$ and  $\|f\prec g\|_{\beta}\lesssim \|f\|_{L^\infty}\|g\|_\beta$
\item if $\alpha<0$, then $f\prec g \in\CC^{\alpha+\beta}$ and  $\|f\prec g\|_{\alpha+\beta}\lesssim \|f\|_{\alpha}\|g\|_\beta$
\item if $\alpha+\beta>0$, then $f\circ g \in\CC^{\alpha+\beta}$ and $\|f\circ g\|_{\alpha+\beta}\lesssim \|f\|_{\alpha}\|g\|_\beta$
\end{itemize}
\end{proposition}

\noindent Summarizing, the previous proposition tells us that the product of \textit{general} $f\in\CC^\alpha$ and $g\in\CC^\beta$ is well-defined if and only if $\alpha+\beta>0$ and in this case $fg\in\CC^\delta$, where $\delta=\min\{\alpha,\beta,\alpha+\beta\}$ (see~\cite[Lemma 2.1]{GIP15}, for the proof in this specific context).

One of the key results of the paracontrolled analysis carried out in~\cite{GIP15}, is a commutation relation between the operators $\prec$ and $\circ$, that we here recall (see~\cite[Lemma 2.4]{GIP15}). 

\begin{proposition}[Commutator Lemma]\label{prop:comm}
Let $\alpha, \beta,\gamma\in\mathbb R$ be such that $\alpha\in(0,1)$,  $\alpha+\beta+\gamma>0$ and $\beta+\gamma<0$. Then, for $f,\,g$ and $h$ smooth, the operator
$$
\mathscr R(f,g,h)=(f\prec g)\circ h-f(g\circ h)
$$
allows for the bound
$$
\|\mathscr R(f,g,h)\|_{\alpha+\beta+\gamma}\lesssim\|f\|_{\alpha}\|g\|_{\beta}\|h\|_{\gamma}
$$
hence, it can be uniquely extended to a bounded trilinear operator on $\CC^\alpha\times\CC^\beta\times\CC^\gamma$. 
\end{proposition}
\noindent In the following Proposition, which summarizes~\cite[Lemma 2.5]{CC13} and~\cite[Lemma A.8]{GIP15}, we describe the action of the heat kernel on Besov-H\"older functions and its relation with the paraproduct. 
 
\begin{proposition}[Schauder's Estimates]\label{prop:Schauder}
Let $P_t=e^{\frac{1}{2}t\Delta}$ be the heat flow, $\theta\geq0$ and $\alpha\in\R$. Let $f\in\CC^\alpha$ and $0\leq s<t$ then we have
$$
\|P_tf\|_{\alpha+2\theta}\lesssim t^{-\theta}\|f\|_{\alpha}\quad\text{and}\quad \|(P_{t-s}-\mathrm{Id})f\|_{\alpha-2\theta}\lesssim|t-s|^{\theta}\|f\|_{\alpha}\,.
$$
If $\alpha\in[0,1]$, the latter bound becomes
\[
\|\big(P_{t-s} -\mathrm{Id}\big)f\|_{L^\infty}\lesssim |t-s|^{\frac{\alpha}{2}} \|f\|_\alpha
\]
Moreover if $\alpha<1$ and $\beta\in\R$, the following commutator estimate holds 
\begin{equation}\label{commSchauder}
\|P_t(f\prec g)-f\prec P_tg\|_{\alpha+\beta+2\theta}\lesssim t^{-\theta}\|f\|_\alpha\|g\|_\beta
\end{equation}
for all $g\in\CC^\beta$.
\end{proposition}
\noindent For notational convenience, let us define $\I(f)(t) \eqdef \int_0^t P_{t-s}f(s)\dd s$, where the operator $P_t$ was introduced in Proposition~\ref{prop:Schauder}. 
Since we will be working with functions exploding at a certain rate as $t$ goes to $0$ and we will need to understand what happens when we convolve them with the heat kernel, we collect in the following corollary some simple results. 

\begin{corollary}\label{cor}
Let $t\in[0,T]$, $\alpha,\,\beta\in\R$, $\gamma,\,\delta\in[0,1)$, $\gamma'\in(0,\gamma]$ and $\varepsilon\in(0,1]$. Let $f\in C_{\eta,T}\CC^\alpha$. Then,
\begin{enumerate} 
\item if $\frac{\alpha-\beta}{2}>-1$ and $\vartheta\eqdef\frac{\alpha-\beta}{2}-\gamma+\delta+1>0$, we have
\[
t^\delta \|\I(f)(t)\|_\beta\lesssim T^\vartheta \sup_{s\in[0,T]} s^\gamma\|f(s)\|_\alpha
\]
\item if $\frac{\alpha-\varepsilon}{2}>-1$, $\gamma'<\delta$, $\frac{\alpha-\varepsilon}{2}-\gamma+\delta+1>0$ and $0\leq s<t$, we have
\[
s^\delta \frac{\|\I(f)(t)-\I(f)(s)\|_{L^\infty}}{|t-s|^\frac{\varepsilon}{2}}\lesssim T^\vartheta\sup_{s\in[0,T]} s^\gamma\|f(s)\|_\alpha
\]
where $\vartheta=\delta-\gamma$ if $\delta>\gamma$ and $\vartheta=\delta-\gamma'$ otherwise.
\end{enumerate}
\end{corollary}

\begin{proof} The proof is a rather straightforward application of Proposition~\ref{prop:Schauder}. %, so we omit it for the sake of conciseness (see Corollary 2.5 in~\cite{CC15}). 
Indeed, for 1. we have
\small\begin{align*}
t^\delta\|\I(f)(t)\|_\beta &\lesssim t^\delta\int_0^t \|P_{t-s} f(s)\|_{\beta}\dd s\lesssim t^\delta\int_0^t (t-s)^{\frac{\alpha-\beta}{2}} s^{-\gamma}\dd s\sup_{s\in[0,t]} s^\gamma\|f(s)\|_\alpha \\
%&\lesssim t^{\frac{\alpha-\beta}{2}-\gamma+\delta+1}\int_0^1 (1-x)^{\frac{\alpha-\beta}{2}} x^{-\gamma}\dd x\sup_{s\in[0,T]}s^\gamma\|f(s)\|_\alpha\\
&\lesssim T^{\frac{\alpha-\beta}{2}-\gamma+\delta+1} \sup_{s\in[0,T]} s^\gamma\|f(s)\|_\alpha
\end{align*}\normalsize
where the last line is justified by the fact that, since $\frac{\alpha-\beta}{2}>-1$ and $\gamma<1$, the integral is finite and $\frac{\alpha-\beta}{2}-\gamma+\delta+1>0$.

For the second part, the quantity on the left hand side is bounded by
\small\begin{multline*}
 \frac{s^\delta}{|t-s|^\frac{\varepsilon}{2}}\bigg(\int_s^t\|P_{t-r}f(r)\|_{\alpha+2-\varepsilon} \dd r + \big\|(P_{t-s}-\textrm{Id})\int_0^s P_{s-r}f(r)\dd r\|_{L^\infty}\bigg)\\
\lesssim \bigg(\frac{s^\delta}{|t-s|^\frac{\varepsilon}{2}}\int_s^t (t-r)^{\frac{\varepsilon}{2}-1} r^{-\gamma}\dd r +T^{\frac{\alpha-\varepsilon}{2}-\gamma+\delta+1}\bigg)\sup_{s\in[0,T]} s^\gamma \|f(s)\|_\alpha 
\end{multline*}\normalsize
where the last inequality follows by applying Proposition~\ref{prop:Schauder} first and the previous result then. In order to conclude, let us take a better look at the integral appearing on the right hand side of the previous. If $\delta>\gamma$ we have
\small\[
\int_s^t (t-r)^{\frac{\varepsilon}{2}-1} r^{-\gamma}\dd r \lesssim s^{-\gamma}\int_s^t (t-r)^{\frac{\varepsilon}{2}-1}\dd r \lesssim s^{-\gamma} (t-s)^{\frac{\varepsilon}{2}}
\]\normalsize
On the other hand if $\delta\leq \gamma$, upon setting $r=s+x(t-s)$ the integral on the left hand side of the previous becomes
\small\begin{multline*}
 (t-s)^\frac{\varepsilon}{2}\int_0^1 (1-x)^{\frac{\varepsilon}{2}-1}(s+x(t-s))^{-\gamma+\gamma'} (s+x(t-s))^{-\gamma'}\text{d}x\lesssim (t-s)^\frac{\varepsilon}{2} s^{-\gamma'}% \int_0^1 (1-x)^{\frac{\varepsilon}{2}-1}x^{-\gamma+\gamma'}\text{d}x
\end{multline*}\normalsize
and the latter integral is finite. From this, the conclusion immediately follows. 
\end{proof}

\begin{remark}
The reader should keep in mind that the convolution with $P_t$ allows to gain $\theta\geq0$ regularity in space at the price of an explosion as time goes to $0$ of order $\theta/2$. One has then to adjust the choice of the parameters so that they fit the assumptions of the previous Corollary and this will be done implicitly throughout the paper not to heavy the presentation.
\end{remark}

%%%%%%%%%%%%%%%%%%%%%%%%%%%%%%%%%%%%%%%%%%%%%%%%%%%%%%%%%%%%%%%%%%%%%%%%%%
%%% THE GENERATOR EQUATION
%%%%%%%%%%%%%%%%%%%%%%%%%%%%%%%%%%%%%%%%%%%%%%%%%%%%%%%%%%%%%%%%%%%%%%%%%%

\section{Solving the Generator equation}\label{sec:generatorequation}

The aim of this section is to show existence and uniqueness of solution for the generator equation connected to the SDE~\eqref{e:SDE}, i.e. the PDE
\begin{equation}\label{eq:generator}
\partial_tu+\frac{1}{2}\Delta u+V\cdot\nabla u = f\,,\qquad u(T,\cdot)=u^T(\cdot)
\end{equation}
where $T>0$ is arbitrary but fixed, $u^T$ is the terminal condition and $f\in C_T\CC^\beta$, for $\beta\in(-\frac{2}{3},0]$.
Let $(t,x)\in[0,T)\times\R$ and, for a function $\psi$, let $\CJ^T(\psi)$ be defined by $\CJ^T(\psi)(t)=\int_t^T\int_{\R^d}P_{r-t}\psi( r) \dd r$, where $P_t\eqdef e^{\frac{1}{2}t\Delta}$ is the usual heat flow.
Using the previous notation, the mild formulation of our generator equation reads
\begin{equation}\label{eq:generatormild}
u(t)= P_{T-t}u^T+\CJ^T\big(f+\nabla u\cdot V\big)(t).
\end{equation}
Now, since $V\in C_T\CC^\beta_{\R^d}$, Schauder's estimates (Proposition~\ref{prop:Schauder}) suggest that the solution $u$ to the previous equation cannot have spatial regularity better than $\beta+2$. 
According to Proposition~\ref{prop:bony}, the product between $\nabla u$ and $V$ is well-posed if and only if the sum of the regularities of the factors is strictly positive, which, in the present case, reads $\beta+1+\beta=2\beta+1>0$, i.e. $\beta>-\frac{1}{2}$. 
Therefore, for $\beta\in(-\frac{1}{2},0)$, we can directly apply Bony's and Schauder's estimates and construct the solution to the equation directly. Even if spaces, notation and tools might look different, this case can be easily shown to correspond to the one treated in~\cite{FFF} (see Remark~\ref{rem:FIR}).
On the other hand, to overcome the $-\frac{1}{2}$ barrier, another method has to be exploited and paracontrolled distributions must be introduced. 

%%%% young %%%%%

\subsection{The Young case: $\beta\in(-\frac{1}{2},0)$}

In order to construct the solution of the generator equation we will use a fixed point argument, i.e. we will introduce a suitable map and prove it is a contraction on a suitable space, hence admitting a unique fixed point according to Banach Fixed Point theorem.
To do so, let us fix a terminal time $T>0$, $\alpha\in(1-\beta,\beta+2)$, a terminal condition $u^T\in\CC^{\beta+2}$ and $f\in C_T\CC^\beta$. Given a function $u$ in $C_{\bar T, T}\CC^\alpha$, for $\bar T\in[0,T)$, we define the map $\Gamma_{\bar T}(u)$ as
\begin{equation}\label{map:Gamma}
\Gamma_{\bar T}(u)(t)\eqdef P_{T-t}u^T+\CJ^T\big(f+\nabla u\cdot V\big)(t)
\end{equation}
where $\CJ^T$ is the operator defined above and we omitted the dependence on space. Notice that
\begin{align*}
\|\Gamma_{\bar T}(u)(t)\|_\alpha &\leq \|P_{T-t}u^T\|_\alpha+\|\CJ^T\big(f+\nabla u\cdot V\big)(t)\|_\alpha\\
&\lesssim \|u^T\|_{\beta+2} + \bar T^{\frac{\beta-\alpha}{2}+1}\|f\|_{C_T\CC^\beta} 
+\bar T^{\frac{\beta-\alpha}{2}+1}\|\nabla u\cdot V\|_{C_T\CC^\beta}\\
&\lesssim \|u^T\|_{\beta+2} + \bar T^{\frac{\beta-\alpha}{2}+1}\big(\|f\|_{C_T\CC^\beta}+\|V\|_{C_T\CC^\beta_{\R^d}}\|u\|_{C_{\bar T,T}\CC^\alpha} \big)
\end{align*}
where the second inequality is a simple application of Corollary~\ref{cor} and the last follows by Bony's estimates Proposition~\ref{prop:bony}. Therefore, setting $\gamma=\frac{\beta-\alpha}{2}+1$ we have 
\[
\|\Gamma_{\bar T}(u)\|_{C_{\bar T,T}\CC^\alpha}\lesssim \|u^T\|_{\CC^{\beta+2}}+\bar T^\gamma\Big(\|f\|_{C_T\CC^\beta}+\|V\|_{C_T\CC^\beta_{\R^d}}\|u\|_{C_{\bar T, T}\CC^\alpha}\Big)\,.
\]
The next proposition summarizes what we have obtained so far and shows how to build a local in time solution to~\eqref{eq:generatormild} for $V\in C([0,T],\CC^\beta(\R^d,\R^d))$, $\beta\in\big(-\frac{1}{2},0\big)$. 

\begin{proposition}\label{prop:fixedpointyoung}
Let $T>0$, $\beta\in\big(-\frac{1}{2},0\big)$ and $\alpha\in(1-\beta,\beta+2)$. For $(u^T,f, V)\in \CC^\alpha\times C_T\CC^\beta\times C_T\CC^\beta_{\R^d}$, let $\Gamma_{\bar T}$ be the map on $C_{\bar T,T}\CC^\alpha$ defined by~\eqref{map:Gamma}.
Then there exists $\gamma>0$ such that the following bounds hold true
\begin{equation}\label{ineq:young1}
\|\Gamma_{\bar T}(u)\|_{C_{\bar T,T}\CC^\alpha}\lesssim \|u^T\|_{\CC^{\beta+2}}+\bar T^\gamma\Big(\|f\|_{C_T\CC^\beta}+\|V\|_{C_T\CC^\beta_{\R^d}}\|u\|_{C_{\bar T,T}\CC^\alpha}\Big) 
\end{equation}
and
\begin{equation}\label{ineq:young2}
\|\Gamma_{\bar T}(u)-\Gamma_{\bar T}(v)\|_{C_{\bar T,T}\CC^\alpha}\lesssim  \bar T^\gamma\|V\|_{C_{T}\CC^\beta_{\R^d}}
 \|u-v\|_{C_{\bar T,T}\CC^\alpha}
\end{equation}
Hence, there exists $T^\star\in[0,T)$ depending only on $\|V\|_{C_T\CC^\beta(\R^d)}$, and a unique function $u\in C([T-T_\star,T],\CC^\alpha)$ that solves the generator equation~\eqref{eq:generatormild}.  
\end{proposition}
\begin{proof}
The bound~\eqref{ineq:young1} is proved above and an analogous argument shows that~\eqref{ineq:young2} holds true as well. Therefore there exists $T^\star\in(0,T)$ sufficiently close to $T$ and depending only on $\|V\|_{C_T\CC^\beta_{\R^d}}$ such that the map $\Gamma_{T^\star}$ is a strict contraction of $C([T-T_\star,T],\CC^\alpha_{\R^d})$ in itself and, by Banach fixed point theorem, it admits a unique fixed point.% and this concludes the proof. 
\end{proof}

We have now all the elements in place to state and prove the following theorem. 

\begin{theorem}\label{th:Generator-Young}
Let $\beta\in (-\frac{1}{2},0)$, $\alpha\in(1-\beta,\beta+2)$ and $T>0$. For any $(u^T,f, V)\in \CC^\alpha\times C_T\CC^\beta\times C_T\CC^\beta_{\R^d}$, there exists a unique solution $u\in C_T\CC^\alpha$ to the generator equation~\eqref{eq:generator}, where the product $\nabla u\cdot V$ is defined according to Proposition~\ref{prop:bony}. Moreover, the solution $u$ satisfies 
\[
\|u\|_{C_T^\eps\CC^\rho}\lesssim\|u^T\|_{\alpha}+\|f\|_{C_T\CC^\beta}+\|u\|_{C_T\CC^\alpha}\|V\|_{C_T\CC^\beta_{\R^d}}
\]
for every $\rho$ and $\eps$ such that $\rho+2\eps\leq\alpha$. 
At last, the flow of the generator equation, i.e. the map assigning to every triplet $(u^T,f,V)\in \CC^\alpha\times C_T\CC^\beta\times C_T\CC^\beta_{\R^d}$ the solution $u$ to~\eqref{eq:generator}, is a locally Lipschitz continuous map.
\end{theorem}

\begin{proof}
Thanks to Proposition~\ref{prop:fixedpointyoung}, we already know that there exists $T^\star\in[0,T)$ and a unique function $u\in C([T-T_\star, T],\CC^\alpha)$ that solves the generator equation~\eqref{eq:generatormild}.  
Now, since the equation is linear and consequently the $T^\star$ determined above depends only on $V$ and not on $u^T$, we can extend our solution to the whole interval $[0,T]$, iterating the construction we just carried out, so that the resulting $u$ is defined on the whole interval $[0,T]$. 

The time regularity of the solution can be easily obtained by an interpolation argument. Finally, taking $V,\,\tilde{V}\in C_T\CC^\beta_{\R^d}$, $f,\,\tilde{f}\in C_T\CC^\beta$, 
$u^T,\,\tilde{u}^T\in C_T\CC^{\beta+2}$ and denoting by $u^V$ (resp. $u^{\tilde{V}}$) the solution of the equation $\mathscr G^Vu=f$ (resp. $\mathscr G^{\tilde{V}}u=\tilde{f}$) with terminal condition $u^T$ (resp. $\tilde{u}^T$), it is easy to show that, if
$$
\max\{\|u^T\|,\|\tilde{u}^T\|,\|f\|,\|\tilde{f}\|,\|V\|,\|\tilde{V}\|\}\leq R
$$
then 
$$
\|u^V-u^{\tilde{V}}\|_{C_T\CC^\alpha}\lesssim_R \|u^T-\tilde{u}^T\|+\|f-\tilde{f}\|+\|V-\tilde{V}\|
$$
which proves that the flow is indeed a locally Lipschitz map (for more details, see for example the proof of an analogous result in~\cite{Gub04}).
\end{proof}

\begin{remark}\label{rem:FIR}
As we pointed out before, the analysis performed in this section corresponds to the case treated in~\cite{FFF} with the only difference that we preferred to work with H\"older spaces of negative regularity instead of Sobolev spaces. There is no doubt that we could have used the latter spaces as well since Bony's and Schauder's estimates (Proposition~\ref{prop:bony} and~\ref{prop:Schauder}) hold also for these spaces (see~\cite[Chapter 2]{BCD11}). 
\end{remark}

%%% ROUGH %%%

\subsection{The rough case: $\beta\in\big(-\frac{2}{3},-\frac{1}{2}\big]$}\label{subsec:geneqrough}
\label{sec:solv-Generator}

The analysis of the rough case is more subtle and requires a better understanding of the structure of the solution to the generator equation. 
Let us assume for the moment that $V$ is a smooth function. Thanks to Bony's decomposition of the product we can write~\eqref{eq:generatormild} as 
\begin{equation}\label{eq:Paracontrolled-distribution}
u(t)=\CJ^T(f+\nabla u\prec V)+u^{\flat}(t)
\end{equation}
where
$$
u^{\flat}(t)\eqdef P_{T-t}u^T+\CJ^T(\nabla u\succ V+\nabla u\circ V)
$$
What we see at this point is that when $V$ is a distribution in $C([0,T],\CC^{\beta}(\R^d,\R^d))$ the only ill-defined term of the equation~\eqref{eq:Paracontrolled-distribution} is the resonant term contained in $u^\flat$. Nevertheless, Proposition~\eqref{prop:bony} suggests that, {\it if} it were well-posed, $u^\flat(t)\in\CC^{2\theta-1}$ for $\theta<\beta+2$. 

As we announced before, we need some insight regarding the expected structure of the solution. Indeed, even if it is not possible to make sense of the ill-posed product for all distributions belonging to spaces whose regularities do not sum up to a strictly positive quantity, maybe it is possible to identify a suitable subspace for which it is. 
To recognize such a subspace we begin with the following lemma, which allows to commute the heat kernel $\CJ^T$ and the paraproduct $\prec$. 

\begin{lemma}\label{lemma:comm-sc}
Let $T>0$, $\theta\in[1,\beta+2)$, $\rho>\frac{\theta-1}{2}$ and $h\in C_T\CC_{\R^d}^{\beta}$. For $\bar T\in[0,T)$, let $g\in C_{\bar T,T}\CC^\theta$ be such that 
$\nabla g\in C^\rho_{\bar T,T}L^\infty_{\R^d}$. Then the following inequality holds
\small\begin{equation*}
\|\CJ^T(\nabla g\prec h)-\nabla g\prec\CJ^T(h)\|_{C_{\bar T,T}\CC^{2\theta-1}}\lesssim \bar T^\kappa\big(\|g\|_{C_{\bar T,T}\CC^\theta}+\|\nabla g\|_{ C^\rho_{\bar T,T}L^\infty_{\R^d}}\big)\|h\|_{C_T\CC^\beta_{\R^d}}
\end{equation*}\normalsize
with $\kappa\eqdef\min\big\{1-\frac{\theta-\beta}{2},\rho-\frac{\theta-1}{2}\big\}>0$. 
\end{lemma}
\begin{proof}
By direct computation, for $t\in[T-\bar T, T]$, we can express the right hand side of the inequality as the sum of two terms $I_1$ and $I_2$, respectively given by
\small\begin{align*}
I_1(t)&=\int_t^T\big( P_{r-t}(\nabla g( r)\prec h( r))-\nabla g( r)\prec P_{r-t}h( r)\big)\dd r\,,\\
I_2(t)&=\int_t^T(\nabla g( r)-\nabla g(t))\prec  P_{r-t}h(r)\,\dd r .
\end{align*}\normalsize
Using the commutation result in~\eqref{commSchauder} we directly get
\small$$
\|I_1(t)\|_{2\theta-1}\lesssim\int_t^T(r-t)^{-\frac{\theta-\beta}{2}}\|g( r)\|_{\theta}\|h( r)\|_{\beta}\dd r\lesssim \bar T^{1-(\theta-\beta)/2}\|g\|_{C_{\bar T,T}\CC^\theta}\|h\|_{C_T \CC^\beta_{\R^d}}
$$\normalsize
For $I_2$ we apply Schauder's estimates and obtain
\small\begin{align*}
\|I_2(t)\|_{2\theta-1}&\lesssim\int_t^T\|\nabla g( r)-\nabla g(t)\|_{L^\infty_{\R^d}}(r-t)^{-(\theta+1)/2}\dd r\|h\|_{C_T\CC^{\theta-2}(\R^d)}\\
&\lesssim\int_t^T(r-t)^{\rho-(\theta+1)/2}\dd r\|\nabla g\|_{C^\rho_{\bar T,T}L^\infty_{\R^d}}\|h\|_{C_T\CC^{\beta,\R^d}}\\
&\lesssim \bar T^{1-(-\rho+\frac{\theta+1}{2})}\|\nabla g\|_{C^\rho_{\bar T,T}L^\infty_{\R^d}}\|h\|_{C_T\CC^\beta_{\R^d}}
\end{align*}\normalsize
and this ends the proof.
\end{proof}

The previous Lemma suggests that, at least at a formal level, the solution $u$ of our equation admits the following expansion
\begin{equation}\label{eq:paraansatz}
u=\CJ^T(f)+\nabla u\prec\CJ^T(V)+u^\sharp\qquad 
\end{equation}
where
$$
u^\sharp=u^\flat+\CJ^T(\nabla u\prec V)-\nabla u\prec\CJ^T(V)
$$
should be more regular than $u$ itself. On the one hand, equation~\eqref{eq:paraansatz} conveys the algebraic structure we expect the solution of~\eqref{eq:generatormild} to have and on the other, it tells us that $u$, in terms of regularity exhibits the same behaviour as $\CJ^T(V)$. This is exactly the core idea of the paracontrolled approach developed in~\cite{GIP15} and it will allow us to conveniently define the ill-posed term.

We are now ready to introduce the space of paracontrolled distributions associated to the equation~\eqref{eq:generator}.

\begin{definition}\label{def:paracontrolled}
Let $T>0$, $\frac{4}{3}<\alpha<\theta<\beta+2$ and $\rho>\frac{\theta-1}{2}$. For $f\in C_T\CC^\beta$ and $\bar T\in[0,T)$, we define the space of paracontrolled distributions $ \DD^{\alpha,\theta,\rho}_{\bar T,T,V}$ as the set of couples of distributions $(u,u')\in C_{\bar T,T}\CC^\theta\times C_{\bar T,T}\CC^{\alpha-1}_{\R^d}$ such that 
$$
u^{\sharp}(t)\eqdef u(t)-u'(t)\prec\CJ^T(V)(t)-\CJ^T(f)(t)\in \CC^{2\alpha-1}
$$
for all $T-\bar T\leq t\leq T$. We equip $\DD^{\alpha,\theta,\rho}_{\bar T, T,V}$ with the norm
\small \[
\|(u,u')\|_{ \mathscr D^{\alpha,\theta,\rho}_{\bar T,T,V}}\eqdef\|u\|_{C_{\bar T,T}\CC^\theta} +\|\nabla u\|_{C^\rho_{\bar T,T}L^{\infty}_{\R^d}}+\|u'\|_{C_{\bar T,T}\CC^{\alpha-1}_{\R^d}}+\|u^\sharp\|_{C_{\alpha-1,\bar T, T}\CC^{2\alpha-1}}%\sup_{t\in[T-\bar T,T]}(T-t)^{\frac{\alpha-1}{2}}\|u^\sharp(t)\|_{2\alpha-1}
\]\normalsize
and we introduce the metric $d_{ \DD^{\alpha,\theta,\rho}_{\bar T,T,V}}$, defined for all $(u,u'),(v,v')\in\DD^{\alpha,\theta,\rho}_{\bar T,T,V}$ by
\[
d_{ \DD^{\alpha,\theta,\rho}_{\bar T,T,V}}\big((u,u'),(v,v')\big)=\|(u,u')-(v,v')\|_{ \DD^{\alpha,\theta,\rho}_{\bar T,T,V}}
\]
Endowed with the metric $d_{\DD^{\alpha,\theta,\rho}_{\bar T,T,V}}$, the space $\big( \DD^{\alpha,\theta,\rho}_{\bar T,T,V},d_{ \DD^{\alpha,\theta,\rho}_{\bar T,T,V}}\big)$ is a complete metric space.
\end{definition}

The advantage of the paracontrolled formulation is that the problem of well-posedness for the product can be transferred from the function $u$, that we have to determine and is therefore unknown, to $V$, or better $\CJ^T(V)$, which on the other hand is given. 
To see how this works, take $(u,u')\in \DD^{\alpha,\theta,\rho}_{\bar T, T,V}$. Differentiating $u$, for $j=1,...,d$, we get
\small\begin{equation}\label{eq:Usharp}
\partial_ju=\CJ^T(\partial_j f)+\sum_{i=1}^{d}u'^{,\,i}\prec \CJ^T(\partial_j V^i)+U^{\sharp,\,j},\,\, U^{\sharp,\,j}=\partial_j u^\sharp+\sum_{i=1}^{d}\partial_ju'^{,\,i}\prec \CJ^T(V^{i})
\end{equation}\normalsize
so that the resonant term, for $V$ smooth, can be written as
\[
\partial_ju\circ V^j=\CJ^T(\partial_jf)\circ V^j+\sum_{i=1}^{d}\big(u'^{,\,i}\prec \CJ^T(\partial_jV^{i})\big)\circ V^j+U^{\sharp,\,j}\circ V^j
\]
By Bony's paraproduct estimate we immediately deduce that $U^{\sharp,\,j}$ is $(2\alpha-2)$-regular in space and, since $\alpha>\frac{4}{3}$, we conclude that the last summand is well-defined even when $V(t)\in \CC^{\beta}_{\R^d}$. In order to make sense of the second summand we need to exploit the commutator in Proposition~\ref{prop:comm} which gives
\small\[
\sum_{i=1}^{d}\big(u'^{,\,i}\prec \CJ^T(\partial_jV^{i})\big)\circ V^j=\sum_{i=1}^{d}u'^{,\,i}\big(\CJ^T(\partial_j V^{i})\circ V^{i}\big)+\sum_{i=1}^d\mathscr R(u'^{,\,i},\CJ^T(\partial_j V^i),V^{j})
\]\normalsize
where the last summand of the previous can be extended in a continuous way to $V\in C_T\CC^{\beta}_{\R^d}$ since $3\alpha-4>0$. 
The only terms which are still ill-posed are $\CJ^T(\partial_j V^{i})\circ V^{i}$ for $i,\,j=1,...,d$. Notice though that they do not depend on $u$ anymore but only on $V$, so if we can build them {\it in some way}, we are done and we can make sense of the product.  This is the reason why we introduce the notion of rough distribution. 

\begin{definition}[Rough Distribution]\label{def:rough}
Let $\beta\in\left(-\frac{2}{3},-\frac{1}{2}\right]$, $\gamma<\beta+2$ and $T>0$. Set $\ch^\gamma=C_T\CC^{\gamma-2}_{\R^d}\times C_T\CC^{2\gamma-3}_{\R^{d^2}}$. We define the space of {\it rough distributions} as  
\[
\cX^\gamma\eqdef\text{cl}_{\ch^\gamma}\big\{ \CK(\eta)\eqdef\left(\eta,(\CJ^T(\partial_j\eta^i)\circ\eta^j)_{i,j=1,\dots,d}\right),\quad \eta\in C_T\CC^{\infty}_{\R^d}\big\}
\]
where $\text{cl}_{\ch^\gamma}\{\cdot\}$ denotes the closure of the set in brackets with respect to the topology of $\ch^\gamma$ and, for a function $\psi:\R^d\to\R$,  $\CJ^T(\psi)$ is the solution of the equation 
\[
\big(\partial_t+\frac{1}{2}\Delta\big)\CJ^T( \psi)=\psi,\quad \CJ^T( \psi)(T,\cdot)=0.
\]
We denote by $\CV=(\CV^1,\CV^2)$ a generic element of $\cX^\gamma$ and whenever $\CV^1=V$ we say that $\CV$ is a lift (or enhancement) of $V$.
\end{definition} 
\begin{remark}
The reader familiar with rough path theory can appreciate the similarity of the space introduced above with the space of rough paths associated to a given path.
 
Let us point out that, as in the above-mentioned situation, there is in general no canonical choice for the extra-term $\mathscr J_T(\partial_j\eta^i)\circ\eta^{j}$ when $\eta$ has space regularity $\gamma-2$. 
However there are several simple cases, as the ones in~\cite[Section 5]{FR} for $d=1$, in which this construction can be successfully carried out. To witness, let us consider the time-independent one dimensional situation with $V=\partial_x Y$ and $Y\in\CC^{\beta+1}$. Then, for $x\in\R$
\[
\CJ^T( \partial_x^2 Y)(t,x)=\int_\R \int_t^T\partial_t P_{r-t}(x-y) \dd r Y(y)\dd y=P_{T-t} Y(x)
\] 
where the first equality follows by the fact that $P_t$ is the fundamental solution to the heat equation. Then, $P_{T-t} Y\in \CC^{\beta+3}$ and $P_{T-t} Y\circ \partial_x Y$ is well-posed if and only if $\beta>-\frac{3}{2}$ which is well-beyond the $-\frac{2}{3}$ barrier. 

\end{remark}

In force of the previous definition and thanks to the computations above, $\nabla u\circ V$ can be decomposed as
\begin{align*}
\nabla u\circ V=&\sum_{j=1}^{d}\CJ^T(\partial_j f)\circ V^j+\sum_{i,j=1}^{d}u'^{,\,i}(\CJ^T(\partial_j V^i)\circ V^j)\\
&+\sum_{i,j=1}^d\mathscr R(u'^{,\,i},\CJ^T(\partial_j V^{i}),V^j)+\sum_{j=1}^dU^{\sharp,\,j}\circ V^j
\end{align*}
which suggests that the left hand side of the previous should be a continuous functional of $(u,u')\in\DD^{\alpha,\theta,\rho}_{\bar T, T,V}$ and $\CV \in \cX^\gamma$. This is exactly what the next proposition proves.

\begin{proposition}\label{prop:reson}
Let $T>0$ and $\frac{4}{3}<\alpha<\theta<\gamma<\beta+2$. Let $\CV=( \CV_1, \CV_2)\in\cX^\gamma$ be an enhancement of $V$, $f$ be either a function in $C_TL^\infty$ or coincide with one of the components of $\CV_1$, i.e. $f\in \{ \CV_1^i,\,i=1,. . .,d\}$, and, for $\bar T\in[0,T)$, $(u,u')\in  \DD^{\alpha,\theta,\rho}_{\bar T, T,V}$. Define $\nabla u\circ V$ by
\begin{align}
\nabla u\circ V \eqdef\sum_{j=1}^{d}H^j(f, \CV)&+\sum_{i,j=1}^{d}u'^{,\,i} \CV_2^{i,j} \notag\\
&+\sum_{i,j=1}^d\mathscr R(u'^{,\,i},\CJ^T(\partial_j  \CV_1^{i}), \CV_1^j)+\sum_{j=1}^dU^{\sharp,j}\circ  \CV_1^j, \label{eq:algebraic}
\end{align}
where, in case $f=\CV_1^j$, $H^j(f, \CV)\eqdef \CV_2^{j,j}$, while if $f\in C_TL^\infty$, $H^j(f, \CV)\eqdef \CJ^T(\partial_j f)\circ \CV_1^j$, and $U^\sharp$ is given by the expression in~\eqref{eq:Usharp}. Then $\nabla u\circ V$ is well-defined and the following estimate holds
\begin{align*}
\|\big(\nabla u\circ V\big)\|_{C_{\alpha-1,\bar T,T}\CC^{2\gamma-3}}\lesssim\,\,&\1_F\|f\|_{C_TL^\infty}\| \CV_1\|_{C_T\CC^{\gamma-2}_{\R^d}}\\
&+\Big(1+\| \CV\|_{ \cX^\gamma}\Big)^2\big(1+\|(u,u')\|_{\DD^{\alpha,\theta,\rho}_{\bar T,T,V}}\big)
\end{align*}
where $F\eqdef \{f\in C_TL^\infty\}$. At last, under the previous assumptions, the product $\nabla u\cdot V$, defined according to Bony's decomposition and equation~\eqref{eq:algebraic}, is well-defined.
\end{proposition}
\begin{proof}
In order to prove the bound in the statement, one has to consider each of the summands in~\eqref{eq:algebraic} separately. For the first three, it is an immediate consequence of the assumption $\CV\in\cX^\gamma$, Bony's estimates (Proposition~\ref{prop:bony}) and the commutator lemma (Proposition~\ref{prop:comm}) respectively. 
For the last, notice that, by the definition of $U^\sharp$ given in~\eqref{eq:Usharp}, the fact that $(u,u')\in\DD^{\alpha,\theta,\rho}_{\bar T, T,V}$ and, again, Bony's paraproduct estimates, we have
\[
\|U^{\sharp,\,j}\|_{2\alpha-2}\lesssim(T-t)^{-\frac{\alpha-1}{2}}\|(u,u')\|_{\DD^{\alpha,\theta,\rho}_{\bar T, T,V}}+\bar T^{\frac{\gamma-\alpha}{2}}\|(u,u')\|_{\DD^{\alpha,\theta,\rho}_{\bar T, T,V}}\|V\|_{C_T\CC^\beta_{\R^d}}
\]
which immediately gives, for $\alpha>1$
\[
\sup_{t\in[T-\bar T,T]}(T-t)^{\frac{\alpha-1}{2}}\|U^{\sharp,\,j}(t)\|_{2\alpha-2}\lesssim \|(u,u')\|_{\DD^{\alpha,\theta,\rho}_{\bar T,T,V}}\big(1+\bar T^{\frac{\gamma-1}{2}}\|V\|_{C_T\CC^\beta_{\R^d}}\Big)
\]
Hence, by applying Bony's estimate for $\circ$ we get the expected bound on the fourth summand as well.
The last part of the statement is once more a consequence of Proposition~\ref{prop:bony}. 
\end{proof}

At this point we have all we need in order to set up our fixed point argument. Indeed, let $\CV=(\CV_1,\CV_2)\in \cX^\theta$ be an enhancement of $V$, i.e. $V=\CV_1$, and set $M_{\bar T}$ to be the map from $\DD^{\alpha,\theta,\rho}_{\bar T,T,V}$ to $C_T\CC^\alpha$ given by 
\begin{equation}
\label{eq:fixed-point-sol}
M_{\bar T}(u,u')=\CJ^T(f)+\CJ^T(\nabla u\cdot V)+\Psi^T_t
\end{equation}
for $(u,u')\in \DD^{\alpha,\theta,\rho}_{\bar T, T,V}$, $\alpha<\theta$ and $\Psi_t^T= P_{T-t}u^T$, where the term $\nabla u \cdot V$ is defined according to Proposition~\ref{prop:reson}. Set
\begin{equation}
\label{eq:fixed-point}
\begin{split}
\mathscr M_{\bar T}:&\mathscr D^{\alpha,\theta,\rho}_{\bar T,T,V}\to C([T-\bar T,T],\mathscr C^\alpha(\mathbb R^d))\times C([T-\bar T,T],\mathscr C^{\alpha-1}(\mathbb R^d,\mathbb R^d))
\\&(u,u')\mapsto(M(u,u'),\nabla u) 
\end{split}
\end{equation}
We can now prove that this map is a contraction in the space $\mathscr D_{\bar T,T,V}^{\alpha,\theta,\rho}$ and therefore it admits a unique fixed point. 

\begin{proposition}\label{prop:fixedpointrough}
Let $0<T<1$, $\frac{4}{3}<\alpha<\theta<\gamma<\beta+2$, $\rho\in(\frac{\theta-1}{2},\frac{\gamma-1}{2})$. Let $u^T\in\CC^\gamma$, $V\in C_T\CC^\beta_{\R^d}$,  $\CV=(\CV_1,\CV_2)\in\mathscr X^\gamma$ be an enhancement of $V$ and $f$ be either a function in $C_TL^\infty$ or coincide with one of the component of $\CV_1$, i.e. $f\in \{ \CV_1^i,\,i=1,. . .,d\}$. Then, for $\bar T\in[0,T)$, there exists $\kappa>0$, depending only on $\alpha,\,\theta,\,\rho$ and $\gamma$, such that the map $\mathscr M_{\bar T}$ defined by~\eqref{eq:fixed-point} satisfies the following estimates
\begin{align}
\|\mathscr M_{\bar T}(u,u')\|_{\DD_{\bar T,T,V}^{\alpha,\theta,\rho}}
\lesssim &\1_{F}\|f\|_{C_TL^\infty_{\R^d}}\|V\|_{C_T\CC^\beta_{\R^d}}+\|u^T\|_{\gamma}\notag\\
&+\big(1+\|\CV\|_{\,\!_{ \cX^\gamma}}\big)^2\big(1+\bar T^\kappa\|(u,u')\|_{\mathscr D^{\alpha,\theta,\rho}_{\bar T,T,V}}\big)\label{eq:estim-map-1}
\end{align}
where $F\eqdef \{f\in C_TL^\infty\}$ and
\small\begin{equation}\label{eq:estim-map-2}
\|\mathscr M_{\bar T}(u,u')-\mathscr M_{\bar T}(v,v')\|_{\DD^{\alpha,\theta,\rho}_{\bar T,T,V}}\lesssim\big(1+\|\CV\|_{\cX^\gamma}\big)^2)\big(1+\bar T^\kappa\|(u,u')-(v,v')\|_{\DD^{\alpha,\theta,\rho}_{\bar T,T,V}}\big)
\end{equation}\normalsize
and is therefore a strict contraction in $\DD_{\bar T, T,V}^{\alpha,\theta\rho}$ for $T-\bar T$ small enough. 
\end{proposition}

\begin{proof}
Let $(u,u')\in\DD_{\bar T,T,V}^{\alpha,\theta,\rho}$. In order to prove that $\mathscr M_{\bar T}(u,u')=(M_{\bar T}(u,u'),\nabla u)\in\mathscr D_{\bar T,T,V}^{\alpha,\theta,\rho}$ it suffices to estimate the terms 
$$
M_{\bar T}(u,u')=\Psi^T_t+\CJ^T(f+\nabla u\circ V), \qquad\qquad M_{\bar T}(u,u')'\eqdef\nabla u
$$
and 
$$
M_{\bar T}(u,u')^\sharp\eqdef M_{\bar T}(u,u')-\CJ^T(f)-\nabla u\prec \CJ^T(V)
$$
in suitable norms. More precisely we have to control the following quantity 
\small\begin{equation*}
\begin{split}
\|\mathscr M_{\bar T}(u,u')\|_{\DD^{\alpha,\theta,\rho}_{\bar T,T,V}}&\eqdef\|M_{\bar T}(u,u')\|_{C_{\bar T,T}\CC^\theta}+\|\nabla M_{\bar T}(u,u')\|_{C^\rho_{\bar T,T}L^{\infty}_{\R^d}}\\
&+\|M_{\bar T}(u,u')'\|_{C_{\bar T,T}\CC^{\alpha-1}_{\R^d}}+\|M_{\bar T}(u,u')^\sharp\|_{C_{\alpha-1,\bar T, T}\CC^{2\alpha-1}}%\sup_{t\in[0,T]}(T-t)^{\frac{\alpha-1}{2}}\|M(u,u')^\sharp(t)\|_{2\alpha-1}
\end{split}
\end{equation*}\normalsize
Let us begin with first. 
According to the definition of $M_{\bar T}(u,u')$ we have to estimate the $C_{\bar T,T}\CC^\theta$-norm of
\begin{equation}\label{terms}
\Psi^T,\,\CJ^T(f),\,\CJ^T(\nabla u\prec V),\,\CJ^T(\nabla u\succ V),\,\CJ^T(\nabla u\circ V).
\end{equation}
Since the heat-flow $P_t$ is a bounded linear operator from $\CC^\theta$ to itself %with operator norm equal to $\sqrt{2\pi}$ 
we get immediately that 
$$
\sup_{t\leq T}\|\Psi_t^T\|_{\theta}\lesssim\|u^T\|_{\theta}\lesssim\|u^T\|_{\gamma}
$$
By Corollary~\ref{cor} we have 
\[
\|\CJ^T(f)(t)\|_{\theta}\lesssim
\begin{cases}
\bar T^{\frac{\gamma-\theta}{2}}\|f\|_{C_TL^\infty} &\text{if $f\in C_TL^\infty$}\\
\bar T^{\frac{\gamma-\theta}{2}}\|f\|_{C_T\CC^{\gamma-2}} &\text{if $f\in\{V^k;k=1,. . .,d\}$}
\end{cases}
\] 
Let us focus on $\CJ^T(\nabla u\prec V)$ and $\CJ^T(\nabla u\prec V)$. Applying once more Corollary~\ref{cor} and Bony's estimates we obtain
\small\begin{align*}\label{eq:bound-2-para-1}
\|\CJ^T(\nabla u\prec V)(t)\|_\theta&\lesssim \bar T^{\frac{\gamma-\theta}{2}}\|\nabla u\|_{C_{\bar T,T}\CC^{\theta-1}_{\R^d}}\|V\|_{C_T\CC^{\gamma-2}_{\R^d}} \lesssim \bar T^{\frac{\gamma-\theta}{2}}\|(u,u')\|_{\DD_{\bar T,T,V}^{\alpha,\theta,\rho}}\|\CV\|_{\cX^\gamma}\\
\|\CJ^T(\nabla u\succ V)(t)\|_\theta&\lesssim \bar T^{\frac{\gamma-1}{2}}\|\nabla u\|_{C_{\bar T,T}\CC^{\theta-1}_{\R^d}}\|V\|_{C_T\CC^{\gamma-2}_{\R^d}} \lesssim \bar T^{\frac{\gamma-1}{2}}\|(u,u')\|_{\DD_{\bar T,T,V}^{\alpha,\theta,\rho}}\|\CV\|_{\cX^\gamma}
\end{align*}\normalsize
We will now treat the resonant term $\CJ^T(\nabla u\circ V)$. By the first part of Corollary~\ref{cor} and Proposition~\ref{prop:reson} we directly see that its $\CC^\theta$-norm is bounded by
\small\begin{multline*}
\bar T^{\frac{2\gamma-\alpha-\theta}{2}}\sup_{t\in[T-\bar T,T]}(T-t)^{\frac{\alpha-1}{2}}\|\nabla u\circ V(t)\|_{2\gamma-3}\\
\lesssim \bar T^{\frac{2\gamma-\alpha-\theta}{2}}\Big(\1_F\|f\|_{C_TL^\infty}\| \CV_1\|_{C_T\CC^{\gamma-2}_{\R^d}}+\Big(1+\| \CV\|_{ \cX^\gamma}\Big)^2\big(1+\|(u,u')\|_{\DD^{\alpha,\theta,\rho}_{\bar T,T,V}}\big)\Big)
\end{multline*}\normalsize
where $F\eqdef \{f\in C_TL^\infty\}$ and this completes the study of the first term. 

Consider now $\|\nabla M_{\bar T}(u,u')\|_{C^\rho_TL^{\infty}_{\R^d}}$. In this case we have to bound the derivative of the terms in~\eqref{terms} in the $C^\rho_{\bar T,T}L^{\infty}_{\R^d}$-norm.
Thanks to Proposition~\ref{prop:Schauder} we see that
\small$$
\|\nabla\Psi^T_t-\nabla \Psi_s^T\|_{\infty}=\|(P_{t-s}-1)P_s\nabla u^T\|_{\infty}\lesssim |t-s|^{\rho}\|\nabla u^T\|_{2\rho}\lesssim|t-s|^{\rho}\|u^T\|_{\gamma}
$$\normalsize
The second part of Corollary~\ref{cor} guarantees that, for $0\leq s<t\leq T$,
\small\[
\frac{\|\CJ^T(\nabla f)(t)-\CJ^T(\nabla f)(s)\|_{\infty}}{|t-s|^\rho}\lesssim
\begin{cases}
\bar T^{\frac{\gamma-1}{2}-\rho}\|f\|_{C_TL^\infty} &\text{if $f\in C_TL^\infty$}\\
\bar T^{\frac{\gamma-1}{2}-\rho}\|f\|_{C_T\CC^{\gamma-2}} &\text{if $f\in\{V^k;k=1,. . .,d\}$}
\end{cases}
\] \normalsize
Analogously, by Bony's estimates we get
\small\[
\frac{\|\CJ^T\big(\nabla(\nabla u\prec V) \big)(t)-\CJ^T\big(\nabla(\nabla u\prec V) \big)(s)\|_{\infty}}{|t-s|^\rho}\lesssim \bar T^{\frac{\gamma-1}{2}-\rho}\|\nabla u\|_{C_{\bar T,T}\CC^{\theta-1}_{\R^d}}\|V\|_{C_T\CC^{\gamma-2}_{\R^d}}
\]\normalsize
and
\small\[
\frac{\|\CJ^T\big(\nabla(\nabla u\succ V) \big)(t)-\CJ^T\big(\nabla(\nabla u\succ V) \big)(s)\|_\infty}{|t-s|^\rho}\lesssim \bar T^{\frac{\theta+\gamma-2}{2}-\rho}\|\nabla u\|_{C_{\bar T,T}\CC^{\theta-1}_{\R^d}}\|V\|_{C_T\CC^{\gamma-2}_{\R^d}}
\]\normalsize
which imply the correct bound. At last, by Corollary~\ref{cor} we have
\small\begin{multline*}
\frac{\|\CJ^T\big(\nabla(\nabla u\circ V) \big)(t)-\CJ^T\big(\nabla(\nabla u\circ V )\big)(s)\|_\infty}{|t-s|^\rho}\\
\lesssim \bar T^{\frac{2\gamma-2\rho-\alpha}{2}}\sup_{t\in[T-\bar T,T]}(T-t)^{\frac{\alpha-1}{2}}\|\nabla u\circ V(t)\|_{2\gamma-3}
\end{multline*}\normalsize
which in turn can be bounded via Proposition~\ref{prop:reson}. Concerning, the so called Gubinelli derivative, by definition, $M_{\bar T}(u,u')'=\nabla u$ and, by assumption, $u\in C_{\bar T,T}\CC^\theta$. Hence
$$
\|M_{\bar T}(u,u')'\|_{C_{\bar T,T}\CC^{\alpha-1}_{\R^d}}\lesssim\|u\|_{C_{\bar T,T}\CC^\alpha}\lesssim\|u\|_{C_{\bar T,T}\CC^\theta}\lesssim\|(u,u')\|_{\mathscr D^{\alpha,\theta,\rho}_{\bar T,T,V}}
$$
However this is not yet the needed bound due to the missing factor $\bar T$ to some positive power. Let us observe that 
\begin{equation*}
\begin{split}
\|\nabla u(t)\|_{\alpha-1}\lesssim\|\nabla u(t)-\nabla u(T)\|_{\alpha-1}+\| u^T\|_{\alpha}
\end{split}
\end{equation*}
Now it suffices to notice that we can estimate the first summand in two different ways
\begin{equation*}
\|\Delta_i(\nabla u(t)-\nabla u(T))\|_{\infty}\lesssim
\left\{
\begin{split}
&2^{-i(\theta-1)}\|u\|_{C_{\bar T,T}\CC^\theta}\\
&(T-t)^{\rho}\|\nabla u\|_{C^\rho_{\bar T,T} L^\infty_{\R^d}}
\end{split}
\right.
\end{equation*}
where $\Delta_i$ is the $i$-th Littlewood-Paley block. Then interpolating this two bounds we get 
\small$$
\|\nabla u(t)-\nabla u(T)\|_{\alpha-1}\lesssim \bar T^{\rho(1-\eps)}\|u\|^\eps_{C_{\bar T,T}\CC^\theta}\|\nabla u\|^{1-\eps}_{C^\rho_{\bar T,T} L^\infty_{\R^d}}\lesssim \bar T^{\rho(1-\eps)}\|(u,u')\|_{\DD_{\bar T, T,V}^{\alpha,\theta,\rho}}
$$\normalsize
with $\eps\eqdef\frac{\alpha-1}{\theta-1}\in(0,1)$. Therefore 
$$
\|M_{\bar T}(u,u')'\|_{C_{\bar T,T}\CC^{\alpha-1}_{\R^d}}=\|\nabla u\|_{C_{\bar T,T}\CC^{\alpha-1}_{\R^d}}\lesssim \bar T^{\rho(1-\eps)}\|(u,u')\|_{\DD_{\bar T, T,V}^{\alpha,\theta,\rho}}+\|u^T\|_{\theta}
$$
and we can now move to the term involving the remainder $M_{\bar T}(u,u')^\sharp$. 
By definition, $M_{\bar T}(u,u')^\sharp$ is given by
\small\begin{equation*}
\begin{split}
M_{\bar T}(u,u')^{\sharp}&=M_{\bar T}(u,u')-\CJ^T(f)-\nabla u\prec \CJ^T(V)
\\&=\Psi^T+\big(\CJ^T(\nabla u\prec V)-\nabla u\prec\CJ^T(V)\big)+\CJ^T(\nabla u\succ V)+\CJ^T(\nabla u\circ V)
\end{split}
\end{equation*}\normalsize
Now, by Schauder's estimates we directly have 
$$
(T-t)^{\frac{\alpha-1}{2}}\|\Psi^T(t)\|_{2\alpha-1}\lesssim (T-t)^{\frac{\gamma-\alpha}{2}}\|u^T\|_\gamma
$$
which gives the needed bound for the term $\Psi^T$. Lemma~\ref{lemma:comm-sc} and the fact that $\alpha<\theta$ imply 
\small\begin{multline*}
(T-t)^{\frac{\alpha-1}{2}}\|\CJ^T(\nabla u\prec V)-\nabla u\prec\CJ^T(V)\|_{C_T\CC^{2\alpha-1}_{\R^d}}\\
\lesssim \bar T^\kappa (T-t)^{\frac{\alpha-1}{2}}\big(\|u\|_{C_{\bar T,T}\CC^\theta}+\|\nabla u\|_{C^\rho_{\bar T,T} L^\infty_{\R^d}}\big)\|V\|_{C_T\CC^\gamma_{\R^d}}
\end{multline*}\normalsize
For $\CJ^T\big(\nabla(\nabla u\succ V \big))$ we exploit once more Corollary~\ref{cor} and Bony's estimates, so that
\begin{equation*}
(T-t)^{\frac{\alpha-1}{2}}\|\CJ^T(\nabla u\succ V)(t)\|_{2\alpha-1}\lesssim \bar T^{\frac{\theta+\gamma-\alpha-1}{2}}\|u\|_{C_{\bar T,T}\CC^\theta}\|V\|_{C_T\CC^\gamma_{\R^d}}
\end{equation*}
At this point it remains only to bound the norm of the term $\CJ^T(\nabla u\circ V)$. Again by Corollary~\ref{cor} and Proposition~\ref{prop:reson} we have
\small\begin{align*}
(T-&t)^{\frac{\alpha-1}{2}}\|\CJ^T(\nabla u\circ V)(t)\|_{2\alpha-1}\lesssim \bar T^{\gamma-\alpha}\sup_{t\in[0,T]}(T-t)^{\frac{\alpha-1}{2}}\|\nabla u\circ V(t)\|_{2\gamma-3}\\
&\lesssim \bar T^{\gamma-\alpha}\bigg(\1_F\|f\|_{C_TL^\infty}\| \CV_1\|_{C_T\CC^{\gamma-2}_{\R^d}}+\Big(1+\| \CV\|_{ \cX^\gamma}\Big)^2\big(1+\|(u,u')\|_{\DD^{\alpha,\theta,\rho}_{\bar T, T,V}}\big)\Big)
\end{align*}\normalsize
where $F\eqdef \{f\in C_TL^\infty\}$.  Now, putting all the previous estimates together we conclude the validity of the bound~\eqref{eq:estim-map-1}. Notice that the map $(u,u')\mapsto M_{\bar T}(u,u')-\Psi^T-\CJ^T(f)$ is linear and therefore~\eqref{eq:estim-map-2} can be obtained by the previous computations, simply replacing $u$ with $u-v$ and $u'$ with $u'-v'$.

At last, thanks to~\eqref{eq:estim-map-2}, we see that there exists $T^\star=T^\star(\|\CV\|_{\cX^\gamma})>0$ small enough such that the map $\mathscr M_{T^\star}$ is a strict contraction from $ \DD^{\alpha,\theta,\rho}_{T^\star, T,V}$ into itself.
\end{proof}

As in the Young case, the previous proposition represents the crucial technical tool through which we can state and prove the following theorem. 

\begin{theorem}\label{th:Generator-r}
Let $\beta\in\left(-\frac{2}{3},-\frac{1}{2}\right]$,  $\frac{4}{3}<\theta<\gamma<\beta+2$ and $T>0$. Let $\mathscr S_c$ be the operator assigning to every triplet $(u^T,f,\eta)\in\CC^\gamma\times C_T\CC^2\times C_T\CC^\infty_{\R^d}$ the solution $u\in C_T\CC^\theta$ to  equation~\eqref{eq:generator}. 

Then, there exists a locally Lipschitz continuous map $\mathscr S_r: \CC^\gamma \times \big(C_TL^\infty\cup\{V^k,k=1,. . .,d\}\big)\times\cX^\gamma\to C_T \CC^\theta$ that extends $\mathscr S_c$ in the following sense 
$$
\mathscr S_c(u^T,f,\eta)(t)=\mathscr S_r(u^T,f, \CK(\eta))(t),
$$
for all $t\leq T$ and $(u^T,f,\eta)\in\CC^\gamma\times C_T\CC^2\times C_T\CC^\infty_{\R^d}$. Moreover, for any $\rho<\frac{\theta-1}{2}$, $\mathscr S_r$ takes values in  $C_T^{\frac{\theta}{2}}L^\infty$ and $\nabla \mathscr S_r\in C^{\rho}_TL^\infty_{\R^d}$.
\end{theorem}

\begin{proof}
As in the proof of Theorem~\ref{th:Generator-Young} and thanks to Proposition~\ref{prop:fixedpointrough}, we can apply Banach fixed point theorem and get the existence of a unique solution $(u,\nabla u)\in\DD^{\alpha,\theta,\rho}_{T^\star, T,V}$ to~\eqref{eq:generator}. 
Moreover, for $T>0$ fixed, $T^\star$ is independent on the terminal condition $u^T$, hence we can iterate our fixed point procedure on $[T-2T^\star,T-T^\star],[T-3T^\star,T-2T^\star],\dots$ and extend our solution to the whole interval $[0,T]$. 

Since the solution $u$ is obtained through a fixed point procedure on the space of paracontrolled distributions, it is well-known that it gives rise to a continuous flow $(u^T,f,\CV)\mapsto\mathscr S_r(u^T,f,\CV)$ (see~\cite{GIP15} for more details). 

Let $V$ be a smooth function and $\CV=(V,\CJ^T(\partial_j V^i)\circ V^j)_j$ its enhancement.  
The algebraic expansion given by the equation~\eqref{eq:algebraic} implies that the term $\nabla u\cdot V$, defined in Proposition~\ref{prop:reson}, coincides with the usual product and, therefore, the solution $u$ constructed via the fixed point argument outlined above corresponds to the classical one by uniqueness. Therefore the relation
$$
\mathscr S_c(u^T,f,V)=\mathscr S_r(u^T,f,(V,\CJ^T(\partial_j V^i)\circ V^j))
$$
is justified, where we recall that $\mathscr S_c$ is the flow of the equation 
\[
\mathscr G^Vu=h \qquad\qquad u(T,\cdot)=u^T
\]
and this completes the proof of Theorem~\ref{th:Generator-r}.
\end{proof}

%%%%%%%%%%%%%%%%%%%%%%%%%%%%%%%%% MARTINGALE PROBLEM %%%%%%%%%%%%%%%%%%%%%%%%

\section{The Martingale Problem}\label{sec:MartProblem}

In the previous section, we solved the generator equation and Theorems~\ref{th:Generator-Young} and~\ref{th:Generator-r} represent the formal version of what was loosely stated in Theorem~\ref{t:GenEq}.
As it was mentioned in the introduction, this was the first step we had to undertake in order to be able to formulate and prove well-posedness for the SDE, formally given by
\begin{equation}\label{e:SDE1}
\dd X_t=V(t,X_t)\dd t+\dd B_t,\qquad X_0=x
\end{equation}
where $B$ is a $d$-dimensional Brownian motion, $x$ a point in $\R^d$ and $V$ is a function of time taking values in $\CC^\beta_{\R^d}$, for $\beta\in(-\frac{2}{3},0)$. 
Before proceeding, let us introduce a simple convention that collects under one name the rough and the Young regime.

\begin{definition}\label{def:ground}
Let $\beta\in(-\frac{2}{3},0)$. We say that $V\in C_T\CC^\beta_{\R^d}$ is a {\it ground drift} if either $\beta\in(-\frac{1}{2},0)$ or $\beta\in(-\frac{2}{3},-\frac{1}{2}]$ and that $V$ can be lifted to an element $\CV\in\cX^\gamma$, for some $\gamma<\beta+2$.
\end{definition}

\noindent We are now ready to formulate a suitable Stroock-Varadhan martingale problem for~\eqref{e:SDE1}, namely

\begin{definition}\label{def:sv}
Let $T>0$ and $V\in C_T\CC^\beta_{\R^d}$ be a ground drift according to Definition~\ref{def:ground}. Let $\Omega=C([0,T],\R^d)$ and $\CF=\cb(C([0,T],\R^d))$, the usual Borel $\sigma$-algebra on it. 
We say that a probability measure $\prob$ on $(\Omega,\CF)$, endowed with the canonical filtration $(\CF_t)_{0\leq t\leq T}$, solves the martingale problem with generator $\mathscr G^V$ starting at $x\in\R^d$, if the canonical process $X_t(\omega)=\omega(t)$ satisfies the two following properties
\begin{enumerate}
\item $\prob(X_0=x)=1$
\item For every $\tau\leq T$, $f\in C_TL^\infty$ and every $u^\tau\in\CC^{\beta+2}$ the process
 $$
 \left\{u(t,X_t)-\int_0^tf(s,X_s)\dd s\right\}_{t\in[0,\tau]}
 $$ 
 is a square integrable martingale under $\mathbb P$, where $u$ is the solution of the generator equation~\eqref{eq:generator} constructed in Theorems~\ref{th:Generator-Young} and~\ref{th:Generator-r}.
\end{enumerate}
\end{definition}

\noindent The next theorem guarantees that the Stroock-Varadhan Martingale Problem formulated in the previous definition is indeed well-posed (see also Theorem~\ref{t:MartProb}). 

\begin{theorem}\label{th:svm}
Let $T>0$ and $V\in C_T\CC^\beta_{\R^d}$ be a ground drift according to Definition~\ref{def:ground}. Then there exists a unique probability measure $\prob$ on $\left(\Omega,\CF,(\CF_t)_{0\leq t\leq T}\right)$ which solves the martingale problem with generator $\mathscr G^V$ starting at $x$, for every $x\in\R^d$. Moreover, the canonical process $X_t(\omega)=\omega(t)$ under $\prob$ is strong Markov.
\end{theorem}

\begin{proof}
We will focus on the case $\beta\in(-\frac{2}{3},-\frac{1}{2}]$, the case $\beta>-\frac{1}{2}$ being analogous. 
From now on we will take $(\rho,\theta,\gamma)\in \R^3$ as in Theorem~\ref{th:Generator-r}, $V\in C_T\CC^\beta_{\R^d}$ such that there exists $V^n$ a smooth regularization of $V$ for which, as $n\to\infty$, $\CK(V^n)$ converges to $\CV$ in $\ch^\gamma$, where the operator $\CK$ is defined according to Definition~\ref{def:rough}. 
\newline

\noindent \textbf{Existence:} 
Let $X^n$ be the unique strong solution of the SDE
\begin{equation}\label{SDEfortheproof}
\dd X_t^n=V^n(t,X^n_t)\dd t+\dd B_t,\qquad X_0= x.
\end{equation}
For $i=1,. . .,d$, let $u^n=(u^{n,1},. . .,u^{n,d})$ be such that for every $i$, $u^{n,i}$ is the unique solution of the equation 
$$
\mathscr G^{V^n}u^{n,i}=V^{n,i},\qquad u^{T}(x)=0.
$$
Take $0<s<t<T$ and apply It\^o's formula to the process $\{u^n(t,X^n_t)\}_t$, so that
\small\begin{align*}
u^n(t,X^n_t)-u^{n}(s,X^n_s)&=\int_s^tV^{n}(r,X^n_r)\dd r+\int_s^t\nabla u^n(r,X^n_r)\dd B_r\\
&=X^n_t-X^n_s-(B_t-B_s)+\int_s^t\nabla u^n(r,X^n_r)\dd B_r
\end{align*}\normalsize
where the last equality is a direct consequence of the fact that $X^n$ solves~\eqref{SDEfortheproof} by construction. In order to prove tightness for the sequence $(X^n)_n$ we want to apply Kolmogorov's criterion, therefore we need to bound the $p$-th moment of the increments of $X^n$, uniformly in $n$. For $p\geq 1$, by standard properties of the Brownian motion $B$ and Burkholder-Davis-Gundy inequality, we obtain
\small\begin{align}
\E\left[|X^n_t-X^n_s|^p\right]\lesssim&\E\left[|u^n(t,X^n_t)-u^n(s,X^n_s)|^{p}\right]\notag\\
&+|t-s|^{p/2}+\mathbb E\left[\left(\int_s^t|\nabla u^n(r,X^n_r)|^2\dd r\right)^{p/2}\right]\label{ineq:inter}
\end{align}\normalsize
Notice that the last term of the previous can be bounded by
\[
\int_s^t|\nabla u^n(r,X^n_r)|^2\dd r\lesssim\int_s^t\|\nabla u^n(r,.)\|^2_{\infty}\dd r\lesssim(t-s)\|u^n\|^2_{C_T\CC^\theta}
\]
where we recall that $\theta>1$ and hence $\CC^{\theta-1}$ is continuously embedded in $L^\infty(\R^d)$. Adding and subtracting $u^n(s,X_t)$, the first summand in~\eqref{ineq:inter} becomes
\[
|u^n(t,X_t)-u^n(s,X_s)|\lesssim\|u^n(t)-u^n(s)\|_{\infty}+\|\nabla u^n(s)\|_{\infty}|X_t^n-X_s^n|
\]
Now, for the first term we can exploit the regularity in time of our solution, while for the second $\|\nabla u^n(s)\|_{\infty}=\|\nabla u^n(T)-\nabla u^n(s)\|_{\infty}\lesssim T^\rho\|\nabla u^n\|_{C^\rho L^\infty_{\R^d}}$, since we chose $u^n$ as the solution to the generator equation with zero terminal condition. 

Since $u^n$ converges to the solution $u$ constructed in the Theorem~\ref{th:Generator-r} in the topology of $\DD^{\alpha,\theta,\rho}_{T,V}$, each of the norms of $u^n$ is bounded by the analogous of $u$ and~\eqref{ineq:inter} becomes
\begin{align*}
\E[|X^n_t-X^n_s|^p]\lesssim& |t-s|^{p\frac{\theta}{2}}\|u\|_{C^{\frac{\theta}{2}}_TL^\infty}\\
&+T^{p\rho}\|\nabla u\|^p_{C^\rho_TL^\infty_{\R^d}} \E[|X_t^n-X_s^n|^{p}]+|t-s|^{p/2}(1+\|u\|^p_{C_T\CC^\theta})
\end{align*}
At this point, the bound~\eqref{eq:estim-map-1} in Proposition~\ref{prop:fixedpointrough} guarantees that it is possible to choose $T^\star>0$ such that $T^{\star}(1+\|\CV\|_{\,\!_{\mathscr H^\gamma}})^2\ll1$. Pulling the second summand of the right hand side to the left hand side, we obtain
$$
\E[|X^n_t-X^n_s|^p]\lesssim |t-s|^{p\frac{\theta}{2}}\|u\|_{C^{\frac{\theta}{2}}_TL^\infty}+|t-s|^{p/2}(1+\|u\|^p_{C_T\CC^\theta})
$$
for all $0<s<t<T^\star$, uniformly in $n$ (the right hand side does not depend on $n$ anymore). Denote by $X^{n,1}(t)=X^n(T^\star+t)$. Since $T^\star$ does not depend on the initial condition $x$ and the solution $u$ is defined on the whole interval $[0,T]$, we can repeat the previous argument so that 
$$
\E[|X^{n}_{t+T^\star}-X^{n}_{s+T^\star}|^p]=\E[|X^{n,1}_t-X^{n,1}_s|^p]\lesssim|t-s|^{p/2}
$$
for all $s,t\leq T^\star$, uniformly in $n$. Now, when $s\leq T^\star\leq t\leq2T^\star$ we have that 
\small$$
\E[|X_t^n-X_s^n|^p]\lesssim_p\E[|X^n_t-X^n_{T^\star}|^p]+\E[|X^n_{T^\star}-X^n_s|^p]\lesssim|T^\star-t|^{p/2}+|T^\star-s|^{p/2}\lesssim|t-s|^{p/2}
$$\normalsize
Iterating the procedure over $[2T^\star,3T^\star], [3T^\star,4T^\star],\dots$, we finally get
$$
\sup_{n} \mathbb E[|X^n_t-X^n_s|^p]\lesssim|t-s|^{p/2}
$$
for all $s,t\leq T$. At this point, we can apply Kolmogorov's criterion which implies tightness  of the sequence $(X^n)_n$ in $C([0,T], \R^d)$. 

It remains to show that every limiting process solves our martingale problem. To this purpose, let $(X^n)_n$ be a converging subsequence, $\tau\leq T$, $(f,u^\tau)\in C_TL^\infty\times \CC^{\gamma}$ and $u^n$ be the solution to the generator equation $\mathscr G^{V^n}u^{n}=f$ with terminal condition $u^\tau$. Applying It\^o's formula to $u^n(t,X^n_t)$ we obtain
\[
u^n(t,X^n_t)-u^n(0,x)-\int_0^t f(s,X^n_s)\dd s= \int_0^t \nabla u^n(s,X^n_s)\dd B_s
\]
Let $Z_t^n$ denote the left hand side of the previous. Then
\[
\E\left|Z_t^n\right|^2\lesssim T\|\nabla u^n\|_{C_TL^\infty_{\R^d}}\lesssim T\|\nabla u\|_{C_TL^\infty_{\R^d}}
\]
which implies that $(Z^n_t,t\leq T)$ is a bounded sequence of square integrable martingales. Now, since for every $n$, $Z^n$ is a martingale, we have that
\begin{equation}\label{eq:ma}
\mathbb E[(Z^n_t-Z^n_s)F(X_r^n,r\leq s)]=0
\end{equation}
holds for any continuous functional $F:C([0,s],\mathbb R^d)\to\mathbb R$. At this point, to complete the proof, we only need to pass to the limit in the previous equality.  
Let us observe that, thanks to the fact that $X^n$ converges in distribution to $X$ and $(u^n,\nabla u^n)$ converges uniformly to $(u,\nabla u)$, also $Z^n$ converges in distribution to $Z_t=u(t,X_t)-u(0,x)-\int_0^t f(s,X_s)\dd s$. 
Analogously, $(Z^n_t-Z^n_s)F(X_r^n,r\leq s)$ converges in distribution to $(Z_t-Z_s)F(X_r,r\leq s)$ and, since $(Z^n_t,t\leq T)$ is a sequence with uniformly bounded second moment, which in particular implies that $((Z^n_t-Z^n_s)F(X_r^n,r\leq s))_{n}$ is a uniformly integrable family, we can interchange limit and expectation in the identity~\eqref{eq:ma} by dominated convergence theorem (and Skorohod representation theorem), so that at last we get
$$
\mathbb E[(Z_t-Z_s)F(X_r,r\leq s)]=0
$$
which proves the claim. 
\newline

\noindent \textbf{Uniqueness and strong Markov property:} Let $\prob_1$ and $\prob_2$ be two solutions of the martingale problem starting at $x$. Let $f\in C([0,T],L^\infty(\R^d))$ and $u$ be the solution of the generator equation $\mathscr G^V u=f$ with zero terminal condition.   Since under both $\prob_1$ and $\prob_2$ the canonical process $X$ is such that $\left\{u(t,X_t)-\int_0^tf(s,X_s)\dd s\right\}_{t\in[0,T]}$ is a martingale, we have
$$
u(0,x)=\E_{\prob_i}\left[u(T,X_T)-\int_0^Tf(s,X_s)\dd s\right]=-\E_{\prob_i}\left[\int_0^Tf(s,X_s)\dd s\right]
$$
for $i=1,\,2$. Therefore,
$$
\E_{\prob_1}\left[\int_0^Tf(s,X_s)\dd s\right]=\E_{\prob_2}\left[\int_0^Tf(s,X_s)\dd s\right]
$$ 
Since the previous holds for every $f\in C([0,T],L^\infty(\R^d))$, we conclude that the process $X$ has the same marginals under $\prob_1$ and $\prob_2$. 
By a straightforward adaptation of \cite[Theorem 4.2]{EK} (the main difference lying on the fact that our generator is time-dependent, but that does not affect the proof in any sense), we deduce that it has the same finite dimensional distributions and it is Markov with respect to both probability measures, which in turn guarantees uniqueness. 
For the strong Markov property we need instead \cite[Theorems 4.6 and 4.2]{EK}.  

\end{proof}

%%%%%%%%%%%%%%%%%%%%%%%%%%%%%%%POLYMER MEASURE%%%%%%%%%%%%%%%%%%%%%%%%%%%%%%%

\section{Construction of the Polymer Measure}\label{sec:Polymer}

In this section we will construct the so called polymer measure in dimension $d=2,3$ and show how to exploit the techniques developed so far to prove Theorem~\ref{t:ConstrPol}. More concretely, our purpose is to make sense of
\begin{equation}\label{eq:polymer}
\mathbb Q_T(\dd\omega)=Z_0^{-1}\exp\left(\int_0^T\xi(\omega_s)\dd s\right)\mathbb W_T(\dd\omega)
\end{equation}
where $\mathbb W$ is the Wiener measure on $C([0,T],\R^d)$, $d=2,3$, $\xi$ a spatial white noise on the $d$-dimensional torus $\mathbb T^d$ independent of $\mathbb W$, and $Z_0$ is an infinite renormalization constant. 
Let us recall that the periodic space Gaussian white noise is a centered Gaussian random field which formally satisfies
\begin{equation}\label{e:WN}
\E[\xi(x)\xi(y)]=\der(x-y)
\end{equation}
for any two points $x,\,y\in\TT^d$, where, again, $\TT^d$ is the $d$-dimensional torus and $d=2$ or $3$. 
As the covariance function in~\eqref{e:WN} suggests, the white noise is too singular for~\eqref{eq:polymer} to make sense. 
In order to have an expression that we can manipulate, we consider a mollified version of the noise, defined by
\begin{equation}\label{e:MollNoise}
\xi^\eps=\sum_{k\in\mathbb Z^d}m(\eps k)\hat \xi(k)e_k
\end{equation}
where $\{\hat \xi(k)\}_{k\in\Z^d}$ is a family of standard normal random variables with covariance $\E[\hat \xi(k_1)\hat \xi(k_2)]=\1_{\{k_1=-k_2\}}$, $e_k$ is the Fourier basis $L^2(\mathbb T^d)$ and $m$ a smooth radial function with compact support such that $m(0)=1$.

Now, given $\xi_\eps$, let $\mathbb Q^\eps$ be the measure defined by
\[
\mathbb Q^\eps_T(\dd\omega)=Z^{-1}_\eps\exp\left(\int_0^T\xi^\eps(\omega_s)\dd s\right)\mathbb W(\dd \omega),\,\,\tilde Z_\eps=\mathbb E_{\mathbb W}\left[\exp\left(\int_0^T\xi^\eps(\omega_s)\dd s\right)\right]
\]
and $h^\eps:\R^+\times\mathbb T^d\to\R$ be the local in time solution to the equation 
\begin{equation}\label{eq:KPZtype}
\partial_th^\eps=\frac{1}{2}\Delta h^\eps+\frac{1}{2}|\nabla h^\eps|^2+\xi^\eps-c_\eps \qquad h(0,x)=0
\end{equation}
where $c_\eps$ is a constant that will be characterized in Theorem~\ref{th:stoc}. For $\xi^\eps$ smooth, $h^\eps$ is known to exist and be regular, therefore the process 
\[
M_t^\eps(\omega_s)=\int_0^t\nabla h^\eps(T-s,\omega_s)\dd \omega_s\quad \langle M^\eps\rangle_t(\omega)=\int_0^t|\nabla h^\eps(T-s,\omega_s)|^2\dd s
\]
where $\langle M\rangle_\cdot$ is the quadratic variation of $M$, is clearly a square integrable martingale. Girsanov's theorem then implies that, under the measure defined by
\[
\tilde{\mathbb Q}^\eps_T(\dd\omega)=\exp\left(M_T^\eps-\frac{1}{2}\langle M^\eps\rangle_T\right)\mathbb W(\dd\omega),
\]
the canonical process has the same law as the solution $X^\eps$ to the SDE
\[
\dd X^\eps_t=V^\eps(t,X^\eps_t)\dd t+\dd B_t,\quad X_0=x
\]
when one chooses $V^\eps(t,x)$ to be $\nabla h^\eps (T-t,x)$. But now, applying It\^o's formula to $h^\eps (T-t,X^\eps_t)$ and recalling that $h^\eps$ solves~\eqref{eq:KPZtype} we conlude that $\tilde{\mathbb Q}_T^\eps(\dd \omega)=\mathbb Q^\eps_T(\dd\omega)$.

%Therefore it is enough to prove that the law of $X^\eps$ converges in order to guarantee the existence of a limiting measure for the sequence $\mathbb Q^\eps$. 
At this point we can take advantage of Theorem~\ref{th:svm}, whose applicability is ensured by the next proposition, which guarantees the existence of a unique limiting measure for the sequence $(\tilde{\mathbb Q}^\eps_T)_\eps$ and consequently for the sequence $(\mathbb Q^\eps_T)_\eps$.

\begin{proposition}\label{prop:conv-kpz}
Let $h^\eps$ be the local in time solution to~\eqref{eq:KPZtype} for $d=2,3$ and $V^\eps(t,x)=\nabla h^\eps(T-t,x)$. Then, there exists $T^\star>0$ such that for all $T\leq T^\star$, $V^\eps(t,x)$ is a ground drift according to definition~\ref{def:ground}, i.e. we have 
\begin{enumerate}
\item for $d=2$ the process $V^\eps$ converges almost surely in $C([0,T^\star],\CC^\beta(\mathbb T^2))$ for all $\beta<0$ to some element $V$.
\item for $d=3$  and all $\beta<-1/2$ the process $\CK(V^\eps)$ converges almost surely in $\ch^{\beta+2}(\mathbb T^3)$ to some element $\CV\in \cX^{\beta+2}$.
\end{enumerate}
Moreover in both cases the limit is independent of the choice of the mollifier $m$.
\end{proposition} 

\begin{remark}
Notice that  we are applying Theorem~\ref{th:svm} to distributions defined on the torus and not on the full space. This is completely harmless since the space $\CC^\gamma(\mathbb T^d)$ can be seen as the space of periodic distributions lying in $\CC^\gamma$ (see also~\cite[Appendix A]{GIP15} for a discussion on this aspect).
\end{remark}

\noindent Let us stress the fact that the proof of Proposition~\ref{prop:conv-kpz} boils down to a well-posedness result for the equation 
\begin{equation}\label{eq:KPZtype1}
\partial_th=\frac{1}{2}\Delta h+\frac{1}{2}|\nabla h|^2+\xi,\qquad\quad h(0,x)=0.
\end{equation}
In the one dimensional case with $\xi$ a space-time white noise, the previous is nothing but the celebrated Kardar-Parisi-Zhang equation~\cite{KPZ}, which was successfully studied by M.Hairer in~\cite{hairer_solving_2013} and subsequently by M. Gubinelli and N.Perkowski in~\cite{GP}. The regularity issues one encounters when dealing with the three dimensional version are morally the same these authors had to face and the techniques we will exploit are somewhat similar to theirs (especially to~\cite{GP}). For the sake of completeness, we will prove Proposition~\ref{prop:conv-kpz} pointing out the difficulties one has to overcome and illustrating the main steps one needs to undertake in order to solve~\eqref{eq:KPZtype1}, still keeping it as concise as possible and referring the interested reader to the quoted papers.

%%%%%%%%%%%%%%%%%%%%%%%%%%%%KPZ TYPE EQUATION%%%%%%%%%%%%%%%%%%%%%%%%%%%%%%%

\section{A KPZ-type equation driven by a purely spatial white noise}\label{section:KPZ}

%\subsection{Expansion of the equation}
The aim of this section is to prove well-posedness of the KPZ-type equation, introduced in~\eqref{eq:KPZtype1} to make sense of the polymer measure with white-noise potential. 
We will focus on the three-dimensional case, since in dimension $2$ the result follows by analogous, but simpler, arguments.  

Let us consider the case of non-zero initial condition, $h_0$, and  write~\eqref{eq:KPZtype1} in its mild formulation
\begin{equation}\label{KPZmild}
h(t)=P_t h_0+ \I(|\nabla h|^2 )(t)+\I(\xi)_t
\end{equation}
where $P_t\eqdef e^{\frac{1}{2}t\Delta}$ is the heat flow, for a function $f$ on $(0,T]\times\TT^3$, $\I(f)(t)\eqdef \int_0^t P_{t-s}f(s)\dd s$ and $\xi$ is the usual space white noise on $\TT^3$, i.e. a centered Gaussian random field whose covariance function is formally given as in~\eqref{e:WN}. 

The problem with the previous equation lies in the fact that, since as a random distribution, $\xi\in\CC^{\theta}(\mathbb T^d)$ for $\theta<-\frac{d}{2}$ (which in $d=2$ means $\theta<-1$ while in $d=3$, $\theta<-\frac{3}{2}$) stantard Schauder's estimates suggest that the spatial regularity of $h$ cannot be better than $\theta+2$ and therefore the non-linearity in~\eqref{KPZmild}, for both $d=2$ and $3$, is not well-defined. 
Now, let us point out that the term determining the regularity of $h$ is  $\I(\xi)$, so maybe, upon subtracting it to the potential solution, what remains is more regular. 
In other words, one defines $h_1\eqdef h- \I(\xi)$, derives the equation it should solve and, as before, guesses its regularity. 
For example, setting $X=\I(\xi)$, $h_1$ should satisfy
\[
h_1(t)=P_t h_0+\I(|\nabla X|^2)(t)+\I(\nabla X\cdot\nabla h_1 + |\nabla h_1|^2)(t)
\]
and its regularity should be as the one of $\I(|\nabla X|^2)$. {\it If} it were well-posed, this last term would be $2\theta+4$-H\"older in space which is strictly greater than $\theta+2$ so that indeed $h_1$ is more regular than $h$. 
While in dimension $2$ this is enough (given that $X$ and $\I(|\nabla X|^2)$ can be constructed and belong to the correct Besov-H\"older space, all the other terms satisfy Bony's condition), it is still not sufficient in $d=3$ and so, we proceed further in the expansion. 

The problem is that after subtracting a finite number of terms, there will be no more gain in regularity and something else is needed in order to define the ill-posed product and consequently solve the equation.  
This is exactly the point in which the paracontrolled approach, as we will see in what follows, enters the game. 
\newline

Now that we have given a heuristic idea of what is going on, let us be more formal. 
We begin by defining the objects that will appear in our expansion. Let $\eta$ be a smooth function and set 
\begin{equation}\label{def:stochterms}
\begin{split}
X_t(\eta) \eqdef \I(\eta)(t),\qquad  X_t^{\<tree12>}(\eta) \eqdef  \I(|\nabla X|^2)&(t), \qquad X^{\<tree122>}_t(\eta) \eqdef  \I(\nabla X^{\<tree12>}\cdot \nabla X)(t),\\
X_t^{\<tree1222>}(\eta)\eqdef  \I(\nabla X^{\<tree122>}\cdot \nabla X)(t),\quad&\quad  X_t^{\<tree124>}(\eta) \eqdef  \I(|\nabla X^{\<tree12>}|^2)(t).
\end{split}
\end{equation} 
As announced before, in case $\eta$ is the space white noise, the previous stochastic processes are not {\it analitically} well-defined and we will have to exploit stochastic calculus tools in order to make sense of them and prove that they satisfy certain regularity requirements. 

Now, let $h$ be the solution of~\eqref{KPZmild} driven by $\eta$ and $v$ be given by
\[
v\eqdef h-X(\eta)-X^{\<tree12>}(\eta) - 2X^{\<tree122>}(\eta)\,.
\]
Plugging this expression back into~\eqref{KPZmild}, we see that $v$ solves
\begin{equation}\label{KPZmilduQ}
v(t) = P_t h_0 + 4 X_t^{\<tree1222>}(\eta) + 2\, \I\Big(\nabla v \cdot \nabla X(\eta)\Big) (t)+R^v(t)
\end{equation}
where $R^v$ is defined as
\begin{equation}\label{remainder1}
R^v(t)\eqdef X_t^{\<tree124>}(\eta)+\I\Big(2 \nabla X^{\<tree12>}(\eta) \cdot \nabla \big(2 X^{\<tree122>}(\eta) + v\big) + \big | \nabla \big(2 X^{\<tree122>}(\eta) + v\big)\big|^2\Big)(t)
\end{equation}
At this point, we will split the analysis of the equation in two distinct modules. 
On one side, with purely analytical arguments, we will identify a suitable subspace of the space of distributions, depending on the processes defined above, for which it is possible to make sense of the ill-posed operations in~\eqref{KPZmilduQ} and formulate a fixed point map that is continuous in these data. 
On the other, we will exploit probabilistic techniques to construct such processes starting with a white noise $\xi$ and prove they have the expected regularity, through a regularization procedure.

\begin{notation*}
From now on, all the functions and distributions we will consider will live on the $d$-dimensional torus. Since no confusion can occur, we will indicate the function spaces with the same notations introduced in Section~\ref{sec:FunctionSpaces}, but the domain will not be $\R^d$ but $\TT^d$. 
\end{notation*}

\subsection{Analytic part}

We begin by specifying the space in which our stochastic processes live. 

\begin{definition}[Rough Distribution]\label{def:rough distribution}
Let $\varrho,r<\frac{1}{2}$ be such that $\varrho+2r<\frac{1}{2}$. For $(a,b,\eta)\in\R^2\times C_T\CC^2$ and $t\leq T$ set $\XX(\eta,a,b)$ to be
\begin{equation}\label{e:6-uple}
\XX_t(\eta,a,b)=\Big(X_t, X^{\<tree12>}_t-at, X^{\<tree122>}_t, X^{\<tree1222>}_t, X^{\<tree124>}_t-bt, Q\circ \nabla X(t)\Big)(\eta)\,,
\end{equation}
where $X^{\<tree12>}$, $X^{\<tree122>}$, $X^{\<tree1222>}$, $X^{\<tree124>}$ are given by~\eqref{def:stochterms}, and 
$$
Q(\eta)\eqdef \I(\nabla X)\qquad\text{and}\qquad\nabla Q\circ\nabla X(\eta)=(\partial_i (Q)^j\circ\partial_iX)_{i,j=1,2,3}\,.
$$
We define the space $\XR^{\varrho,r}$ of {\it rough distributions} as 
$$
\XR^{\varrho,r}=\text{cl}_{\CH^{\varrho,r}}\left\{\XX(\eta,a,b), \quad (a,b,\eta)\in\mathbb R^2\times C([0,T],\mathscr C^2(\mathbb T^3))\right\}
$$
where $\text{cl}_{\CH^{\varrho,r}}\{\cdot\}$ denotes the closure of the set in brackets with respect to the topology of $\CH^{\varrho,r}$ and the space $ \CH^{\varrho,r}=  C^{r}_T\CC^\varrho\times  C^r_T\CC^{2\varrho}\times C^r_T\CC^{3\varrho}\times  C^r_T\CC^{\varrho+1} \times C^r_T\CC^{4\varrho}\times  C^r_T\CC^{2\varrho-1}_{\R^3}$ equipped with its usual norm.  
We will denote by $\XX$ a generic $6$-uple given as in~\ref{e:6-uple} belonging to this space. 
Moreover if $\eta\in C_T\CC^\varrho$ coincides with the first component of $\XX\in\XR^{\varrho,r}$ we will say that $\XX$ is a enhancement (or lift) of $\eta$. 
\end{definition}

\begin{remark}
The reason why in~\eqref{e:6-uple} we had to add an extra term to the ones introduced in~\eqref{def:stochterms} will soon be clarified. 
Intuitively, this is the term we will need to define the ill-posed product between the gradient of the expected solution $v$ of~\eqref{KPZmilduQ} and the gradient of $X$. 
This is very similar to what we have done for the generator equation in Section~\ref{sec:solv-Generator}. 
\end{remark}

\begin{remark}
It is important to notice that the two constants $a,b$ appearing in Definition~\ref{def:rough distribution} play the role of renormalization constants. 
As we said at the beginning of Section~\ref{section:KPZ}, if $\xi$ denotes the three dimensional space white noise, then there is simply no hope to define some of the terms of $\XX(\xi)$ as the limit of smooth approximations. However, we will see that, upon subtracting suitable diverging constants, it is still possible to obtain a nontrivial limit. 
To exemplify, for $X^{\<tree12>}$, if $X^\eps$ a mollification of $X$, then $\I(|\nabla X^\eps|^2)(t)$ does not converge, but there exists a diverging sequence of $c_\eps$ such that $\mathcal I(|\nabla X^\eps|^2)(t)-c_\eps t$ indeed does (see Theorem~\ref{th:stoc} for a complete proof).  
\end{remark}

\begin{remark}
One of the main differences with the KPZ equation studied in~\cite{hairer_solving_2013} and~\cite{GP} is the stochastic term $X^{\<tree1222>}$. 
Indeed, while in the latter case this term requires a non trivial renormalization, in our it does not (see again Theorem~\ref{th:stoc}).
\end{remark}

Given $\XX\in\XR^{\varrho,r}$, $\varrho<\frac{1}{2}$, the goal of this section is to setup a fixed point argument for equation~\eqref{KPZmilduQ}. 
Now, from the definition of $\XX$ we see that the expected spatial regularity of the solution $v$ should be $\varrho+1$ and not better so that all the terms are well-defined, thanks to Proposition~\ref{prop:bony}, with the exception of $\nabla v\cdot \nabla X$, and, in particular, the resonant part of it. This difficulty can be handled in the same way as in Section~\ref{subsec:geneqrough}, namely we will exploit once more the idea of paracontrolled distributions introduced in~\cite{GIP15}.

\begin{definition}[Paracontrolled Distributions]\label{paracontrolled}
Let $\frac{2}{5}<\alpha<\frac{1}{2}$. For $Q\in C_T\CC^{\alpha+1}$, we define the space of paracontrolled distributions $\D^{\alpha}_{Q,T}$ as the set of couple of functions $(v,v')\in C_T\CC^{\alpha+1}\times C_T\CC^{\alpha}_{\R^3}$ such that
\[
v^\sharp(t)\eqdef v(t)-(v'\prec Q)(t)\in \CC^{4\alpha}
\]
for all $0\leq t\leq T$. We equip $\D^{\alpha}_{Q,T}$ with the norm 
\[
\|(v,v')\|_{\D^{\alpha}_{Q,T}} = \|v\|_1 + \|v'\|_2+\|v^\sharp\|_3
\]
where, for $\beta\in(0,3\alpha-1)$, $\gamma\in(2\alpha,\alpha+\frac{1}{2})$, $\delta\in(2\alpha-\frac{1}{2},\alpha)$, the norms $\|\cdot\|_i$, $i=1,2,3$ are defined by
\small\begin{align*}
&\|v\|_1\eqdef\|v\|_{1,x}+\|v\|_{1,T}\eqdef\sup_{t\in[0,T]}  t^{\frac{\alpha}{2}} \|v(t)\|_{3\alpha} + \sup_{0\leq s<t\leq T} s^{\frac{1+\delta-\alpha}{2}}\frac{\|\nabla v(t)-\nabla v(s)\|_{L^\infty}}{|t-s|^{\frac{\delta}{2}}}\\
&\|v'\|_2\eqdef\sup_{t\in[0,T]}t^{\frac{\gamma}{2}} \|v'(t)\|_{3\alpha-1}\qquad\qquad\qquad\qquad\quad\|v^\sharp\|_3\eqdef\sup_{t\in[0,T]}t^{\frac{\beta+1}{2}} \|v^\sharp(t)\|_{\alpha+\beta +1}
\end{align*}\normalsize
For $(u,u')\in\D^{\alpha}_{Q,T}$, we say that $u$ is {\it paracontrolled by $Q$} and we endow $\D^{\alpha}_{Q,T}$ with the metric
\[
d_{\D^{\alpha}_{Q,T}}\left((v_1,v_1'),(v_2,v_2')\right)\eqdef \|v_1-v_2\|_1 + \|v_1'-v_2'\|_2+\|v_1^\sharp-v_2^\sharp\|_3
\] 
for $(v_1,v_1'),(v_2,v_2')\in\D^{\alpha}_{Q,T}$. 
\end{definition}

At this point let $(v,v')\in\D^{\alpha}_{Q,T}$. Then, upon decomposing the product in the paraproduct and resonant part, and exploiting the paracontrolled structure of $v$, we can write $\nabla X\circ \nabla v$ as
\begin{align*}
\nabla X\circ \nabla v&=\nabla X \circ \nabla(v'\prec Q)+ \nabla X \circ \nabla v^\sharp\\
&=\nabla X \circ (\nabla v'\prec Q)+\nabla X \circ(v'\prec \nabla Q)+\nabla X \circ \nabla v^\sharp
\end{align*}
where, thanks to Bony's estimates and since $\alpha>\frac{2}{5}$, all the terms are well-defined apart from the second summand. But now thanks to the commutator lemma, Proposition~\ref{prop:comm}, $\nabla X\circ \nabla v$ equals
\begin{equation}\label{Prod}
 \underbrace{\nabla X \circ (\nabla v'\prec Q)}_{5\alpha-2}+\underbrace{v' (\nabla X \circ \nabla Q)}_{2\alpha-1} +\underbrace{ \mathscr R(v',\nabla X,\nabla Q)}_{2\alpha-1}+\underbrace{\nabla X \circ \nabla v^\sharp}_{5\alpha-2}
\end{equation}
Provided we can make sense of $\nabla Q\circ\nabla X$ through other means, the resonant term is now well posed and has spatial regularity given by $2\alpha-1$. In the next proposition, we will derive suitable estimates for the convolution of the latter with the heat kernel, which is exactly what we need in the proof of the fixed point (see Proposition~\ref{p:KPZbounds}).

\begin{proposition}\label{resonant}
Let $\frac{2}{5}<\alpha<\varrho<\frac{1}{2}$, $\XX\in\XR^{\varrho,r}$, $v\in\D^{\alpha}_{Q,T}$ and assume $\nabla Q\circ \nabla X$ is well-defined and belongs to $C_T\CC^{2\alpha-1}$. Then $\nabla X\circ \nabla v$ is well-posed and is given by the expansion in~\eqref{Prod}. Moreover, when convolved with the heat kernel, it satisfies the following estimate
\small\begin{equation}\label{res}
\Big\|\mathcal I\Big(\nabla X\circ \nabla v\Big)\Big\|_1+\Big\|\mathcal I\Big(\nabla X\circ \nabla v\Big)\Big\|_3\lesssim  T^\vartheta\|\XX\|_{\XR^{\varrho,r}} \Big(1+\|\XX\|_{\XR^{\varrho,r}}\Big) \Big(\|v'\|_2+\|v^\sharp\|_3\Big)  
\end{equation}\normalsize
where $\vartheta=\frac{1+\delta-\alpha-\gamma}{2}>0$ and $\|\cdot\|_i$, for $i=1,2,3$, are defined as in Definition~\ref{paracontrolled}.
\end{proposition}
\begin{proof} The argument above justifies the expansion we made and guarantees the well-posedness of the resonant term. In order to obtain the required bounds it is sufficient to apply Corollary \ref{cor}, Bony's estimates (Proposition \ref{prop:bony}) and the commutator Lemma (Proposition \ref{prop:comm}). Indeed, its $\|\cdot\|_{1,x}$-norm is bounded by
\small\begin{align*}
&T^{\frac{4\alpha-\gamma}{2}} \sup_s s^\frac{\gamma}{2}\|\nabla X_s \circ (\nabla v_s'\prec Q_s)\|_{5\alpha-2} +T^{\frac{\alpha+1-\gamma}{2}}\sup_s s^\frac{\gamma}{2} \|v'(s)(\nabla X \circ \nabla Q)_s\|_{2\alpha-1}\\
&+T^{\frac{\alpha+1-\gamma}{2}}\sup_s s^\frac{\gamma}{2} \|C(v',\nabla X,\nabla Q)(s)\|_{2\alpha-1} +T^\frac{\alpha-\beta}{2} \sup_s s^\frac{\beta+1}{2}\|\nabla X_s\circ \nabla v^\sharp(s)\|_{5\alpha-2}\\
&\lesssim T^\frac{\alpha-\beta}{2}\|\XX\|_{\mathcal{X}^\varrho} \Big(1+\|\XX\|_{\mathcal{X}^\varrho}\Big)\sup_s \Big(s^\frac{\gamma}{2}\|v'(s)\|_{3\alpha-1} + s^\frac{\beta+1}{2}\|v^\sharp(s)\|_{\alpha+\beta+1}\Big)
\end{align*}\normalsize
and its $\|\cdot\|_{3}$-norm by
\small\begin{align*}
&T^{\frac{4\alpha-\gamma}{2}} \sup_s s^\frac{\gamma}{2}\|\nabla X_s \circ (\nabla v_s'\prec Q_s)\|_{5\alpha-2} +T^{\frac{\alpha+1-\gamma}{2}}\sup_s s^\frac{\gamma}{2} \|v'(s)(\nabla X \circ \nabla Q)_s\|_{2\alpha-1}\\
&+T^{\frac{\alpha+1-\gamma}{2}}\sup_s s^\frac{\gamma}{2} \|C(v',\nabla X,\nabla Q)(s)\|_{2\alpha-1} +T^\frac{4\alpha-\beta-1}{2} \sup_s s^\frac{\beta+1}{2}\|\nabla X_s\circ \nabla v^\sharp(s)\|_{5\alpha-2}\\
&\lesssim T^\frac{4\alpha-\beta-1}{2}\|\XX\|_{\mathcal{X}^{\varrho,0}} \Big(1+\|\XX\|_{\mathcal{X}^{\varrho,0}}\Big)\sup_s \Big(s^\frac{\gamma}{2}\|u'(s)\|_{3\alpha-1} + s^\frac{\beta+1}{2}\|u^\sharp(s)\|_{\alpha+\beta+1}\Big)
\end{align*}\normalsize
where, in both cases, the first inequality follows by Corollary \ref{cor} part 1, and the second by Propositions \ref{prop:bony} and \ref{prop:comm}. 
As before, the $\|\cdot\|_{1,T}$-norm is less or equal to
\small\begin{align*}
&T^{\frac{1+\delta-\alpha-\gamma}{2}}\bigg( \sup_s s^\frac{\gamma}{2}\|\nabla X_s \circ (\nabla v_s'\prec Q_s)\|_{5\alpha-2} +\sup_s s^\frac{\gamma}{2} \|v'(s)(\nabla X \circ \nabla Q)_s\|_{2\alpha-1}\\
&+\sup_s s^\frac{\gamma}{2} \|C(v',\nabla X,\nabla Q)(s)\|_{2\alpha-1} + \sup_s s^\frac{\beta+1}{2}\|\nabla X_s\circ \nabla v^\sharp(s)\|_{5\alpha-2}\bigg)\\
&\lesssim T^\frac{1+\delta-\alpha-\gamma}{2}\|\XX\|_{\mathcal{X}^{\varrho,0}}\Big(1+\|\XX\|_{\mathcal{X}^{\varrho,0}}\Big)\sup_s \Big(s^\frac{\gamma}{2}\|v'(s)\|_{3\alpha-1} + s^\frac{\beta+1}{2}\|v^\sharp(s)\|_{\alpha+\beta+1}\Big)
\end{align*}\normalsize
but we apply the second part of Corollary \ref{cor} instead of the first. Since $\frac{1+\delta-\alpha-\gamma}{2}<\min\{\frac{4\alpha-\beta-1}{2},\frac{\alpha-\beta}{2}\}$ the conclusion follows.
\end{proof}

At this point we need to indentify $v'$, $Q$ and $v^\sharp$ so that we can establish a fixed point map in the space of paracontrolled distributions. To do so, let $(v,v')\in\D^{\alpha}_{Q,T}$ and notice that $v$ solves~\eqref{KPZmilduQ} if and only if $v^\sharp$ solves
\begin{align*}
v^\sharp(t) = P_t (u_0-v'\prec Q(0)) &+ \mathcal I\Big(\nabla\big(4X^{\<tree122>}+2 v) \prec \nabla X\Big) (t)-v'\prec Q(t)\\
&+2\, \mathcal I\Big(\nabla v \circ \nabla X\Big)(t)\\
&+ R^{v}(t)
\end{align*}
where $R^v$ was introduced in~\eqref{remainder1}. 
Now, we expect $v^\sharp$ to have spatial regularity greater than the one of $v$ but all the terms in the first line, not involving the initial condition have regularity $\alpha+1$ and not better. 
The point here is to take advantage of the difference and prove it is more regular than each of its summands. 
As the next Proposition shows, this is indeed the case upon choosing $v'$ and $Q$ wisely.

\begin{proposition}\label{ParaSchauder}
Let $\alpha,\,\beta,\,\gamma$ and $\delta$ be as in Definition \ref{paracontrolled}. Let $\XX\in\XR^{\varrho,r}$, for $\frac{2}{5}<\alpha<\varrho<\frac{1}{2}$, and $f$ be such that
\[
\|f\|_\star= \|f\|_{\star,T}+\|f\|_2\eqdef\|f\|_{C_{1+\delta-\alpha,T}^{3\alpha-1}L^\infty}+\|f'\|_{C_{\gamma,T}\CC^{3\alpha-1}_{\R^3}}<\infty
\]
then
\begin{equation}\label{vsharp}
\|\I(f\prec\nabla X)-f\prec\I(\nabla X)\|_3\lesssim T^\frac{4\alpha-1-\delta}{2}\|f\|_\star \|X\|_\varrho
\end{equation}
\end{proposition}
\begin{proof}
%The proof is analogous to the one of Lemma~\eqref{lemma:comm-sc}.
Let us rewrite the left hand side of (\ref{vsharp}) as
\small\begin{align}
t^\frac{\beta+1}{2}\|\mathcal I(f\prec&\nabla X)(t)-f(t)\prec\mathcal I(\nabla X)(t)\|_{\alpha+\beta+1}\notag\\
& \leq t^{\frac{\beta+1}{2}}\int_0^t\big\|P_{t-s}(f(s)-f(t))\prec\nabla X_s\big\|_{\alpha+\beta+1}\text{d}s\label{irr1}\\
&+t^{\frac{\beta+1}{2}}\int_0^t \big\|P_{t-s}(f(t)\prec \nabla X_s)\text{d}s - f(t)\prec P_{t-s}\nabla X_s\big\|_{\alpha+\beta+1}\text{d}s\label{irr2}
\end{align}\normalsize
Let us consider the two summands separately. Thanks to Proposition~\ref{prop:Schauder},~\eqref{irr1} is bounded by
\small\begin{multline*}
t^{\frac{\beta+1}{2}} \int_0^t (t-s)^{-1-\frac{\beta}{2}} \|f(s)-f(t)\|_{L^\infty}\text{d}s\|X\|_\alpha\\
\leq t^{\frac{\beta+1}{2}} \int_0^t (t-s)^{-1-\frac{\beta}{2}+\frac{3\alpha-1}{2}}s^{-\frac{1+\delta-\alpha}{2}} \text{d}s \|f\|_{\star,T}\|X\|_\alpha
\leq T^\frac{4\alpha-1-\delta}{2}\|f\|_{\star,T}\|X\|_\varrho
\end{multline*}\normalsize
where we used the fact that $\beta<3\alpha-1$ and $\delta<\alpha<\varrho$. 

For~\eqref{irr2} we apply the commutator~\eqref{commSchauder} in Proposition~\ref{prop:Schauder}, so that we obtain
\small\begin{align*}
t^\frac{\beta+1}{2}\int_0^t (t-s)^{-\frac{3-3\alpha+\beta}{2}}\text{d}s\|f(t)\|_{3\alpha-1}\|X\|_\alpha&\lesssim t^\frac{\beta+1-\gamma}{2}\int_0^t (t-s)^{-\frac{3-3\alpha+\beta}{2}}\text{d}s\|f\|_2\|X\|_\alpha\\
&\leq T^\frac{3\alpha-\gamma}{2}\|f\|_2\|X\|_\varrho
\end{align*}\normalsize
the last passage being justified by the fact that $\beta<3\alpha-1$. Since $2\alpha-\frac{1}{2}<\delta$ we obtain (\ref{vsharp}). 
\end{proof}

\noindent Proposition \ref{ParaSchauder} conveys that if we take 
\begin{align*}
Q_t \eqdef \mathcal I(\nabla X)(t),\,\qquad v' \eqdef 2\nabla v + 4 \nabla X^{\<tree122>}
\end{align*}
we should be in business, i.e. we should be able to determine a fixed point map in the space of paracontrolled distributions $\D^{\alpha}_{Q,T}$. 

So far we have put all the elements in place and we have now the tools we need in order to prove that, for a given rough distribution $\XX\in\XR^{\varrho,r}$, equation (\ref{KPZmilduQ}) has a unique local in time solution. 

\begin{proposition}\label{p:KPZbounds}
Let $\alpha$, $\varrho\in(\frac{2}{5},\frac{1}{2})$ with $\alpha<\varrho$. Let $\XX\in\XR^{\varrho,r}$ and $h_0\in\CC^\alpha$. For $(v,v')\in\D^{\alpha}_{Q,T}$, let $\CG:\D^{\alpha}_{Q,T}\to C_T\CC^{\alpha+1}\times C_T\CC^\alpha_{\R^3}$ be the map defined by $\CG(v,v')=(\tilde v,\tilde v')$,  
\[
\tilde v\eqdef P_t h_0 + 4 X_t^{\<tree1222>} + 2\, \mathcal I\Big(\nabla v \cdot \nabla X\Big) (t)+R^v(t),\quad\tilde v'\eqdef 2\nabla v + 4 \nabla X^{\<tree122>}
\]
where the product term has to be understood according to Proposition~\ref{resonant} (see the expansion~\eqref{Prod}) and $R^v$ is defined by (\ref{remainder1}). Then $\CG(v,v')\in\D^{\alpha}_{Q,T}$ and there exists $\vartheta>0$ such that
\begin{equation}\label{FixedPoint}
\|\mathscr{G}(v,v')\|_{\D^{\alpha}_{Q,T}}\lesssim \Big(1+\|\XX\|_{\XR^{\varrho,r}}+\|h_0\|_\alpha\Big)^2\Big(1+T^\vartheta\|(v,v')\|_{\D^{\alpha}_{Q,T}}\Big)^2
\end{equation}
and, for $V_1=(v_1,v_1'),\,V_2=(v_2,v_2')\in\D^{\alpha}_{Q,T}$, 
\small\begin{equation}
d_{\D^{\alpha}_{Q,T}}\big(\mathscr{G}(V_1),\mathscr{G}(V_2)\big)\lesssim T^\vartheta d_{\D^{\alpha}_{Q,T}}\left(V_1,V_2\right)\big( 1+\|V_1\|_{\D^{\alpha}_{Q,T}}+\|V_2\|_{\D^{\alpha}_{Q,T}}\big)\Big(1+\|\XX\|_{\XR^{\varrho,r}}\Big)^2
\end{equation}\normalsize
\end{proposition}

\begin{proof} As already pointed out, $\mathscr{G}(v,v')$ has indeed the algebraic structure of a distribution paracontrolled by $Q$ once we set
\begin{align*}
\tilde v' &= 2\nabla v + 4 \nabla X^{\<tree122>}\\
\tilde v^\sharp&=P_t h_0^Q + \mathcal I\Big(\tilde v'\prec \nabla X\Big) (t)-\tilde v'\prec Q(t) +2\, \mathcal I\Big(\nabla v \circ \nabla X\Big)(t)+ \widetilde{R}^v(t)
\end{align*}
In order to obtain the bound (\ref{FixedPoint}), let us separately consider each term. 
Let us begin with $\|\tilde v'\|_2$. 
\begin{equation}\label{Gammaprime}
t^{\frac{\gamma}{2}}\|\tilde v'(t)\|_{3\alpha-1}\lesssim t^{\frac{\gamma}{2}}\Big(\| v(t)\|_{3\alpha}+\|X^{\<tree122>}\|_{3\alpha}\Big)\lesssim T^{\frac{\gamma-2\alpha}{2}}\|v\|_{1,x} + t^{\frac{\gamma}{2}}\|\XX\|_{\XR^{\varrho,r}}
\end{equation}
For $\|\tilde v^\sharp\|_3$, set $I_i(t)$, $i=1,\dots,4$ to be the corresponding summand in the definition of $\tilde v^\sharp$, where $I_2$ is the difference, so that
\[
t^{\frac{\beta+1}{2}}\|\tilde v^\sharp(t)\|_{\alpha+\beta+1}\lesssim \sum_{i=1}^4  t^{\frac{\beta+1}{2}}\|I_i(t)\|_{\alpha+\beta+1}
\]
Now, as a trivial consequence of Proposition~\ref{prop:Schauder} and since, by definition, $Q_0=0$ , we have
\[
t^{\frac{\beta+1}{2}}\|I_1(t)\|_{\alpha+\beta+1}=t^{\frac{\beta+1}{2}}\|P_th_0\|_{\alpha+\beta+1}\lesssim \|h_0\|_\alpha
\]
For $I_2$, Proposition \ref{ParaSchauder} tells us that
\[
t^{\frac{\beta+1}{2}}\|I_2(t)\|_{\alpha+\beta+1}\lesssim T^\frac{4\alpha-1-\delta}{2}\big(\|\tilde v'\|_{\star,T}+\|\tilde v'\|_2\big)\|\XX\|_{\XR^{\varrho,r}}
\]
It remains to prove that $\|\tilde v'\|_{\star,T}$ can be bounded in terms of $\|v\|_{\D^{\alpha}_{Q,T}}$ ((\ref{Gammaprime}) is taking care of $\|\tilde v'\|_2$). But now
\small\begin{multline*}
\|\tilde v'\|_{\star,T}=\sup_{0\leq s<t\leq T} s^{\frac{1+\delta-\alpha}{2}}\frac{\|\tilde v'(t)-\tilde v'(s)\|_{\infty}}{|t-s|^{\frac{3\alpha-1}{2}}}\lesssim \sup_{0\leq s<t\leq T} s^{\frac{1+\delta-\alpha}{2}}\frac{\|\nabla v(t)-\nabla v(s)\|_{\infty}}{|t-s|^{\frac{3\alpha-1}{2}}}\\
+s^{\frac{1+\delta-\alpha}{2}}\|\nabla X^{\<tree122>}\|_{C^{\frac{3\alpha-1}{2}}_TL^\infty}\lesssim T^\frac{\delta+1-3\alpha}{2} \|v\|_{1,T} + \|\XX\|_{\XR^{\varrho,r}}
\end{multline*}\normalsize
$I_3$ is covered by Proposition \ref{resonant}, while for $I_4$ we have 
\small\begin{align*}
t^{\frac{\beta+1}{2}}\|I_4(t)&\|_{\alpha+\beta+1}\lesssim t^{\frac{\beta+1}{2}}\|\XX\|_{\XR^\varrho} + T^\frac{1-\alpha}{2}\sup_s s^\alpha \|\nabla v(s) \succ \nabla X_s\|_{2\alpha-1}\\
&+T^\frac{\alpha+1-\gamma}{2}\sup_s s^\frac{\gamma}{2} \| \nabla X_s^{\<tree12>} \cdot \tilde v'(s)\|_{2\alpha-1}+T^\frac{2\alpha-2\gamma+1}{2}\sup_s s^\gamma \|\tilde v'(s)\|_{3\alpha-1}^2\\
&\lesssim  \Big(1+\|\XX\|_{\XR^{\varrho,r}}\Big)^2\Big( 1+T^\frac{2\alpha-2\gamma+1}{2}\|v(s)\|_{1,x}\Big)^2
\end{align*}\normalsize
where the first inequality follows by Corollary \ref{cor} while the second by Bony's estimate (Proposition \ref{prop:bony}) and (\ref{Gammaprime}). Hence, collecting the bounds obtained so far, we conclude that  $\|\tilde v^\sharp\|_3$ satisfies (\ref{FixedPoint}).  
\newline

\noindent By analogous arguments, we proceed with $\|\tilde v\|_{1,x}$.  Corollary \ref{cor} implies
\small\begin{align*}
t^\alpha\|\tilde v(t)\|_{3\alpha}\lesssim& \|h_0\|_\alpha + t^\alpha \|\XX\|_{\XR^{\varrho,r}}+T^\frac{1-2\alpha}{2}\sup_s s^\alpha\|\nabla v(s) \prec \nabla X_s\|_{\alpha-1} \\
&+ \big\|\mathcal I\Big(\nabla v \circ \nabla X\Big) (t)\big\|_{1,x} +T^\frac{1-\alpha}{2}\sup_s s^\alpha\|\nabla v(s) \succ \nabla X_s\|_{2\alpha-1}\\
&+T^\frac{\alpha+1-\gamma}{2}\sup_s s^\frac{\gamma}{2} \|\nabla X_s^{\<tree12>} \cdot \tilde v'(s)\|_{2\alpha-1}+ T^\frac{2\alpha+1-2\gamma}{2} \sup_s s^\gamma \|\tilde v'(s)\|_{3\alpha-1}^2
\end{align*}\normalsize
while Proposition \ref{resonant} takes care of the resonant term, Proposition \ref{prop:bony} and (\ref{Gammaprime}) allow us to conclude that $\|\tilde v\|_{1,x}$ satsfies (\ref{FixedPoint}) for $\vartheta= \frac{\alpha-\beta}{2}$.
\newline

Finally, let us bound the last norm. Let $0\leq s<t\leq T$. At first, notice that a straightforward application of Proposition~\ref{prop:Schauder} gives
\[
s^\frac{1+\delta-\alpha}{2} \frac{\|P_t\nabla h_0-P_s\nabla h_0 \|_{\infty}}{|t-s|^\frac{\delta}{2}} = s^\frac{1+\delta-\alpha}{2} \frac{\|\big(P_{t-s}-Id\big)P_s\nabla h_0\|_{\infty}}{|t-s|^\frac{\delta}{2}}\lesssim \|h_0\|_\alpha
\]
then, using the fact that $P_t$ and $\nabla$ commute and the second part of Corollary~\ref{cor} we have
\small\begin{align*}
s^{\frac{1+\delta-\alpha}{2}} \frac{\|\nabla \tilde v(t)-\nabla \tilde v(s)\|_{\infty}}{|t-s|^{\frac{\delta}{2}}} &\lesssim\|h_0\|_\alpha + s^\frac{1+\delta-\alpha}{2}\|\XX\|_{\mathcal{X}^\varrho}+ T^\frac{1+\delta-\alpha-\gamma}{2} \sup_s s^\gamma \|\tilde v'(s)\|_{3\alpha-1}^2\\
&+s^\frac{1+\delta-\alpha}{2} \frac{\| \mathcal I\Big(\nabla\big(\nabla v \cdot \nabla X\big)\Big) (t)-\mathcal I\Big(\nabla\big(\nabla v \cdot \nabla X\big)\Big) (s)\|_{\infty}}{|t-s|^\frac{\delta}{2}}\\
&+  T^\frac{1+\delta-\alpha-\gamma}{2}\sup_s s^\frac{\gamma}{2} \|\nabla X_s^{\<tree12>} \cdot \tilde v'(s)\|_{2\alpha-1}
\end{align*}\normalsize
Now, Proposition \ref{resonant} deals with the resonant term, while the paraproducts can be bounded by
\small\begin{multline*}
s^\frac{1+\delta-\alpha}{2} \frac{\| \mathcal I\Big(\nabla\big(\nabla v \prec\succ \nabla X\big)\Big) (t)-\mathcal I\Big(\nabla\big(\nabla v \prec\succ \nabla X\big)\Big) (s)\|_{\infty}}{|t-s|^\frac{\delta}{2}}\\
\lesssim T^\frac{1-3\alpha+\delta}{2}\sup_s s^\frac{\alpha}{2}\|\nabla v(s) \prec \nabla X_s\|_{\alpha-1} +T^\frac{1-3\alpha+\delta}{2}\sup_s s^\frac{\alpha}{2} \|\nabla v(s) \succ \nabla X_s\|_{2\alpha-1} 
\end{multline*}\normalsize
where we used the more compact notation $f\prec\succ g\eqdef f\prec g+f\succ g$. Arguing as before we conclude that (\ref{FixedPoint}) holds true. The second bound in the statement can be obtained analogously.
\end{proof}
Summarizing what achieved so far, we have the following statement.

\begin{theorem}\label{th:flow-kpz}
Let $\frac{2}{5}<\alpha<\varrho<\frac{1}{2}$, $\eta\in\CC^{\infty}$ and let  $ \CS_{cKPZ}: \CC^2\times\R\to C([0,+\infty),\CC^{2})$ be the classical flow of the equation 
\begin{equation}\label{eq:kpz-smooth}
\partial_t h(t,x)=\frac{1}{2}\Delta h(t,x)+\frac{1}{2}|\nabla h(t,x)|^2+\eta(x)-(a+b),\quad h(0,x)=0
\end{equation}
$(t,x)\in[0,+\infty[\times\mathbb T^3$. Then there exist a lower semi-continuous time $T^\star:\XR^{\varrho,r}\times\mathbb R\to(0,+\infty]$ and a unique locally Lipschitz map $\mathcal S_{rKPZ}:\XR^{\varrho,r}\to C([0,\infty[,\CC^\alpha)$ such that $\mathcal S_{rKPZ}$ extends $\mathcal S_{cKPZ}$ in the following sense 
$$
\mathcal S_{rKPZ}\left(\XX(\eta,a,b)\right)(t)=\mathcal S_{cKPZ}(\eta,a+b)(t)
$$
for all $t\leq T^\star(\XX(\eta,a,b))$ and $(\eta,a,b)\in\mathscr C^\infty\times\mathbb R^2$.
\end{theorem}

\begin{proof}
Given the bounds in Proposition~\ref{p:KPZbounds}, the proof is completely analogous to the one of Theorem~\ref{th:Generator-r} provided in Section~\ref{sec:solv-Generator}.
\end{proof}

\begin{remark}
The fact that the equation~\eqref{eq:kpz-smooth} is globally well-posed, i.e. its solution $h$ does not explode in finite time, when $\eta$ is a smooth function is ensured by the fact that, thanks to the Cole-Hopf transform, $e^h$ is the solution of the linear equation 
$$
\partial_te^h=\frac{1}{2}\Delta e^h+e^h\eta 
$$
which is known to admit a unique global strictly positive solution when the initial condition is identically equal to $1$ (for example by Feynmann-Kac formula). 
\end{remark}
\begin{remark}\label{rem:two}
As we pointed out at the beginning of this section, we notice that in $d=2$, the space white noise belongs to $\CC^{\eta}$ for $\eta<-1$, and that an expansion of order one is sufficient in order to make sense of the equation~\eqref{eq:KPZtype}. 
Moreover in this case the map $\mathcal S_{rKPZ}$ is simply a locally Lipschitz functional of $(X,X^{\<tree12>})(\eta)$. 
\end{remark}

%%%%%%%%%%%%%%%%%%%%%%%%%%%%%%%%%%%STOCHASTIC TERMS%%%%%%%%%%%%%%%%%%%%%%%%

\subsection{Stochastic part}

Let $\xi$ be a space white noise on $\mathbb T^3$ and $\xi^\eps$ its regularization as in Theorem~\ref{th:polymer}, i.e.
\begin{equation}\label{e:Molli}
\xi^\eps(x)= \sum_{k\in\Z_0^3} m(\eps k) \hat{\xi}(k) e^{i k\cdot x}
\end{equation}
where $m$ is a smooth radial function with compact support such that $m(0)=1$ and $(\hat{\xi}(k))_k$ is a family of Gaussian random variables with covariance structure given by
\[
\E [\hat{\xi}(k_1)\hat{\xi}(k_2)]= \1_{\{k_1=-k_2\}}\,.
\]
In order to complete the study of equation~\eqref{eq:KPZtype1} we have to prove that the process $X=\I(\xi)$ can be indeed lifted to the space of rough distributions $\XR^{\varrho,r}$. To do so, we will show that, upon defining the processes $X^\eps$, $X^{\eps,\<tree12>}$, $X^{\eps,\<tree122>}$, $X^{\eps,\<tree1222>}$, $X^{\eps,\<tree124>}, \nabla Q^\eps\circ\nabla X^\eps $ according to~\eqref{def:stochterms} and Definition~\ref{def:rough distribution}, we have the following theorem.

\begin{theorem}
\label{th:stoc}
Let $\varrho<\frac{1}{2}$ and $(\Omega, \CF, \prob_\xi)$ be a probability space on which the space white noise $\xi$ is defined. Let $m$ be a smooth radial function with compact support such that $m(0)=1$ and $\xi_\eps$ be defined as in~\eqref{e:Molli}. Then, upon choosing the constants $c_\eps^{\<tree12>}$, $c_\eps^{\<tree124>}\in\R$ as
\begin{equation}\label{e:Constants}
c_\eps^{\<tree12>}=\sum_{k\ne0}\frac{|m(\eps k)|^2}{|k|^2},\quad c_\eps^{\<tree124>}=2\sum_{k_1,k_2}m(\eps k_1)^2m(\eps k_2)^2\frac{|k_1\cdot k_2|^2}{|k_{12}|^2|k_1|^4|k_2|^4}
\end{equation}
the sequence
\[
\XX^\eps_t\eqdef\left( X^\eps_t, X_t^{\eps,\<tree12>}-c_\eps^{\<tree12>}t, X_t^{\eps,\<tree122>}, X_t^{\eps,\<tree1222>}, X^{\eps,\<tree124>}_t-c_\eps^{\<tree124>}t, \nabla Q^\eps\circ\nabla X^\eps(x)\right)
\]
converges to a process $\XX=(X$, $X^{\<tree12>}$, $X^{\<tree122>}$, $X^{\<tree1222>}$, $X^{\<tree124>}, \nabla Q\circ\nabla X )\in \CH^\varrho$ in $L^p(\Omega, \CH^\varrho)$ for every $p>1$. The limiting process $\XX$ is independent of the choice of the mollifier and of the sequence of constants $c_\eps^{\<tree12>}$, $c_\eps^{\<tree124>}$. 
 
Moreover, replacing $\xi^\eps$ with $\delta \xi^\eps$ for $\delta>0$, the corresponding renormalizing constants are such that $c_{\eps,\delta}^{\<tree12>}=\delta^2c_\eps^{\<tree12>}$, $c_{\eps,\delta}^{\<tree124>}=\delta^4c_\eps^{\<tree124>}$. 
\end{theorem}

\begin{remark}\label{rem:Constants}
The choice of the constants made in~\eqref{e:Constants} is not unique. Clearly, being them diverging, adding any real number would not prevent the sequence $\XX^\eps$ from converging to $\XX$. 
What instead is unique, is their behaviour as $\eps$ goes to $0$ and it is possible to prove that they asymptotically satisfy
\[
c_\eps^{\<tree12>}\sim\eps^{-1},\qquad c_\eps^{\<tree124>}=O((\log(\eps))^2)\,.
\]
\end{remark}

\begin{proof} The proof of results of this type makes always use of the same tools (see~\cite{GIP15,CC13,GP}) and follows a, by now, standard procedure. For $\tau\in\{\cdot,\<tree12>,\<tree122>,\<tree1222>,\<tree124>\}$, at first one obtains $L^2$ bounds of the different Wiener-chaos components of $X^{\eps,\tau}_{s,t}-X^{\eps',\tau}_{s,t}$, where $X_{s,t}\eqdef X_t-X_s$, and then the conclusion is attained thanks to Besov embedding (Proposition~\ref{proposition:Bes-emb}) and Garsia-Rodemich-Rumsey lemma (see~\cite{GRR}). 
%For $\tau\in\{\<tree1>,\<tree12>,\<tree122>,\<tree1222>,\<tree124>\}$, very similar estimates were already showed in~\cite[Section 9]{GP} so we refrain from reproducing them here. Nevertheless, the interested reader can consult~\cite{CC15} for a more complete version.

In the following paragraphs, we will prove only the $L^2$ bounds for the time increment %of $X^{\eps,\,\<tree12>}$ and $\nabla Q^\eps\circ\nabla X^\eps(x)$ 
as well as the evaluation of the diverging constants necessary to renormalize the KPZ-type equation presented above. 

\end{proof}

\begin{notation*}
Since we will run into long formulas, we reckon convenient to introduce some notations we will exploit in the rest of the chapter. As already pointed out, the time increment of a process $X$ will be abbreviated as $X_{s,t}\eqdef X_t-X_s$ and for a function of time $f(\cdot)$ we will write $f(s,t)\eqdef f(t)-f(s)$. For vectors $k_1,k_2\in\R^d$, we will indicate by $k_{12}\eqdef k_1+k_2$, by $k_1\cdot k_2$ their scalar product and $k_1\,k_2^\star\in\R^{d\times d}$ the matrix generated by the column by vector product.
\end{notation*}

\subsection*{Definition of $X$}

By definition
\small\begin{align*}
X_t^\eps(x)&= \I(\xi^\eps)(t,x) = \int_0^t P_{t-s} \xi^\eps (x)\dd s\\
&= \sum_{k\in\Z_0^3} F_t^\eps(k)\hat{\xi}(k) e^{i k\cdot x}=\sum_{k\in\Z_0^3}m(\eps k)\frac{1-e^{-\frac{1}{2}|k|^2t}}{|k|^2} \hat{\xi}(k) e^{i k\cdot x}
\end{align*}\normalsize
%where
%\[
%F_t^\eps(k) = f(\eps k)\int_0^t e^{-\frac{1}{2}|k|^2(t-s)}\dd s = \qquad\qquad k\in\Z_0^3,\,\,t\geq0
%\]
The well-posedness of this term is straightforward and has already been shown in a slightly different context, for example, in~\cite{GIP15,CC13}.

\subsection*{Definition of $X^{\<tree12>}$}

As before we have
\small\[
X_t^{\eps,\<tree12>}(x)= \I(|\nabla X^\eps|^2)(t,x) = -\sum_{k\in\Z^3} \sum_{\substack{k_1,k_2\in\Z_0^3\\k_{12}=k}} F^{\eps,\<tree12>}_t(k,k_1,k_2) (k_1\cdot k_2)\hat{\xi}(k_1)\hat{\xi}(k_2) e_k
\]\normalsize
where $F^{\eps,\<tree12>}_t(k,k_1,k_2) = \int_0^t e^{-\frac{1}{2}|k|^2(t-s)}F_s^\eps(k_1)F_s^\eps(k_2)\dd s$, for $k_1$, $k_2\in\Z_0^3$ and $t\geq 0$.

\subsubsection*{$0^{th}$-chaos}

The $0$-th chaos component of $X^{\eps,\<tree12>}$ is given by
\small\begin{align*}
c_\eps^{\<tree12>}(t)&=\E[X_t^{\eps,\<tree12>}(x)]\\
&= -\sum_{\substack{ k_1,k_2\in\Z_0^3\\k_{12}=k}} F^{\eps,\<tree12>}_t(k,k_1,k_2) (k_1\cdot k_2)\E[\hat{\xi}(k_1)\hat{\xi}(k_2)] e_{k_{12}}=\sum_{k\in\Z_0^3} F^{\eps,\<tree12>}_t(0,k,k) |k|^2%f(\eps k)^2 \frac{\int_0^t (1-e^{-\frac{1}{2}|k|^2s})^2\dd s}{|k|^2}
\end{align*}\normalsize
where the first equality follows by Wick's theorem. Expanding the kernel, we obtain
\small\begin{align*}
c_\eps^{\<tree12>}(t)= t\sum_{k\in\Z_0^3}\frac{m(\eps k)^2}{|k|^2}-4\sum_{k\in\Z_0^3}m(\eps k)^2\frac{1-e^{-\frac{1}{2}|k|^2 t}}{|k|^4} + \sum_{k\in\Z_0^3}m(\eps k)^2\frac{1-e^{-|k|^2 t}}{|k|^4}
\end{align*}\normalsize
and the latter two summands converge for every $t\geq0$. This means that, in order to renormalize $X_t^{\eps,\<tree12>}$ it is enough to subtract the first term $c_\eps^{\<tree12>}\,t$, where $c_\eps^{\<tree12>}=\sum_{k\in\Z_0^3}\frac{m(\eps k)^2}{|k|^2}$. 

\subsubsection*{$2^{nd}$-chaos}

Thanks again to Wick's theorem, the second moment of the second chaos component of $X^{\eps,\<tree12>}$ is
\small\begin{align*}
\E\Big|\Delta_q \Big(X_{s,t}^{\eps,\<tree12>}-c_\eps^{\<tree12>}(s,t)\Big)\Big|^2=2\sum_{k\in\Z_0^3} \varrho_q(k)^2\sum_{k_{12}=k} |F^{\eps,\<tree12>}_{s,t}(k,k_1,k_2)|^2|k_1\cdot k_2|^2
\end{align*}\normalsize
where we recall that $\varrho_q(\cdot)\eqdef\varrho(2^{-q}\cdot)$. Now, the modulus of the kernel $F^{\eps,\<tree12>}$ can be bounded by
\small\begin{align*}
\int_s^t e^{-\frac{1}{2}|k|^2(t-r)} &F_r^\eps(k_1)F_r^\eps(k_2)\dd r +|1-e^{-\frac{1}{2}|k|^2(t-s)}|\int_0^s e^{-\frac{1}{2}|k|^2(s-r)} F_r^\eps(k_1)F_r^\eps(k_2)\dd r\\
&\lesssim m(\eps k_1)m(\eps k_2)\frac{|1-e^{-\frac{1}{2}|k|^2(t-s)}|}{|k|^2|k_1|^2|k_2|^2}\lesssim m(\eps k_1)m(\eps k_2)\frac{|t-s|^\vartheta}{|k|^{2-2\vartheta}|k_1|^2|k_2|^2}
\end{align*}\normalsize
where in the last passage we used geometric interpolation for $\vartheta\in(0,1)$. Therefore
\small\begin{align*}
\E\Big|\Delta_q \Big(X_{s,t}^{\eps,\<tree12>}-c_\eps^{\<tree12>}(s,t)\Big)\Big|^2&\lesssim |t-s|^{2\vartheta} \sum_{k\in\Z_0^3} \frac{\varrho_q(k)^2}{|k|^{4-4\vartheta}}\sum_{k_{12}=k} \frac{|k_1\cdot k_2|^2}{|k_1|^4|k_2|^4}\\
&\lesssim |t-s|^{2\vartheta} 2^{-2q(2-2\vartheta)}\sum_{\substack{k_{12}=k\\|k|\sim2^q}}\frac{1}{|k_1|^2|k_2|^2}\,.
\end{align*}\normalsize
Now, the latter sum is bounded by
\small\[
\sum_{\substack{k\in\Z_0^3\\|k|\sim2^q}}\frac{1}{|k|^{1-\delta}}\sum_{k_1:|k_1|\leq|k_2|}\frac{1}{|k_1|^{3+\delta}}\lesssim2^{-2q(-\frac{3}{2}+\frac{1}{2}-\frac{\delta}{2})}\sum_{k_1\in\Z_0^3}\frac{1}{|k_1|^{3+\delta}}
\]\normalsize
and the last sum is finite. 

\subsection*{Definition of $X^{\<tree122>}$}

Analogously we proceed with the next term
\small\[
X^{\eps,\<tree122>}_t(x)= \sum_{\substack{k\in\Z^3\\k_{123}=k}} F_t^{\eps,\<tree122>} (k,k_{12},k_1,k_2,k_3)(k_{12}\cdot k_3)(k_1\cdot k_2)\hat{\xi}(k_1)\hat{\xi}(k_2)\hat{\xi}(k_3) e_k
\]\normalsize
where $k_{123}\eqdef k_1+k_2+k_3$ and $F_t^{\eps,\<tree122>} (k,k_{12},k_1,k_2,k_3)= \int_0^t e^{-\frac{1}{2}|k|^2(t-s)}F_s^{\eps,\<tree12>}(k_{12},$ $k_1,k_2)F_s^\eps(k_3)\dd s$, for $k_1$,  $k_2$,  $k_3\in\Z_0^3$ and $t\geq 0$. 

\subsubsection*{$1^{st}$ chaos}

Let $\pi_i$ be the projection onto the $i$-th Wiener chaos. Notice that $\pi_1\Big(X^{\eps,\<tree122>}\Big)(t,x)$ is given by
\small\[
\sum_{\substack{k\in\Z^3\\k_{123}=k}} F_t^{\eps,\<tree122>} (k,k_{12},k_1,k_2,k_3)(k_{12}\cdot k_3)(k_1\cdot k_2)\pi_1\Big(\hat{\xi}(k_1)\hat{\xi}(k_2)\hat{\xi}(k_3) \Big)e_k
\]\normalsize
and
\small\begin{align*}
\pi_1\Big(\hat{\xi}(k_1)\hat{\xi}(k_2)\hat{\xi}(k_3) \Big) = \hat{\xi}(k_1) \1_{\{k_2=-k_3\}}+\hat{\xi}(k_2) \1_{\{k_1=-k_3\}}+\hat{\xi}(k_3) \1_{\{k_1=-k_2\}}
\end{align*}\normalsize
Now, the first two summands give the same contribution since the role of $k_1$ and $k_2$ is completely symmetric while the last summand does not give any, since the sum depends linearly on $k_{12}$. Therefore we get
\small\[
\pi_1\Big(X^{\eps,\<tree122>}\Big)(t,x)= -2\sum_{k_1\in\Z_0^3} \sum_{k_2\in\Z_0^3}F_t^{\eps,\<tree122>} (k_1,k_{12},k_1,k_2,-k_2)(k_{12}\cdot k_2)(k_1\cdot k_2)\hat{\xi}(k_1) e_{k_1}
\]\normalsize
hence
\small\begin{align*}
\E\Big|\Delta_q \pi_1\big(X^{\eps,\<tree122>}\big)_{s,t}\Big|^2=4\sum_{k_1}\varrho_q(k_1)^2\sum_{k_2} \left(F_{s,t}^{\eps,\<tree122>} (k_1,k_{12},k_1,k_2,k_1)\right)^2|k_{12}\cdot k_2|^2|k_1\cdot k_2|^2
\end{align*}\normalsize
If we bound the kernel $F^{\eps,\<tree122>} (k_1,k_2)$ too boldly we would not obtain the hoped result, therefore we have to proceed more subtly, following the scheme exploited by Gubinelli and Perkowski in~\cite{GP}. We can always write
\small\begin{align*}
F_{s,t}^{\eps,\<tree122>} (k_1,k_{12},&k_1,k_2,k_2) = \int_s^t e^{-\frac{1}{2}|k_1|^2(t-r)}F_r^{\eps,\<tree12>}(k_{12},k_1,k_2)F_r^\eps(k_2)\dd r \\
&-\big(1-e^{-\frac{1}{2}|k_1|^2(t-s)}\big)\int_0^s e^{-\frac{1}{2}|k_1|^2(s-r)}F_r^{\eps,\<tree12>}(k_{12},k_1,k_2)F_r^\eps(k_2)\dd r
\end{align*}\normalsize
Let us focus on the first summand, the computation for the second being identical. Then we have
\small\begin{multline}\label{kernel_first_chaos}
\sum_{k_2} \int_s^t e^{-\frac{1}{2}|k_1|^2(t-r)}\int_0^r e^{-\frac{1}{2}|k_{12}|^2(r-\bar{r})}F_{\bar{r}}^\eps(k_2)F_{\bar{r}}^\eps(k_1)\dd\bar{r} F_r^\eps(k_2)\dd r (k_{12}\cdot k_2)(k_1\cdot k_2)\\
=\int_0^t\dd r\int_0^r \dd \bar{r} e^{-\frac{1}{2}|k_1|^2(t-r)} F_{\bar{r}}^\eps(k_1) k_1\cdot \Big(\sum_{k_2}e^{-\frac{1}{2}|k_{12}|^2(r-\bar{r})}F_{\bar{r}}^\eps(k_2)F_r^\eps(k_2) (k_{12}\cdot k_2)\,k_2\Big)
\end{multline}\normalsize
where $F^\eps$ was introduced in the definition of $X^\eps$ above. Now, let us have a closer look at the quantity enclosed in the parenthesis. For $k_1=0$ the sum is equal to 0 since the summand is odd in $k_2$, i.e.
\small\[
\sum_{k_2}e^{-\frac{1}{2}|k_{2}|^2(r-\bar{r})}F_{\bar{r}}^\eps(k_2)F_r^\eps(k_2) (k_{2}\cdot k_2)\,k_2=0
\]\normalsize
therefore it equals
\small\begin{align*}
&\sum_{k_2}\Big(e^{-\frac{1}{2}|k_{12}|^2(r-\bar{r})} (k_{12}\cdot k_2) -e^{-\frac{1}{2}|k_{2}|^2(r-\bar{r})} (k_{2}\cdot k_2)\Big) F_{\bar{r}}^\eps(k_2)F_r^\eps(k_2) \,k_2\\
&=\int_0^1 \dd \tau \sum_{k_2} e^{-\frac{1}{2}|\tau k_1 + k_2|^2(r-\bar{r})} \big( k_1\cdot k_2 - 2 (k_{12}\cdot k_1)(r-\bar{r})\big)F_{\bar{r}}^\eps(k_2)F_r^\eps(k_2) \,k_2
\end{align*}\normalsize
where in the last passage we simply applied Taylor formula to $G(k)=e^{-\frac{1}{2}|k+k_2|^2(r-\bar{r})} (k+ k_2)\cdot k_2$. Getting back to (\ref{kernel_first_chaos}), its modulo is bounded by
\small\begin{align}
\int_0^1\dd\tau&\sum_{k_2} \int_s^t e^{-|k_1|^2(t-r)}\int_0^r e^{-|\tau k_1 + k_2|^2(r-\bar{r})} \big| k_1\cdot k_2 - 2 (k_{12}\cdot k_1)(r-\bar{r})\times\notag\\
&\times\big|F_{\bar{r}}^\eps(k_1)F_{\bar{r}}^\eps(k_2)\dd \bar{r} F_r^\eps(k_2) \dd r|k_1\cdot k_2|\lesssim\int_0^1\dd\tau\frac{1}{|k_1|}\sum_{k_2}\frac{1}{|k_2|^3}\int_s^t e^{-|k_1|^2(t-r)}\times \notag\\
&\times\int_0^r e^{-|\tau k_1 + k_2|^2(r-\bar{r})}\big| k_1\cdot k_2 - 2 (k_{12}\cdot k_1)(r-\bar{r})\big|\dd\bar{r}\dd r\lesssim\int_0^1\dd\tau\frac{1}{|k_1|}\sum_{k_2}\frac{1}{|k_2|^3}\times\notag\\
&\times\int_s^t e^{-|k_1|^2(t-r)}|k_1|\int_0^r e^{-|\tau k_1 + k_2|^2(r-\bar{r})} \big(|k_2| + |k_{12}|)\big)\dd\bar{r}\dd r\notag\\
&\lesssim \frac{1}{|k_1|}\int_s^t e^{-|k_1|^2(t-r)}\dd r\,|k_1|\int_0^1 \sum_{k_2}\frac{1}{|k_2|^2|\tau k_1 + k_2|^2}\dd \tau \label{latter}
\end{align}\normalsize
At this point notice that the sum in the last term can be bounded in two ways: through Cauchy-Schwarz inequality one obtain directly that it is uniformly bounded in $k_1$; alternatively, one has
\small\[
\sum_{k_2}\frac{1}{|k_2|^2|\tau k_1 + k_2|^2}\lesssim \int\frac{1}{|k_2|^2|\tau k_1 + k_2|^2} \dd k_2\lesssim \frac{1}{\tau |k_1|}\int \frac{1}{|y|^2\big|\frac{ k_1}{|k_1|} + y\big|^2}\dd y
\]\normalsize
where the integral is taken over a subset of $\R^3$ in which the integrand is well defined and the last one converges thanks to Cauchy-Schwarz once more. Interpolating these two bounds we get that (\ref{latter}) is less than
\small\[
\frac{1}{|k_1|}\int_s^t e^{-|k_1|^2(t-r)}\dd r |k_1|^\delta \int_0^1 \frac{1}{\tau^{1-\delta}}\dd \tau\lesssim \frac{1}{|k_1|^{3-\delta}}\big(1- e^{-|k_1|^2(t-s)}\big)\lesssim \frac{(t-s)^\vartheta}{|k_1|^{3-\delta-2\vartheta}}
\]\normalsize
Therefore we can conclude that
\small\[
\E\Big|\Delta_q \pi_1\big(X^{\eps,\<tree122>}\big)_{s,t}\Big|^2\lesssim (t-s)^{2\vartheta}\sum_{k_1} \frac{\varrho_q(k_1)}{|k_1|^{6-2\delta-4\vartheta}}\lesssim (t-s)^{2\vartheta} 2^{-2q\big(\frac{3}{2} - \delta-2\vartheta\big)} 
\]\normalsize

\subsubsection*{$3^{rd}$-chaos}

By Wiener chaos decomposition and Cauchy-Schwarz inequality we have
\small\begin{align*}
\E\Big|\Delta_q\Big(X_{s,t}^{\eps,\<tree122>} -\pi_1\big(X^{\eps,\<tree122>}\big)_{s,t}\Big)\Big|^2
\lesssim\sum_{\substack{k\in\Z^3\\k_{123}=k}} \varrho_q(k)^2 F_{s,t}^{\eps,\<tree122>} (k,k_{12},k_1,k_2,k_3)^2|k_{12}\cdot k_3|^2|k_1\cdot k_2|^2
\end{align*}\normalsize
Proceeding as in the previous cases, the kernel $F^{\eps,\<tree122>}$ can be bounded as follows
\small\begin{align*}
|F_{s,t}^{\eps,\<tree122>}(k,k_{12},k_1,k_2,k_3)|\lesssim \frac{(t-s)^{\vartheta}}{|k|^{2-2\vartheta}|k_{12}|^2 |k_1|^2|k_2|^2|k_3|^2}
\end{align*}\normalsize
Therefore
\small\[
\E\Big|\Delta_q\Big(X_{s,t}^{\eps,\<tree122>} -\pi_1\big(X^{\eps,\<tree122>}\big)_{s,t}\Big)\Big|^2
\lesssim (t-s)^{2\vartheta} 2^{-2q(2-2\vartheta)}\sum_{\substack{k_{123}=k\\|k|\sim2^q}} \frac{1}{|k_{12}|^2 |k_1|^2|k_2|^2|k_3|^2}
\]\normalsize
Let us focus on the latter sum
\small\begin{align*}
&\sum_{\substack{k_{123}=k\\|k|\sim2^q}} \frac{1}{|k_{12}|^2 |k_1|^2|k_2|^2|k_3|^2}=\sum_{\substack{k_{12}+k_3=k\\|k|\sim2^q}} \frac{1}{|k_{12}|^2 |k_3|^2}\sum_{k_1+k_2=k_{12}}\frac{1}{|k_1|^2|k_2|^2} \\
&=2 \sum_{\substack{k_{12}+k_3=k\\|k|\sim2^q}} \frac{1}{|k_{12}|^2 |k_3|^2}\sum_{\substack{k_1:|k_1|\leq|k_2|\\k_1+k_2=k_{12}}}\frac{1}{|k_1|^2|k_2|^2}\lesssim \sum_{\substack{k_{12}+k_3=k\\|k|\sim2^q}} \frac{1}{|k_{12}|^{3-\delta} |k_3|^2}\sum_{k_1}\frac{1}{|k_1|^{3+\delta}}\\
&\lesssim \sum_{k:|k|\sim 2^q}\Big(\sum_{\substack{k_3:|k_3|\leq|k_{12}|\\k_{123}=k}} + \sum_{\substack{k_3:|k_3|\geq|k_{12}|\\k_{123}=k}}\Big)\frac{1}{|k_{12}|^{3-\delta} |k_3|^2}\\
&\lesssim \sum_{k:|k|\sim2^q}\frac{1}{|k|^{2-2\delta}}\sum_{\bar{k}}\frac{1}{|\bar{k}|^{3+\delta}}\lesssim 2^{q(1+2\delta)}
\end{align*}\normalsize
And we conclude
\small\[
\E\Big|\Delta_q\Big(X_{s,t}^{\eps,\<tree122>} -\pi_1\big(X^{\eps,\<tree122>}\big)_{s,t}\Big)\Big|^2\lesssim(t-s)^{2\vartheta} 2^{-2q(2-\frac{1}{2}-\delta-2\vartheta)}=(t-s)^{2\vartheta} 2^{-2q\big(\frac{3}{2}-\delta-2\vartheta\big)}
\]\normalsize

\subsection*{Definition of $X^{\<tree1222>}$}

By definition, the term $X^{\eps,\<tree1222>}$ is given by
\small\begin{align*}
X_t^{\eps,\<tree1222>}(x)=-\sum_{\substack{k\in\Z\\k_{1234}=k}} &F_t^{\eps,\<tree1222>}(k,k_{123},k_{12},k_1,k_2,k_3,k_4)\times\\
&\times (k_{123}\cdot k_4)(k_{12}\cdot k_3)(k_1\cdot k_2)\hat{\xi}(k_1)\hat{\xi}(k_2)\hat{\xi}(k_3)\hat{\xi}(k_4) e_k
\end{align*}\normalsize
where, as before, $k_{1234}\eqdef k_1+k_2+k_3+k_4$, $F_t^{\eps,\<tree1222>}(k,k_{123},k_{12},k_1,k_2,k_3,k_4) = \int_0^t e^{-|k|^2(t-s)}F_s^{\eps,\<tree122>}(k_{123},k_{12},k_1,k_2,k_3)F_s^\eps(k_3)\dd s$, for $k_1$, $k_2$, $k_3$, $k_4\in\Z_0^3$ and $t\geq 0$, which in turn admits the following bound
\small\begin{equation}\label{kernel2}
|F_{s,t}^{\eps,\<tree1222>}(k,k_{123},k_{12},k_1,k_2,k_3,k_4)|\lesssim \frac{(t-s)^\vartheta}{|k|^{2-2\vartheta}|k_{123}|^2|k_{123}|^2|k_{12}|^2|k_1|^2|k_2|^2|k_3|^2|k_4|^2}
\end{equation}\normalsize

\subsubsection*{$0^{th}$-chaos}

By Wick's theorem we know that the expectation of the product of $\hat{\xi}(k_1)\,,\hat{\xi}(k_2)\,,\hat{\xi}(k_3)$ and $\hat{\xi}(k_4)$ is given by
\small\[
\1_{\{k_1=-k_2\}}\1_{\{k_3=-k_4\}}+\1_{\{k_1=-k_3\}}\1_{\{k_2=-k_4\}}+\1_{\{k_1=-k_4\}} \1_{\{k_2=-k_3\}}
\]\normalsize

\noindent therefore we have
\small\[
c_\eps^{\<tree1222>}(t) = \E\Big[ X_t^{\eps,\<tree1222>}(x)\Big]=-2\sum_{k_1,k_2\in\Z_0^3} F_t^{\eps,\<tree1222>}(0,k_1,k_{12},k_1,k_2,k_2,k_1) |k_1|^2(k_{12}\cdot k_2)(k_1\cdot k_2)
\]\normalsize
indeed, the first summand does not give any contribution because of the linear dependence on $k_{12}$, while the other two are the same since the role of $k_1$ and $k_2$ is perfectly symmetric. 
It is not clear yet that this term converges but we prefer to postpone the proof of this fact to Section~\ref{sec:const}

\subsubsection*{$2^{nd}$-chaos}

Since the second chaos component of $\hat{\xi}(k_1)\hat{\xi}(k_2)\hat{\xi}(k_3)\hat{\xi}(k_4)$ is
\small\begin{align*}
\1_{\{k_1=-k_2\}}& \hat{\xi}(k_3)\hat{\xi}(k_4)+\1_{\{k_1=-k_3\}} \hat{\xi}(k_2)\hat{\xi}(k_4)+\1_{\{k_2=-k_3\}} \hat{\xi}(k_1)\hat{\xi}(k_4)\\
&+\1_{\{k_1=-k_4\}} \hat{\xi}(k_2)\hat{\xi}(k_3)+\1_{\{k_2=-k_4\}} \hat{\xi}(k_1)\hat{\xi}(k_3)+\1_{\{k_3=-k_4\}} \hat{\xi}(k_1)\hat{\xi}(k_2)
\end{align*}\normalsize
we can decompose $\pi_2\Big(X^{\eps,\<tree1222>}\Big)(t,x)$ as the sum of three terms, $\pi_2^1(t,x)$, $\pi_2^2(t,x)$ and $\pi_2^3(t,x)$,
where the first is the term coming from the second and third summand, the second the one coming from the fourth and fifth and the first from the last. The reason why different summands give the same contribution (or no contribution at all) is the one we spelled out before, i.e. the symmetric role of $k_1$ and $k_2$ and the linear dependence of the sum on $k_{12}$. Let us separately consider each of the $\pi_2^i$, $i=1,2,3$. $\pi_2^1(t,x)$ equals
\small\[
2\sum_{k_1,k_2,k_4\in\Z_0^3} F_t^{\eps,\<tree1222>}(k_{24},k_2,k_{12},k_1,k_2,k_1,k_4) (k_2\cdot k_4)(k_{12}\cdot k_1)(k_1\cdot k_2)\hat{\xi}(k_2)\hat{\xi}(k_4) e^{i k_{24}\cdot x}
\]\normalsize
Then, using (\ref{kernel2}) and Cauchy-Schwarz inequality, we obtain
\small\[
\E|\Delta_q \pi_2^1(s,t)|^2\lesssim (t-s)^{2\vartheta}\sum_{k_2+k_4=k_{24}}\varrho_q(k_{24})^2\frac{1}{|k_{24}|^{4-4\vartheta}|k_2|^4 |k_4|^2} \sum_{k_1\in\Z_0^3}\frac{1}{|k_1|^4 |k_{12}|^2}
\]\normalsize
Now, the sum over $k_1$ is finite and can be bounded uniformly over $k_2$ using Cauchy-Schwarz inequality. Hence,
\small\begin{align*}
\E|\Delta_q \pi_2^1|^2&\lesssim (t-s)^{2\vartheta}\sum_{k_2+k_4=k_{24}}\varrho_q(k_{24})^2\frac{1}{|k_{24}|^{4-4\vartheta}|k_2|^4 |k_4|^2}\\
&\lesssim (t-s)^{2\vartheta} 2^{-2q(2-2\vartheta)}\sum_{|k_{24}|\sim 2^q}\sum_{k_2+k_4=k_{24}}\frac{1}{|k_2|^4 |k_4|^2}
\end{align*}\normalsize
and for the latter sum we have
\small\begin{align*}
\sum_{|k_{24}|\sim 2^q} &\left(\sum_{\substack{k_2: |k_2|\leq|k_4|\\k_2+k_4=k_{24}}} + \sum_{\substack{k_2: |k_4|\leq|k_2|\\k_2+k_4=k_{24}}}\right)\frac{1}{|k_2|^4 |k_4|^2}\\
&\lesssim \sum_{|k_{24}|\sim 2^q} \frac{1}{|k_{24}|^2}\sum_{k_2}\frac{1}{|k_2|^4}+\sum_{|k_{24}|\sim 2^q} \frac{1}{|k_{24}|^{3-\delta}}\sum_{k_2}\frac{1}{|k_2|^{3+\delta}}\lesssim2^{\frac{q}{2}}
\end{align*}\normalsize
and this gives the bound we expected. Let us proceed with $\pi_2^2(t,x)$ which is
\small\[
2\sum_{k_1,k_2,k_3\in\Z_0^3} F_t^{\eps,\<tree1222>}(k_{23},k_{123},k_{12},k_1,k_2,k_3,k_1) (k_{123}\cdot k_1)(k_{12}\cdot k_3)(k_1\cdot k_2)\hat{\xi}(k_2)\hat{\xi}(k_3) e^{i k_{23}\cdot x}
\]\normalsize
As before,
\small\begin{align*}
\E|\Delta_q \pi_2^2(s,t)|^2&\lesssim  (t-s)^{2\vartheta}\sum_{k_2+k_3=k_{23}}\varrho_q(k_{23})^2\frac{1}{|k_{23}|^{4-4\vartheta} |k_2|^2 |k_3|^2}\sum_{k_1}\frac{1}{|k_{123}|^2|k_{12}|^2|{k_1}|^4}\\
&\lesssim (t-s)^{2\vartheta} 2^{-2q(2-2\vartheta)}\sum_{\substack{k_2+k_3=k_{23}\\|k_{23}|\sim 2^q}}\frac{1}{ |k_2|^2 |k_3|^2}\sum_{k_1}\frac{1}{|k_{123}|^2|k_{12}|^2|{k_1}|^4}
\end{align*}\normalsize
Now, notice that for the inner sum we have
\small\begin{multline*}
\left(\sum_{\substack{k_1:|k_1|\leq|k_{12}|\\k_{12}-k_1=k_2}} + \sum_{\substack{k_1:|k_{12}|\leq|k_1|\\k_{12}-k_1=k_2}}\right)\frac{1}{|k_{123}|^2|k_{12}|^2|{k_1}|^4}\\
\lesssim \frac{1}{|k_2|^2}\sum_{k_1}\frac{1}{|k_{123}|^2|{k_1}|^4} + \frac{1}{|k_2|^4}\sum_{k_1}\frac{1}{|k_{123}|^2|k_{12}|^2}\lesssim \frac{1}{|k_2|^2}
\end{multline*}\normalsize
since both the sums are bounded uniformly over $k_2$ and $k_3$. Consequently, for the outer sum
\small\begin{multline*}
\sum_{|k_{23}|\sim 2^q}\Big(\sum_{\substack{|k_2|\leq|k_3|\\k_2+k_3=k_{23}}}+\sum_{\substack{|k_3|\leq|k_2|\\k_2+k_3=k_{23}}}\Big)\frac{1}{ |k_2|^4 |k_3|^2}\\
\lesssim\sum_{|k_{23}|\sim 2^q}\frac{1}{|k_{23}|^2}\sum_{k_2}\frac{1}{|k_2|^4} +\sum_{|k_{23}|\sim 2^q}\frac{1}{|k_{23}|^{3-\delta}}\sum_{k_2}\frac{1}{|k_2|^{3+\delta}} \lesssim 2^q
\end{multline*}\normalsize
Finally, $\pi_2^3(t,x)$ is given by
\small\[
\sum_{k_1,k_2,k_3\in\Z_0^3} F_t^{\eps,\<tree1222>}(k_{12},k_{123},k_{12},k_1,k_2,k_3,k_3) (k_{123}\cdot k_3)(k_{12}\cdot k_3)(k_1\cdot k_2)\hat{\xi}(k_1)\hat{\xi}(k_2) e^{i k_{12}\cdot x}
\]\normalsize
and $\E|\Delta_q \pi_2^3(s,t)|^2$ is bounded by
\small\begin{align*}
&(t-s)^{2\vartheta}\sum_{k_1+k_2=k_{12}}\varrho_q(k_{12})^2\frac{1}{|k_{12}|^{6-4\vartheta}|k_1|^2|k_2|^2}\sum_{k_3}\frac{1}{|k_{123}|^2|k_3|^4}\\
&\lesssim (t-s)^{2\vartheta} 2^{-2q(3-2\vartheta)}\sum_{\substack{k_1+k_2=k_{12}\\|k_{12}|\sim2^q}}\frac{1}{|k_1|^2|k_2|^2}\lesssim (t-s)^{2\vartheta} 2^{-2q(3-\frac{3}{2}-2\vartheta)}\\
&\lesssim (t-s)^{2\vartheta} 2^{-2q(\frac{3}{2}-2\vartheta)}
\end{align*}\normalsize
and this concludes the analysis of the $2^{nd}$-chaos component of $X^{\<tree1222>}$.

\subsubsection*{$4^{th}$-chaos}

As for the $3^{rd}$-chaos component of $X^{\<tree122>}$, using the Wiener chaos decomposition of $X^{\<tree1222>}$ and Cauchy-Schwarz inequality we have
\small\begin{align*}
\E\Big| \Delta_q \Big( X_{s,t}^{\eps,\<tree1222>}-&\pi_2\big(X^{\eps,\<tree1222>}\big)_{s,t} - c_\eps^{\<tree1222>}(s,t)\Big)\Big|^2\\
&\lesssim (t-s)^{2\vartheta}\sum_{\substack{k\in\Z_0^3\\k_{1234}=k}} \frac{\varrho_q(k)^2}{|k|^{4-4\vartheta}|k_{123}|^2|k_{12}|^2|k_1|^2|k_2|^2|k_3|^2|k_4|^2}\\
&\lesssim (t-s)^{2\vartheta}2^{-2q(2-2\vartheta)}\sum_{\substack{k_{1234}=k\\|k|\sim2^q}} \frac{1}{|k_{123}|^2|k_{12}|^2|k_1|^2|k_2|^2|k_3|^2|k_4|^2}
\end{align*}\normalsize
The latter sum is bounded by

\small\begin{align*}
\sum_{\substack{k_{1234}=k\\|k|\sim2^q}}&\frac{1}{|k_{123}|^2|k_{12}|^2|k_3|^2|k_4|^2}\sum_{\substack{|k_1|\leq|k_2|\\k_1+k_2=k_{12}}}\frac{1}{|k_1|^2|k_2|^2}\\
&\lesssim\sum_{\substack{k_{1234}=k\\|k|\sim2^q}}\frac{1}{|k_{123}|^2|k_{12}|^{3-\delta}|k_3|^2|k_4|^2}\sum_{k_1}\frac{1}{|k_1|^{3+\delta}}\\
&\lesssim \sum_{\substack{k_{1234}=k\\|k|\sim2^q}}\frac{1}{|k_{123}|^2|k_4|^2}\Big(\sum_{\substack{|k_{12}|\leq|k_3|\\k_{12}+k_3=k_{123}}}+\sum_{\substack{|k_{12}|\geq|k_3|\\k_{12}+k_3=k_{123}}}\Big)\frac{1}{|k_{12}|^{3-\delta}|k_3|^2}\\
&\lesssim \sum_{\substack{k_{1234}=k\\|k|\sim2^q}}\frac{1}{|k_{123}|^{4-2\delta}|k_4|^2}\sum_{k_{12}}\frac{1}{|k_{12}|^{2\delta}}\\
&\lesssim\sum_{|k|\sim2^q}\Big(\sum_{\substack{|k_{123}|\leq|k_4|\\k_{123}+k_4=k_{1234}}}+\sum_{\substack{|k_{123}|\geq|k_4|\\k_{123}+k_4=k_{1234}}}\Big)\frac{1}{|k_{123}|^{4-2\delta}|k_4|^2}\lesssim 2^q
\end{align*}\normalsize
Therefore
\small\[
\E\Big| \Delta_q \Big( X_{s,t}^{\eps,\<tree1222>}-\pi_2\big(X^{\eps,\<tree1222>}\big)_{s,t} - c_\eps^{\<tree1222>}(s,t)\Big)\Big|^2\lesssim (t-s)^{2\vartheta}2^{-2q(\frac{3}{2}-2\vartheta)}
\]\normalsize

\subsection*{Definition of $X^{\<tree124>}$}

By definition
\small\begin{align*}
X_t^{\eps,\<tree124>}(x) =-\sum_{\substack{k\in\Z\\k_{1234}=k}}& F_t^{\eps,\<tree124>}(k,k_{12},k_{34},k_1,k_2,k_3,k_4)\times\\
&\times (k_{12}\cdot k_{34})(k_1\cdot k_2)(k_3\cdot k_4)\hat{\xi}(k_1)\hat{\xi}(k_2)\hat{\xi}(k_3)\hat{\xi}(k_4) e_k
\end{align*}\normalsize
where $F_t^{\eps,\<tree124>}(k,k_{12},k_{34},k_1,k_2,k_3,k_4)=\int_0^t e^{-|k|^2(t-s)} F_s^{\eps,\<tree12>}(k_{12},k_1,k_2)F_s^{\eps,\<tree12>}(k_{34},$ $k_3,k_4)\dd s$ for $k_1$, $k_2$, $k_3$, $k_4\in\Z_0^3$ and $t\geq 0$, which in turn admits the following bound
\small\[
|F_{s,t}^{\eps,\<tree124>}(k,k_{12},k_{34},k_1,k_2,k_3,k_4)|\lesssim \frac{(t-s)^\vartheta}{|k|^{2-2\vartheta}|k_{12}|^2|k_{34}|^2|k_1|^2|k_2|^2|k_3|^2|k_4|^2}
\]\normalsize

\subsubsection*{$0^{th}$-chaos}

Applying Wick's theorem as we did for $X^{\<tree1222>}$ and taking into account the symmetry of the sum in both $k_1,k_2$ and $k_3,k_4$, we obtain
\small\[
c_\eps^{\<tree124>}(t)=2\sum_{k_1,k_2\in\Z_0^3}F_t^{\eps,\<tree124>}(0,k_{12},k_{12},k_1,k_2,k_1,k_2) |k_{12}|^2|k_1\cdot k_2|^2
\]\normalsize
This corresponds to the divergent part of our term and we postpone its analysis to  Section~\ref{sec:const}.

\subsubsection*{$2^{nd}$-chaos}

The second chaos component $\pi_2\Big(X^{\eps,\<tree124>}\Big)(t,x)$ is 
\small\[
4\sum_{k_1,k_2,k_3} F_t^{\eps,\<tree124>}(k_{23},k_{12},k_{3(-1)},k_1,k_2,k_3,k_1) (k_{12}\cdot k_{3(-1)})(k_1\cdot k_2)(k_3\cdot k_1)\hat{\xi}(k_2)\hat{\xi}(k_3)e^{i k_{23}\cdot x}
\]\normalsize
hence
\small\begin{multline*}
\E\Big[\Delta_q \Big(\pi_2\big(X^{\<tree124>}\big)_{s,t}^2\Big)\Big]\\
\lesssim (t-s)^{2\vartheta}\sum_{k_2+k_3=k_{23}}\varrho_q(k_{23})^2\frac{1}{|k_{23}|^{4-4\vartheta}|k_2|^2|k_3|^2}\sum_{k_1\in\Z_0^3}\frac{1}{|k_{12}|^2|k_{3(-1)}|^2|k_1|^4}
\end{multline*}\normalsize
and the sum is bounded by
\small\begin{align*}
\sum_{|k_{23}|\sim2^q}&\frac{1}{|k_2|^2|k_3|^2}\Big(\sum_{|k_1|\leq|k_{12}|}+\sum_{|k_{12}|\leq|k_1|}\Big)\frac{1}{|k_{12}|^2|k_{3(-1)}|^2|k_1|^4}\\
&\lesssim \sum_{|k_{23}|\sim2^q}\frac{1}{|k_2|^3|k_3|^2}\Big(\sum_{|k_1|\leq|k_{3(-1)}|}+\sum_{|k_{3(-1)}|\leq|k_1|}\Big)\frac{1}{|k_{3(-1)}|^2|k_1|^5}\\
&\lesssim\sum_{|k_{23}|\sim2^q}\frac{1}{|k_2|^3|k_3|^3}\sum_{k_1}\frac{1}{|k_1|^6}\lesssim \sum_{|k_{23}|\sim2^q}\frac{1}{|k_{23}|^{3-\delta}}\sum_{k_2}\frac{1}{|k_2|^{3+\delta}}\lesssim 2^{\frac{q\delta}{2}}
\end{align*}\normalsize

\subsubsection*{$4^{th}$-chaos}

As for $X^{\<tree1222>}$ we have
\small\begin{align*}
\E\Big| \Delta_q \Big( &X_{s,t}^{\eps,\<tree124>}-\pi_2\big(X^{\eps,\<tree124>}\big)_{s,t} - c_\eps^{\<tree124>}(s,t)\Big)\Big|^2\\
&\lesssim (t-s)^{2\vartheta} \sum_{\substack{k\in\Z_0^3\\k_{1234}=k}} \frac{\varrho_q(k)^2}{|k|^{4-4\vartheta}|k_{12}|^2|k_{34}|^2|k_1|^2|k_2|^2|k_3|^2|k_4|^2}\\
&\lesssim  (t-s)^{2\vartheta} 2^{-2q(2-2\vartheta)}\sum_{\substack{|k|\sim2^q\\k_{1234}=k}}\frac{1}{|k_{12}|^2|k_{34}|^2}\sum_{k_1+k_2=k_{12}}\frac{1}{|k_1|^2|k_2|^2}\sum_{k_3+k_4=k_{34}}\frac{1}{|k_3|^2|k_4|^2}
\end{align*}\normalsize
Proceeding as before the latter sum is bounded by
\small\begin{multline*}
\sum_{\substack{|k|\sim2^q\\k_{1234}=k}}\frac{1}{|k_{12}|^{3-\delta}|k_{34}|^{3-\delta}}\\
\lesssim \sum_{|k|\sim 2^q}\sum_{\substack{k_{12}:|k_{12}|\leq |k_{34}|\\k_{1234}=k}}\frac{1}{|k_{12}|^{3-\delta}|k_{34}|^{3-\delta}}\lesssim \sum_{|k|\sim 2^q}\frac{1}{|k|^{3-2\delta}}\sum_{k_{12}}\frac{1}{|k_{12}|^{3+\delta}}\lesssim 2^{2q\delta}
\end{multline*}\normalsize

\subsection*{Definition of $\nabla Q\circ\nabla X$}

Recall the definition of $Q^\eps$
\small\[
Q_t(x)=\I\big(\nabla X^\eps\big)(t,x)=i\sum_{k\in\Z_0^3} F_t^{\eps,Q}(k) \,k\,\hat{\xi}(k) e_k
\]\normalsize
where $F_t^{\eps,Q}(k)=\int_0^t e^{-\frac{1}{2}|k|^2(t-s)}F_s^\eps(k)\dd s$, then $\nabla Q^\eps\circ\nabla X^\eps(t,x)$ is
\small\[
-i \sum_{\substack{k\in\Z_0^3\\k_{12}=k\\|i-j|\leq1}}\varrho_i(k_1)\varrho_j(k_2)F_t^{\eps,Q}(k_1) F_t^\eps(k_2) (k_1k_1^\star \cdot k_2 )\hat{\xi}(k_1)\hat{\xi}(k_2) e_k\,.
\]\normalsize
Notice that
\small\begin{multline*}
\E\Big[\nabla Q^\eps\circ\nabla X^\eps(t,x)\Big]=i \sum_{|i-j|\leq1}\sum_{k\in\Z_0^3} \varrho_i(k)\varrho_j(k)F_t^{\eps,Q}(k) F_t^\eps(k)(kk^\star \cdot k )=0
\end{multline*}\normalsize
where the last equality follows by the fact that the argument of the previous sum is odd. 
\small\begin{multline*}
\E\Big| \Delta_q \Big(\nabla Q^\eps\circ\nabla X^\eps\Big)_{s,t}\Big|^2\\
\lesssim \sum_{k\in\Z_0^3} \varrho_q(k)^2 \sum_{\substack{k_{12}=k\\|i-j|\leq1}}\varrho_i(k_1)^2\varrho_j(k_2)^2F_{s,t}^{\eps,Q}(k_1)^2 F_{s,t}^\eps(k_2)^2 |k_1k_1^\star \cdot k_2|^2
\end{multline*}\normalsize
Since, $|F_{s,t}^{\eps,Q}(k)|\lesssim \frac{(t-s)^\vartheta}{|k|^{4-2\vartheta}}$ we have
\small\[
\E\Big| \Delta_q \Big(\nabla Q^\eps\circ\nabla X^\eps\Big)_{s,t}\Big|^2 \lesssim (t-s)^{2\vartheta}\sum_{k\in\Z_0^3} \varrho_q(k)^2 \sum_{\substack{k_{12}=k\\|i-j|\leq1}}\varrho_i(k_1)^2\varrho_j(k_2)^2 \frac{1}{|k_1|^{4-4\vartheta}|k_2|^2}
\]\normalsize
and the sum is bounded by
\small\begin{align*}
&\sum_{\substack{q\lesssim j\\|i-j|\leq1}}\sum_{|k|\sim2^q}\Big(\sum_{\substack{k_1:|k_1|\leq|k_2|\\k_{12}=k}}+\sum_{\substack{k_1:|k_2|\leq|k_1|\\k_{12}=k}}\Big)\varrho_i(k_1)^2\varrho_j(k_2)^2 \frac{1}{|k_1|^{4-4\vartheta}|k_2|^2} \\
&\lesssim\sum_{\substack{q\lesssim j\\|i-j|\leq1}}\sum_{|k|\sim2^q}\Big(2^{-2j} 2^{-i(1-4\vartheta -\delta)}\sum_{k_1}\frac{1}{|k_1|^{3+\delta}}+2^{-i(3-4\vartheta-\delta)}\sum_{k_2}\frac{1}{|k_2|^{3+\delta}}\Big)\\
&\lesssim2^{3q}\sum_{q\lesssim j} 2^{-j(3-4\vartheta-\delta)}\lesssim2^{-2q(-2\vartheta-\frac{\delta}{2})}
\end{align*}\normalsize
and this concludes the proof. 

\subsection{The Renormalization Constants}
\label{sec:const}

As in the case of $X^{\<tree12>}$, in order to ensure the convergence of $X_t^{\eps,\<tree1222>}$ and $X_t^{\eps,\<tree124>}$ to well-defined stochastic processes, it is necessary to carefully study their $0$-th chaos component. Indeed, since in principle these are deterministic functions of time, it might happen that they diverge and in that case they need to be subtracted in order to obtain some sensible limit. They are respectively given by 
\[
c_\eps^{\<tree1222>}(t)=\mathbb E\left[X_t^{\eps,\<tree1222>}\right]\quad\text{and}\quad c_\eps^{\<tree124>}(t)= \mathbb E\left[X_t^{\eps,\<tree124>}\right]\,.
\]
Let us begin by analyzing $c_\eps^{\<tree1222>}(t)$. By Wick theorem, $\mathbb E[\nabla X^\eps(t)\nabla X^{\eps,\<tree122>}(t)]$ is
\small\begin{multline*}
-2\sum_{k_1,k_2\ne0;k_1\ne-k_2}(k_2\cdot k_{12})(k_1\cdot k_2)(1-e^{-\frac{1}{2}|k_1|^2t})|m(\eps k_1)|\\
\times\int_0^t\dd s\int_0^s\dd\sigma e^{-\frac{1}{2}|k_1|^2(t-s)}e^{-\frac{1}{2}|k_{12}|^2(s-\sigma)}F_s^\eps(k_2)F_\sigma^\eps(k_1)F_\sigma^\eps(k_2)%(1-e^{-\frac{1}{2}|k_2|^2s})(1-e^{-\frac{1}{2}|k_2|^2\sigma})(1-e^{-\frac{1}{2}|k_1|^2\sigma})
\end{multline*}\normalsize
Now is not difficult to see that, upon pulling out of the previous integrals all the quantities not depending on $\sigma$ or $s$,  the latter is less than
\small\begin{multline*}
\int_0^t\int_0^s\dd s\dd\sigma e^{-\frac{1}{2}|k_1|^2(t-s)}e^{-\frac{1}{2}|k_{12}|^2(s-\sigma)}(1-e^{-\frac{1}{2}|k_1|^2\sigma})\times\\
\times(e^{-\frac{1}{2}|k_2|^2s}+e^{-\frac{1}{2}|k_2|\sigma}+e^{-\frac{1}{2}|k_2|^2(s+\sigma)})
\lesssim\frac{t^{-1+\frac{3}{2}\nu}}{|k_{12}|^{2-\nu}|k_2|^{2-\nu}|k_1|^{2-\nu}} 
\end{multline*}\normalsize
for $\nu>0$ small enough. Then, the previous sum can be bounded by 
\small\begin{multline*}%\label{eq:cons-bound1}
t^{3/2\nu-1}\sum_{k_1,k_2}\frac{1}{|k_1|^{3-\nu}|k_2|^{4-\nu}|k_{12}|^{1-\nu}}\\%\lesssim  t^{3/2\nu-1}\Big(\sum_{|k_1|\leq|k_2|}+\sum_{|k_2|\leq|k_1|}\Big)\frac{1}{|k_1|^{3-\nu}|k_2|^{4-\nu}|k_{12}|^{1-\nu}} \\
\lesssim t^{3/2\nu-1}(\sum_{k_1,k_2}|k_1|^{3+\nu}|k_{12}|^{5-3\nu}+\sum_{k_1,k_2}|k_2|^{4-\nu}|k_{12}|^{4-2\nu})<+\infty
\end{multline*}\normalsize
for $\nu>0$ small enough. Then, to show convergence of $\mathbb E[\nabla X(t)\nabla X^{\<tree122>}(t)]$, it remains to study the contribution given by the two following integrals
\small\begin{align*}
&\int_0^t\dd s\int_0^s\dd\sigma  e^{-\frac{1}{2}|k_1|^2(t-s)}e^{-\frac{1}{2}|k_{12}|^2(s-\sigma)}\,,\\
&\int_0^t\dd s\int_0^s\dd\sigma  e^{-\frac{1}{2}|k_1|^2(t-s)}e^{-\frac{1}{2}|k_{12}|^2(s-\sigma)}e^{-\frac{1}{2}|k_1|^2\sigma}.
\end{align*}\normalsize
A direct computation gives
\small\begin{multline}\label{eq:int}
\int_0^t\dd s\int_0^s\dd\sigma  e^{-\frac{1}{2}|k_1|^2(t-s)}e^{-\frac{1}{2}|k_{12}|^2(s-\sigma)}\\
=2\frac{1-e^{-\frac{1}{2}|k_1|^2t}}{|k_{12}|^2|k_1|^2}-\frac{2}{|k_{12}|^2}\int_0^t\dd s e^{-\frac{1}{2}|k_1|^2(t-s)}e^{-\frac{1}{2}|k_{12}|^2s}
\end{multline}\normalsize
and we observe
\small\begin{multline*}
\sum_{k_1,k_2}\frac{|(k_2\cdot k_{12})(k_1\cdot k_2)|}{|k_1|^2|k_2|^4|k_{12}|^2}(1-e^{-\frac{1}{2}|k_1|^2t})\int_0^t\dd s e^{-\frac{1}{2}|k_1|^2(t-s)}e^{-\frac{1}{2}|k_{12}|^2s}\\
\lesssim\sum_{k_1,k_2}\frac{1}{t^{1-\nu}|k_1|^{3-\nu}|k_2|^2|k_{12}|^{3-\nu}}<+\infty
\end{multline*}\normalsize
where we have bounded the integral term by $t^{\nu-1}|k_1|^{\nu-1}|k_{12}|^{\nu-1}$ for $\nu>0$ small enough, and the convergence of the sum appearing at the right hand side is obtained as before. 
Now, let us focus on the contribution to the sum given by the first summand at the right hand side of~\eqref{eq:int}, i.e. 
\small$$
\sum_{k_1,k_2\ne0;k_1\ne-k_2}\frac{(k_2\cdot k_{12})(k_1\cdot k_2)}{|k_1|^4|k_2|^4|k_{12}|^2}(1-e^{-\frac{1}{2}|k_1|^2s})^2|m(\eps k_1)|^2|m(\eps k_2)|^2
$$\normalsize
Splitting this sum according to the following decomposition $(1-e^{-|k_1^2t})^2=1+e^{-|k_1|^2t}(2+e^{-|k_1|^2t})$ we are lead to the following terms  
\small\begin{align*}
\sum_{k_1,k_2,k_{12}\ne0}&\frac{(k_2\cdot k_{12})(k_1\cdot k_2)}{|k_1|^4|k_2|^4|k_{12}|^2}|m(\eps k_1)|^2|m(\eps k_2)|^2,\\
\sum_{k_1,k_2,k_{12}\ne0}&\frac{(k_2\cdot k_{12})(k_1\cdot k_2)}{|k_1|^4|k_2|^4|k_{12}|^2}e^{-\frac{1}{2}|k_1|^2t}(2+e^{-\frac{1}{2}|k_1|^2t})|m(\eps k_1)|^2|m(\eps k_2)|^2\,.
\end{align*}\normalsize
For the first, notice that
\small\begin{multline*}
\sum_{k_1,k_2,k_{12}\ne0}\frac{(k_2\cdot k_{12})(k_1\cdot k_2)}{|k_1|^4|k_2|^4|k_{12}|^2}|m(\eps k_1)|^2|m(\eps k_2)|^2\\
=\sum_{k_1,k_2,k_{12}\ne0}\frac{k_1\cdot k_2}{|k_1|^4|k_2|^4}|m(\eps k_1)|^2|m(\eps k_2)|^2\\
-\sum_{k_1,k_2,k_{12}\ne0}\frac{(k_1\cdot k_{12})(k_1\cdot k_2)}{|k_1|^4|k_2|^4|k_{12}|^2}|m(\eps k_1)|^2|m(\eps k_2)|^2
\end{multline*}\normalsize
and thus 
\small\begin{multline*}
\sum_{k_1,k_2,k_{12}\ne0}\frac{(k_2\cdot k_{12})(k_1\cdot k_2)}{|k_1|^4|k_2|^4|k_{12}|^2}|m(\eps k_1)|^2|m(\eps k_2)|^2\\
=\frac{1}{2}\sum_{k_1,k_2,k_{12}\ne0}\frac{k_1\cdot k_2}{|k_1|^4|k_2|^4}|m(\eps k_1)|^2|m(\eps k_2)|^2=\frac{1}{2}\sum_{k_1\ne0}|k_1|^{-6}|m(\eps k_1)|^4
\end{multline*}\normalsize
where we have used that the function $m$ is even, and by dominated convergence theorem we conclude that the right hand side converges to $\sum_{k_1\ne0}|k_1|^{-6}<+\infty$ as $\eps$ goes to zero. The second term instead
\small\begin{multline*}
\sum_{k_1,k_2,k_{12}\ne0}\frac{(k_2\cdot k_{12})(k_1\cdot k_2)}{|k_1|^4|k_2|^4|k_{12}|^2}e^{-\frac{1}{2}|k_1|^2t}(2+e^{-\frac{1}{2}|k_1|^2t})|f(\eps k_1)|^2|f(\eps k_2)|^2\\
=\sum_{k_1,k_2,k_{12}\ne0}\frac{k_1\cdot k_2}{|k_1|^4|k_2|^4}|f(\eps k_1)|^2|f(\eps k_2)|^2e^{-\frac{1}{2}|k_1|^2t}(2+e^{-\frac{1}{2}|k_1|^2t})\\
-\sum_{k_1,k_2,k_{12}\ne0}\frac{(k_1\cdot k_{12})(k_1\cdot k_2)}{|k_1|^4|k_2|^4|k_{12}|^2}e^{-\frac{1}{2}|k_1|^2t}(2+e^{-\frac{1}{2}|k_1|^2t})|f(\eps k_1)|^2|f(\eps k_2)|^2
\end{multline*}\normalsize
Since $m$ is even, the first sum of this decomposition is finite. To study the second one we simply use  the following elementary estimate 
\small\begin{multline*}
\sum_{k_1,k_2,k_{12}\ne0}\frac{(k_1\cdot k_{12})(k_1\cdot k_2)}{|k_1|^4|k_2|^4|k_{12}|^2}e^{-\frac{1}{2}|k_1|^2t}(2+e^{-\frac{1}{2}|k_1|^2t})|f(\eps k_1)|^2|f(\eps k_2)|^2\\
\lesssim t^{\nu/2-1}\sum_{k_1,k_2}\frac{1}{|k_1|^{4-\nu}|k_2|^3|k_{12}|}<+\infty
\end{multline*}\normalsize
and this concludes the bound for this term. To obtain the needed bound for the expectation $E[\nabla X(t)\nabla X^{\<tree122>}(t)]$ it remains to study
\small\begin{multline}\label{comp1}
%\sum_{k_1,k_2} \frac{(k_2\cdot k_{12})(k_1\cdot k_2)}{|k_1|^2|k_{12}|^2}\left(1-e^{-\frac{1}{2}|k_1|^2t}\right)\int_0^t\dd s\int_0^s \dd \sigma e^{-\frac{1}{2}|k_1|^2(t-s)} e^{-\frac{1}{2}|k_{12}|^2(s-\sigma)} e^{-\frac{1}{2}|k_1|^2\sigma} \\
\sum_{k_1} \frac{1-e^{-\frac{1}{2}|k_1|^2t}}{|k_1|^2}\int_0^t\dd s\int_0^s \dd \sigma e^{-\frac{1}{2}|k_1|^2(t-s)}e^{-\frac{1}{2}|k_1|^2\sigma} k_1\cdot \\
\cdot\left(\sum_{k_2} \frac{1}{|k_2|^4}e^{-|k_{12}|^2(s-\sigma)} (k_{12}\cdot k_2) k_2\right)
\end{multline}\normalsize
Let us have a closer look at the quantity in the parenthesis. Notice that by symmetry ($k_2\to -k_2$) the sum in the parenthesis at the right hand side is $0$. Therefore, it can be written as
\small\begin{multline*}
%\sum_{k_2} \frac{1}{|k_2|^4}e^{-\frac{1}{2}|k_{12}|^2(s-\sigma)} (k_{12}\cdot k_2) k_2\\
\sum_{k_2} \frac{1}{|k_2|^4}\left(e^{-\frac{1}{2}|k_{12}|^2(s-\sigma)}(k_{12}\cdot k_2)-e^{-\frac{1}{2}|k_2|^2(s-\sigma)}(k_{2}\cdot k_2)\right)  k_2\\
=\sum_{k_2}\frac{1}{|k_2|^4}\int_0^1 \dd \tau e^{-\frac{1}{2}|\tau k_1+k_2|^2(s-\sigma)}\times\\
\times \left( k_1\cdot k_2 -2(k_1\cdot (\tau k_1+k_2))((\tau k_1+k_2)\cdot k_2)(s-\sigma)\right) k_2
\end{multline*}\normalsize
where in the last line we applied Taylor's theorem to the function $G(k)=e^{-\frac{1}{2}|x+k_2|^2(s-\sigma)}(k+k_2)\cdot k_2$. The modulus of the sum in~\eqref{comp1} can consequently be bounded by
\small\begin{multline}\label{comp2}
\sum_{k_1} \frac{1}{|k_1|}\int_0^1\sum_{k_2} \frac{1}{|k_2|^3}\int_0^t\dd s\int_0^s \dd \sigma e^{-\frac{1}{2}|k_1|^2(t-s)}e^{-\frac{1}{2}|k_1|^2\sigma} e^{-\frac{1}{2}|\tau k_1+k_2|^2(s-\sigma)}\times\\
\times |k_1||k_2|\left( 1+|\tau k_1+k_2|^2(s-\sigma)\right)
=\sum_{k_1} \frac{1}{|k_1|}\int_0^t\dd s\int_0^s \dd \sigma e^{-\frac{1}{2}|k_1|^2(t-s)}e^{-\frac{1}{2}|k_1|^2\sigma}\times\\
 \int_0^1\sum_{k_2} \frac{1}{|k_2|^2}e^{-\frac{1}{2}|\tau k_1+k_2|^2(s-\sigma)} |k_1|\left( 1+|\tau k_1+k_2|^2(s-\sigma)\right)
\end{multline}\normalsize
Let us write the right hand side as the sum of two terms and call them $\Sigma_1$ and $\Sigma_2$ respectively. % in  consider the two summands $\Sigma_1$ and $\Sigma_2$, respectively given by
%\small\begin{align*}
%\Sigma_1&\eqdef\sum_{k_1} \frac{1}{|k_1|}\int_0^t\dd s\int_0^s \dd \sigma e^{-\frac{1}{2}|k_1|^2(t-s)}e^{-\frac{1}{2}|k_1|^2\sigma} \int_0^1\sum_{k_2} \frac{1}{|k_2|^2}e^{-\frac{1}{2}|\tau k_1+k_2|^2(s-\sigma)} |k_1|\\
%\Sigma_2&\eqdef \sum_{k_1} \frac{1}{|k_1|}\int_0^t\dd s\int_0^s \dd \sigma e^{-\frac{1}{2}|k_1|^2(t-s)}e^{-\frac{1}{2}|k_1|^2\sigma}\times\\
%&\times\int_0^1\sum_{k_2} \frac{1}{|k_2|^2} e^{-\frac{1}{2}|\tau k_1+k_2|^2(s-\sigma)} |k_1||\tau k_1+k_2|^2(s-\sigma)
%\end{align*}\normalsize
%separately although the way to bound them is analogous. 
Now, for $\Sigma_1$, notice that for $\eps>0$ sufficiently small, one has
\small\begin{multline*}
\sum_{k_2} \frac{1}{|k_2|^2}e^{-\frac{1}{2}|\tau k_1+k_2|^2(s-\sigma)} |k_1|\lesssim |k_1|\sum_{k_2}\frac{1}{|k_2|^2|\tau k_1+k_2|^{2-2\eps}(s-\sigma)^{1-\eps}}\\
\lesssim\frac{|k_1|}{(s-\sigma)^{1-\eps}} \int \frac{\dd y}{|y|^2|\tau k_1+y|^{2-2\eps}}\lesssim
\frac{|k_1|^\eps}{\tau^{1-\eps}(s-\sigma)^{1-\eps}}\int \frac{\dd y}{|y|^2|\frac{k_1}{|k_1|}+y|^{2-2\eps}}
\end{multline*}\normalsize
where the integral is taken over a suitable subset of $\R^3$ where the integrand is well-defined. It is immediate to see that the latter is bounded (for example by Cauchy-Schwartz). 

Analogously, for $\Sigma_2$, upon setting $\bar{\eps}\eqdef\frac{\eps}{2}>0$, we get
\small\[
\sum_{k_2} \frac{1}{|k_2|^2} e^{-\frac{1}{2}|\tau k_1+k_2|^2(s-\sigma)} |k_1||\tau k_1+k_2|^2(s-\sigma)\lesssim |k_1| \sum_{k_2}\frac{1}{|k_2|^2|\tau k_1+k_2|^{2-\eps}(s-\sigma)^{1-\eps}}
\]\normalsize
and the latter can be treated as before. 

At this point, given $\delta$, $\gamma>0$,~\eqref{comp2} is bounded by
\small\begin{align*}
\sum_{k_1}\frac{1}{|k_1|}\int_0^t\dd s&\int_0^s \dd \sigma e^{-\frac{1}{2}|k_1|^2(t-s)}e^{-\frac{1}{2}|k_1|^2\sigma} \int_0^1\frac{|k_1|^\eps}{\tau^{1-\eps}(s-\sigma)^{1-\eps}}\\
%&\lesssim \sum_{k_1}\frac{1}{|k_1|^{1-\eps}}\int_0^t\dd s\int_0^s \dd \sigma  \frac{e^{-\frac{1}{2}|k_1|^2(t-s)}e^{-\frac{1}{2}|k_1|^2\sigma}}{(s-\sigma)^{1-\eps}}\\
&\lesssim \sum_{k_1}\frac{1}{|k_1|^{1-\eps}}\int_0^t\dd s \frac{1}{|k_1|^{2-2\gamma}(t-s)^{1-\gamma}}\int_0^s \dd \sigma \frac{1}{|k_1|^{2-2\delta}\sigma^{1-\delta}(s-\sigma)^{1-\eps}}\\
&\lesssim \sum_{k_1}\frac{1}{|k_1|^{5-2\gamma-2\delta-\eps}}\int_0^t \frac{1}{s^{1-\delta-\eps}(t-s)^{\delta}}\lesssim \sum_{k_1} \frac{t^{-1+\gamma+\delta+\eps}}{|k_1|^{5-2\gamma-2\delta-\eps}}
\end{align*}\normalsize
and the last term is bounded provided that $5-2\gamma-2\delta-\eps>3$ and $-1+\gamma+\delta+\eps>0$. Therefore $\sup_{t\in[0,T]}t^{3/2\nu-1}|\mathbb E[\nabla X^\eps(t)\nabla X^{\eps,\<tree122>}(t)]|$ is convergent and by dominated convergence we can conclude that the constant $c_\eps^{\<tree1222>}$ does not diverge and can therefore be omitted. 

We can now focus on $c_\eps^{\<tree124>}(t)$. In particular we would like to show that $c_\eps^{\<tree124>}(t) = c_\eps^{\<tree124>}t +R^\eps(t)$, where $c_\eps^{\<tree124>}$ is a diverging constant and $R^\eps(t)$ is finite uniformly in $\eps$. Applying Wick's theorem as we did for $X^{\<tree1222>}$ and taking into account the symmetry of the sum in both $k_1,k_2$ and $k_3,k_4$, we obtain
\small\begin{align*}
c_\eps^{\<tree124>}(t)=2\sum_{k_1,k_2}\prod_{h=1}^2 m(\eps k_h)^2\int_0^t  \left(I_1(s)^2+2 I_1(s)I_2(s) + I_2(s)^2\right) \dd s\frac{|k_{12}|^2|k_1\cdot k_2|^2}{|k_1|^4|k_2|^4}\,.
\end{align*}\normalsize
where $I_1$ and $I_2$ are given by 
%\small\[
%\int_0^s e^{-\frac{1}{2}|k_{12}|^2(s-r)}\left(1-e^{-\frac{1}{2}|k_1|^2r} \right)\left(1-e^{-\frac{1}{2}|k_2|^2r} \right)\dd r=I_1(s)+I_2(s)
%\]\normalsize
%where 
\small\begin{align*}
I_1(s):&=\frac{1- e^{-\frac{1}{2}|k_{12}|^2s}}{|k_{12}|^2}& & I_2(s)\eqdef\sum_{i=1}^3\int_0^s e^{-\frac{1}{2}|k_{12}|^2(s-r)} e^{-\frac{1}{2}a_i r}\dd r%\\
%I_3(s):&=\int_0^se^{-|k_{12}|^2(s-r)} e^{-|k_2|^2r}\dd s& &I_4(s)\eqdef\int_0^s e^{-|k_{12}|^2(s-r)}e^{-(|k_1|^2+|k_2|^2)r}\dd s
\end{align*}\normalsize
and $a_1=|k_1|^2$, $a_2=|k_2|^2$ and $a_3=|k_1|^2+|k_2|^2$. Let us begin with the term involving $I_1^2$, which, by expanding the square, equals
\small\begin{align*}
%\sum_{k_1,k_2}\prod_{h=1}^2 m(\eps k_h)^2\left(\frac{t}{|k_{12}|^4}-2\frac{1-e^{-|k_{12}|^2t}}{|k_{12}|^6}+\frac{1-e^{-2|k_{12}|^2t}}{2|k_{12}|^6}\right)\frac{|k_{12}|^2|k_1\cdot k_2|^2}{|k_1|^4|k_2|^4}\\
\frac{1}{2}c_\eps^{\<tree124>}t+ \sum_{k_1,k_2}\prod_{h=1}^2 m(\eps k_h)^2\left(-2\frac{1-e^{-|k_{12}|^2t}}{|k_{12}|^6}+\frac{1-e^{-2|k_{12}|^2t}}{2|k_{12}|^6}\right)\frac{|k_{12}|^2|k_1\cdot k_2|^2}{|k_1|^4|k_2|^4}
\end{align*}\normalsize
where the constant $c_\eps^{\<tree124>}$ is defined as
\[
c_\eps^{\<tree124>}=2\sum_{k_1,k_2}m(\eps k_1)^2m(\eps k_2)^2\frac{|k_1\cdot k_2|^2}{|k_{12}|^2|k_1|^4|k_2|^4}
\]
For the other two summands notice that, by applying the same strategy as above, we have
\small\begin{align}
\sum_{k_1,k_2}\prod_{h=1}^2 m(\eps k_h)^2&\frac{|1-e^{-|k_{12}|^2t}|}{|k_{12}|^6}\frac{|k_{12}|^2|k_1\cdot k_2|^2}{|k_1|^4|k_2|^4}\lesssim \sum_{k_1,k_2}\frac{1}{|k_{12}|^4|k_1|^2|k_2|^2}\label{computations}\\
%&=\sum_{k_2}\frac{1}{|k_2|^2}\bigg(\sum_{\substack{k_2=k_{12}-k_1\\|k_1|\leq|k_{12}|}}+\sum_{\substack{k_2=k_{12}-k_1\\|k_{12}|\leq|k_1|}}\bigg)\frac{1}{|k_{12}|^4|k_1|^2}\notag\\
&\lesssim \sum_{k_2}\frac{1}{|k_2|^{5-\eps}}\sum_{k_1}\frac{1}{|k_1|^{3+\eps}}+\sum_{k_2}\frac{1}{|k_2|^{4}}\sum_{k_1}\frac{1}{|k_{12}|^{4}}\notag
\end{align}\normalsize
which converges for any $\eps>0$ small enough.

For the other terms, notice at first that if $\delta_1$, $\delta_2>0$, we have
\small\begin{align*}
\sum_{i=1}^3\int_0^s e^{-|k_{12}|^2(s-r)} e^{-a_i r}\dd r\lesssim \sum_i\int_0^s \frac{(s-r)^{-1+\delta_1}r^{-1+\delta_2}}{|k_{12}|^{2-2\delta_1}a_i^{1-1\delta_2}}\dd r \lesssim\max_i \frac{s^{-1+\delta_1+\delta_2}}{|k_{12}|^{2-2\delta_1}a_i^{1-\delta_2}}
\end{align*}\normalsize
It will be enough to consider $a_1=|k_1|^2$, since for $a_2$ the same bounds hold and $a_3>a_1^2$. 
Upon choosing $\delta_1+\delta_2>\frac{1}{2}$, we have
\small\begin{align*}
&\sum_{k_1,k_2}\prod_{h=1}^2 m(\eps k_h)^2\left|\int_0^t I_2(s)^2 \dd s\right|\frac{|k_{12}|^2|k_1\cdot k_2|^2}{|k_1|^4|k_2|^4}\lesssim \sum_{k_1,k_2} \frac{1}{|k_{12}|^{2-4\delta_1}|k_1|^{6-4\delta_2}|k_2|^2}\\
&\lesssim \sum_{k_1}\frac{1}{|k_1|^{6-4\delta_2}}\left(\sum_{|k_2|\leq|k_{12}|}+\sum_{|k_{12}|\leq|k_2|}\right)\frac{1}{|k_{12}|^{2-4\delta_1}|k_2|^2}\lesssim \sum_{k_1}\frac{1}{|k_1|^{6-4\delta_2}}\sum_{k_2}\frac{1}{|k_2|^{4-4\delta_1}}
\end{align*}\normalsize
which converges, provided that $\delta_2<\frac{3}{4}$ and $\delta_{1}<\frac{1}{4}$. 
Let us consider $I_1I_2$. In this case
\small\begin{multline*}
\sum_{k_1,k_2}\prod_{h=1}^2 m(\eps k_h)^2\left|\int_0^t I_1(s)I_2(s) \dd s\right|\frac{|k_{12}|^2|k_1\cdot k_2|^2}{|k_1|^4|k_2|^4}\\\lesssim \sum_{k_1,k_2}\frac{1}{|k_{12}|^{2-2\delta_1}|k_1|^{4-2\delta_2}|k_2|^2}
%&\lesssim \sum_{k_1}\frac{1}{|k_1|^{4-2\delta_2}}\left(\sum_{|k_2|\leq|k_{12}|}+\sum_{|k_{12}|\leq|k_2|}\right)\frac{1}{|k_{12}|^{2-2\delta_1}|k_2|^2}\\
\lesssim \sum_{k_1}\frac{1}{|k_1|^{4-2\delta_2}}\sum_{k_2}\frac{1}{|k_2|^{4-2\delta_1}}
\end{multline*}\normalsize
and the last converges provided that $\delta_1$, $\delta_2<\frac{1}{2}$. Then we conclude that the divergent part of the term $X^{\eps\<tree124>}$ is simply given by $t c^{\<tree124>}$
%\small$$
%2t\sum_{k_1,k_2}m(\eps k_1)^2m(\eps k_2)^2\frac{|k_1\cdot k_2|^2}{|k_{12}|^2|k_1|^4|k_2|^4}.
%$$\normalsize
and the proof of Theorem~\ref{th:stoc} is completed. 
\newline

To conclude this section, we want to analyze the constants $c^{\<tree12>}_\eps$ and $c_\eps^{\<tree124>}$, and understand their asymptotic behavior. 
Now, by Riemann-sum approximation, it easy to see that, as $\eps$ goes to $0$,
$$
c^{\<tree12>}_\eps=\sum_{k\in\mathbb Z^3,k\ne0}\frac{|m(\eps k)|^2}{|k|^2}\sim_{\eps\to0}\eps^{-1}\int_{|x|>1}|x|^{-2}|m(x)|^2
$$
where the integral is clearly finite. For the other, by elementary estimates, we have  
\small\begin{equation*}
c_\eps^{\<tree124>}%\lesssim_m\sum_{|k_1|,|k_2|\lesssim \eps^{-1}}\frac{|k_1\cdot k_2|^2}{|k_1|^4|k_2|^2|k_{12}|^2}\\
\lesssim \sum_{|k_1|,|k_2|\lesssim \eps^{-1}}|k_1|^{-2}|k_2|^{-2}|k_{12}|^{-2}\lesssim \sum_{|k_1|,|k_2|\lesssim \eps^{-1}}|k_1|^{-3}|k_2|^{-3}\lesssim (\log(\eps))^2
\end{equation*}\normalsize
where the latter follows once again by Riemann-sum approximation. 
%we have that 
%$$
%c_\eps^{\<tree124>}\lesssim \Big(\sum_{|k|\lesssim \eps^{-1}}|k|^{-3}\Big)^2\lesssim (\log(\eps))^2
%$$
%as stated in Remark~\ref{rem:Constants}. 

\section{The Polymer Measure with White Noise Potential and its properties}\label{sec:PropPol}

This section is devoted to the study of the Polymer Measure with white noise potential and its properties. Thanks to the result of the previous section, we can now state the following theorem whose proof will occupy the rest of the paper (compare with Theorems~\ref{t:ConstrPol} and~\ref{thm:PropPol}). 

\begin{theorem}\label{th:polymer}
Let $T>0$ and $\xi$ be spatial white-noise on the $d$-dimensional torus $\TT^d$ for $d=2,\,3$ and $\xi^\eps$ be given by 
\[
\xi^\eps=\sum_{k\in\mathbb Z^d}m(\eps k)\hat \xi(k)e_k
\]
where $\{\hat \xi(k)\}_{k\in\Z^d}$ is a family of standard normal random variables with covariance $\E[\hat \xi(k_1)\hat \xi(k_2)]=\1_{\{k_1=-k_2\}}$, $e_k$ is the Fourier basis $L^2(\mathbb T^d)$ and $m$ a smooth radial function with compact support such that $m(0)=1$.
For any $\eps>0$, define the probability measure $\Q^\eps_{T,x}$ on $C([0,T],\R^d)$ as
$$
\mathbb Q^\eps_{T,x}(\dd\omega)=Z^{-1}_\eps\exp\left(\int_0^T\xi^\eps(B_s)\dd s\right)\mathbb W_x(\dd\omega),\,\, Z^\eps\eqdef\mathbb E_{\mathbb W_x}\left[\exp{\int_0^T\xi^\eps(B_s)\dd s}\right]
$$
with $\mathbb W_x$ is the Wiener measure starting at $x$ (the white noise being independent of $\mathbb W$). 
Then there exists $T^\star>0$, depending only on $\xi$, such that for all $T\leq T^\star$ the family of probability measures $\mathbb Q^\eps_T$ converges to a measure $\mathbb Q_T$ independent of the choice of the mollifier ($\xi$-almost surely). 

Moreover $\mathbb Q_T$ is singular to the Wiener measure and we can choose $T^\star=\infty$ .
\end{theorem} 

\begin{remark}
Unfortunately, we are not allowed to consider a spatial white-noise on the full space $\R^d$, the reason being that such a noise does not live in any Besov-H\"older space $\CC^\beta(\R^d,\R)$. 
However, we believe that the problem can be handled by introducing some sort of weighted Besov-H\"older spaces and, in this direction, we mention the works of Hairer and Labb\'e~\cite{HL1,HL2}, where the authors prove a well-posedness result for the linear parabolic Anderson equation on $\R^d$, $d=2$ and $3$, and the recent paper of Mourrat and Weber~\cite{MW15}, in which they obtain an analogous result for the $\Phi_2^4$-equation. 
\end{remark}

\begin{remark}
The factor $1$ in front of the white-noise $\xi$ does not play any role in our study and can be replaced by any constant $\beta>0$. By Section~\ref{section:KPZ} and the analysis carried out therein, we guess that the behavior of the polymer measure as $\beta\to0$ is crucially related to that of the KPZ-type equation~\eqref{eq:KPZtype} with vanishing noise. In this direction, large deviation results have been recently investigated in the context of singular SPDEs, more specifically for the case of the stochastic quantization equation, by M. Hairer and H. Weber in~\cite{HW}.  
\end{remark}

We now begin with the proof of Proposition~\ref{prop:conv-kpz}, which represents the core of the existence part of the previous statement.

\subsection{Proof of Proposition~\ref{prop:conv-kpz}}

Thanks to Theorems~\ref{th:flow-kpz} and~\ref{th:stoc}, we know that there exists $T^\star>0$ such that for all $T\leq T^\star$ and $\delta>0$, $h^\eps(T-t,x)$ converges to $h(T-t,x)$  in $C_T \CC^{a(d)-\delta}$ in  probability, where $a(d)=1$ for $d=2$ and $a(d)=1/2$ for $d=3$. 
Hence, $V^\eps(t,x)\eqdef\nabla h^\eps(T-t,x)$ converges to $V(t,x)=\nabla h(T-t,x)$ in $C_T\CC^{a(d)-1-\delta}$. For $d=2$ the proof is completed. 

In the three dimensional case, recall that the solution $h$ to~\eqref{eq:KPZtype} admits the following decomposition
\[
h(t)= X_t + h^1(t)\qquad \text{where}\qquad h^1(t)=X^{\<tree12>}_t + 2X^{\<tree122>}_t + v(t).
\]
In the previous section we proved that $X$ and  $h^1$ are, respectively, almost $\frac{1}{2}$ and $1$, regular in space, hence, in $\CJ^T(\nabla\cdot\nabla h)\circ\nabla h$, given by
\small\[
\CJ^T(\nabla\cdot\nabla X)\circ\nabla X+\CJ^T(\nabla\cdot\nabla X)\circ\nabla h^1+\CJ^T(\nabla\cdot\nabla h^1)\circ\nabla X+\CJ^T(\nabla\cdot\nabla h^1)\circ\nabla h^1\,,
\]\normalsize
the only term not analitytically well defined is the first. At the same time, since $\CJ^T$ and $\nabla$ commute we have
\[
\CJ^T(\nabla\cdot\nabla X)\circ\nabla X = \nabla \CJ^T(\nabla X)\circ\nabla X
\]
Now, as we already pointed out the estimates for $\CJ_T$ and $\I$ are the same and $Q=\I(\nabla X)$, therefore the bounds for this process are the same as the ones for $\nabla Q\circ\nabla X$ which were derived in the proof of Theorem~\ref{th:stoc}, and the proof is completed. 

\subsection{Global existence, Parabolic Anderson equation and Feynman-Kac representation}
\label{section:pam}

In this section we want to show the global existence of the polymer measure. 
By the construction carried out in Section~\ref{sec:Polymer}, the result immediately follows if we are able to prove that the solution to the KPZ-type equation~\eqref{eq:KPZtype1} does not explode in finite time.  
So, let once again $h^\eps$ satisfy
\[
\partial_t h^\eps=\frac{1}{2}\Delta h^\eps+\frac{1}{2}|\nabla h^\eps|^2+\xi^\eps-c_\eps,\qquad\quad h(0,x)=0.
\]
Now, we know that there exists a time $T^\star>0$ such that $h^\eps$ converges to a function $h$ in $C_T \CC^{a(d)-\delta}$ in  probability, where $a(d)=1$ for $d=2$ and $a(d)=1/2$ for $d=3$, independently of the mollifier, for all $T\leq T^\star$. 
Moreover in~\cite{GIP15} and~\cite{HL2}, the authors proved global well-posedness for the parabolic Anderson equation in dimension $d=2$ and $3$, respectively, namely they showed that there exists a constant $b_\eps$ for which if $v^\eps$ denotes the solution of  
\begin{equation}\label{eq:pam}
\partial_t v^\eps=\frac{1}{2}\Delta v^\eps+v^\eps\xi^\eps-b_\eps v^\eps,\,\qquad v^\eps(0,x)=1
\end{equation} 
then for all $T>0$, $v^\eps$ converges in $C_T \CC^{a(d)-\delta}$ (in  probability) to a function $v$ independently of the approximation of the noise, and the constant $b_\eps$ can be chosen to be $c_\eps$. 
We need to take into account the following two facts
\begin{enumerate}
\item By Feynman-Kac formula, $v^\eps$ can be written as 
$$
v^\eps(t,x)=\mathbb E_x\left[e^{\int_0^t(\xi^\eps(B_s)-c_\eps)\dd s}\right]
$$
\item The Cole-Hopf transform of $h_\eps$ solves~\eqref{eq:pam}, i.e. $v^\eps(t,x)=e^{h^\eps(t,x)}$ for all $t\leq T^\star$ and $x\in\mathbb T^2$. Therefore, taking the limit as $\eps$ tends to $0$ we have $v(t,x)=e^{h(t,x)}>0$ for $t\leq T^\star$.
\end{enumerate}
Now, the point is that for all $t\leq T^\star$, the Markov property implies
\small$$
v^\eps(t+T^\star,x)=\mathbb E_{\mathbb W_x}\left[e^{\int_0^{T^\star}(\xi^\eps(B_s)-c_\eps)\dd s}v^\eps(t,B_{T^\star})\right]=v^\eps(T^\star,x)\mathbb E_{\mathbb Q^\eps_{x,T^\star}}[v^\eps(t,B_{T^\star})]
$$\normalsize
thus, since $\mathbb Q^\eps_{x,T}$ converges weakly to a probability measure $\mathbb Q_{x,T^\star}$ and $\|v^\eps-v\|_{C_TL^\infty}\to0$ we get immediately, by taking the limit as $\eps\to0$, that 
$$
v(T^\star+t,x)=v(T^\star,x)\mathbb E_{\mathbb Q_{x,T^\star}}[v(t,B_{T^\star})]>0
$$  
for all $t\leq T^\star$ which implies that $v(t,x)>0$ for all $t\leq 2T^\star$ and $x\in\mathbb T^2$. Iterating the same argument, we get that $v(t,x)>0$ for all $t\geq0$ and all $x$. At this point, it is not difficult to see that we can extend  $h^\eps$ and $h$ to the whole half-line $[0,+\infty)$ by setting $h^\eps(t)=\log(v^\eps(t))$ and $h(t)=\log(v(t))$ for all $t\geq0$. Therefore, it is immediate to verify  that the $h$ constructed in this way is a global in time solution of the equation~\eqref{eq:KPZtype1}. 

%\begin{remark}
%As a latter remark we want to point out that this proof works in $d=3$ modulo a global well-posedness result for the linear Parabolic-Anderson equation~\eqref{eq:pam}. 
%\end{remark}

\subsection{Singularity with respect to the Wiener measure}

We will give the proof in the case of dimension $3$ since an analogous, but simpler argument, holds for $d=2$.

In order to verify that the Polymer measure $\mathbb Q_{T,x}$ is singular with respect to the Wiener measure $\mathbb W_x$, we will begin by showing that the sequence of densities $\dd\mathbb Q_{T,x}^N/\dd\mathbb W$ (we are taking $\eps$ to be $1/N$) admits a subsequence converging to 0 $\mathbb W$-a.s. 
At first notice that, by Feynman-Kac formula, for all $T\leq T^\star$ we have
$$
e^{h_N(T,x)}=e^{-c_N T}\mathbb E_{\mathbb W_x}\left[e^{\int_0^T\xi^N(\omega_s)\dd s}\right]=Z_N^Te^{-c_N T}
$$
where $h_N$ is the solution to~\eqref{eq:KPZtype} driven by the mollified noise $\xi^N$, $c_N\eqdef c_N^{\<tree12>}+c_N^{\<tree124>}$  and $c_N^{\<tree12>},\,c_N^{\<tree124>}$ are the diverging constants introduced in Theorem~\ref{th:stoc}. 
Now, let $Y_N$ be the random variable given by 
\[
Y_N(\omega)\eqdef(Z_N^T)^{-1} e^{\int_0^T\xi^N(\omega_s)\dd s}
\]
and $\der<1$ be fixed. Thanks to the last part of the above mentioned theorem, we immediately see that 
\small\begin{equation}\label{e:DensityMoments}
\mathbb E_{\mathbb W}[(Y_N)^\der]=(Z_N^T)^{-\der}\mathbb E_{\mathbb W}\left[e^{\int_0^T\der\xi^N(\omega_s)\dd s}\right]=e^{(-c^{\<tree12>}_N\der(1-\der)+\der^4c^{\<tree124>}_N)T}e^{h_\der^N(T,x)-h^N(T,x)}
\end{equation}\normalsize
where, this time, $h_\der^N$ is the solution to the KPZ-type equation~\eqref{eq:KPZtype} driven by $\der\xi^N$ and $\xi^N$ is the same as before. 
Since, by Remark~\ref{rem:Constants} as $N\to\infty$, $c^{\<tree12>}_N\sim N$ while $c^{\<tree124>}_N=O((\log N)^2)$, we have
\[
\lim_{N\to+\infty}c^{\<tree12>}_N\der(1-\der)-(\der^4-\der)c^{\<tree124>}_N=\lim_N c^{\<tree12>}_N\der(1-\der)=+\infty
\]
it follows that there exists a subsequence (denoted once again as the argument of the limit at the left hand side of the previous) for which the the first exponential at right hand side of~\eqref{e:DensityMoments} is summable. 
Moreover, thanks to Theorem~\ref{th:flow-kpz} the second exponential converges, as $N\to\infty$ to $\exp{\big(h_\der(T,x)-h(T,x)\big)}$. Therefore
$$
\sum_{N}\mathbb E_{\mathbb W}\left[(Y^N)^\der\right]\lesssim\sum_{N}e^{(-c^{\<tree12>}_N\der(1-\der)+(\der^4-\der)c^{\<tree124>}_N)T}<+\infty
$$
which in particular implies that, if $A_N^r=\left\{Y_N<r\right\}$, then $\mathbb W(\limsup_{N} A_N^r)=1$. 

The point is to prove that instead $\mathbb Q(\limsup_{N} A_N^r)=0$. By Portemanteau theorem we have  
$$
\mathbb Q(A_N^r)\leq\liminf_{L\to+\infty}\mathbb Q^L(A_N^r)
$$ 
so we need to suitably bound $\mathbb Q^L(A_N^r)$. Let us notice that 
$$
\mathbb Q^L(A_N^r)=\mathbb E_{W}\left[Y_L1_{A_N^r}\right]\lesssim r^\delta (Z_L)^{-1}(Z_N)^\der\mathbb E_{\mathbb W}\left[e^{\int_0^T\xi^L(\omega_s)-\der\xi^N(\omega_s)\dd s}\right]\,.
$$
Denoting by $c_L$ the diverging constant associated to $h^L$ and using once again the inequality $(Z_L)^{-1}(Z_N)^\der\lesssim e^{T(\der c_N-c_L)}$, we get 
$$
\mathbb Q^L(A_N^r)\lesssim e^{T(\der c_N-c_L)}\mathbb E_{\mathbb W}\left[e^{\int_0^T\xi^L(\omega_s)-\der\xi^N(\omega_s)\dd s}\right]
$$
Let $h_\der^{N,L}$ be the solution of the equation
$$
\partial_t h_{\der}^{N,L}=\frac{1}{2}\Delta h_{\der}^{N,L}+\frac{1}{2}|\nabla h_\der^{N,L}|^2+\xi^L-\der\xi^N-(c^{\<tree12>}_{N,L}+c_{N,L}^{\<tree124>}),\quad h_\der^{N,L}(0,x)=0.
$$ 
and we apply again Feynmann-Kac formula, so that
$$
\mathbb E_{\mathbb W}\left[e^{\int_0^T\xi^L(\omega_s)-\der\xi^N(\omega_s)\dd s}\right]=e^{h_{\der}^{N,L}(T,x)+T(c^{\<tree12>}_{N,L}+c_{N,L}^{\<tree124>})}\,.
$$
Now we claim that, if we take the constants $c^{\<tree12>}_{N,L}$ and $c^{\<tree124>}_{N,L}$ as 
\begin{equation}\label{eq:constant}
c^{\<tree12>}_{N,L}=c^{\<tree12>}_L+(\der^2-2\der)c^{\<tree12>}_N,\,\quad\text{and}\quad c_{N,L}^{\<tree124>}=c_L^{\<tree124>}+\der(\der^3-4\der^2+5\der^2-4)c_N^{\<tree124>}
\end{equation}
then there exists a function $h$ such that, for all $T\leq T^\star$
\begin{equation}\label{eq:mixed}
\lim_{N\to\infty}\lim_{L\to\infty}h^{N,L}_{\der}(T,x)=h(T,x)\,.
\end{equation}
Assuming~\eqref{eq:mixed} holds true we are done. Indeed,
\small\begin{align*}
\mathbb Q(A_N^r)&\lesssim\liminf\mathbb Q^L(A_N^r)\\
&\lesssim e^{(-c_N^{\<tree12>}\der(1-\der)+\der(\der^3-4\der^2+5\der^2-3)c_N^{\<tree124>})T}e^{h(T,x)}\lesssim e^{-c_N^{\<tree12>}\der(1-\der)T/2}
\end{align*}\normalsize
is valid for $N$ large enough, where the last passages are due to the asymptotic behaviour of the two constants. 
At this point, Borel-Cantelli lemma guarantess that $\mathbb Q(\limsup_{N} A_N^r)=0$, which in turn  implies that $\mathbb Q$ is singular with respect to the Wiener measure. 
\newline

\noindent The only missing ingredient is the proof of the claim~\eqref{eq:mixed}. Recall that 
\small$$
h^{N,L}=\mathcal S_{cKPZ}\left(\xi^L-\der\xi^N,c^{\<tree12>}_{N,L}+c_{N,L}^{\<tree124>}\right)=\mathcal S_{rKPZ}\left(\XX\left(\xi^L-\der\xi^N,c^{\<tree12>}_{N,L},c^{\<tree124>}_{N,L}\right)\right)
$$\normalsize
where the two maps $\mathcal S_{cKPZ}$ and $\mathcal S_{rKPZ}$ were introduced in Theorem~\ref{th:flow-kpz}. To ensure the convergence of $h^{N,L}$ it suffices to exploit the continuity of the map $\mathcal S_{rKPZ}$ and show that there exists a choice of $c^{\<tree12>}_{N,L}$, $c^{\<tree124>}_{N,L}$ for which the sequence $\XX\big(\xi^{L}-\der\xi^N,c^{\<tree12>}_{N,L},c^{\<tree124>}_{N,L}\big)$ converges in $\mathcal X^\varrho$. 

Now, the first two components of $\XX\big(\xi^{L}-\der\xi^N,c^{\<tree12>}_{N,L},c^{\<tree124>}_{N,L}\big)$ are given by
$$
X^{N,L}=\mathcal I(\xi^L-\der\xi^N),\quad X^{N,L,\<tree12>}=\mathcal I(|\nabla X^{N,L}|^2)
$$ 
and, expanding the product at the second term we get
$$
X^{N,L,\<tree12>}=X^{L,\<tree12>}+\der^2X^{N,\<tree12>}-2\der \mathcal I(\nabla X^L\nabla X^N)
$$
where we have set $X^{L,\<tree12>}=\mathcal I(|\nabla X^L|^2)$. 
Since the first two summands at the right hand side were already treated in the Theorem~\ref{th:stoc}, we will only focus one the last one. 
We can assume, without loss of generality, that $L>N+5$ (remember that we want to take the limit in $L$ before the one in $N$). Of course the only ill-defined part of this term in the limit is given by the resonant term $\mathcal I(\nabla X^L\circ\nabla X^N)$. Observe that 
$$
\mathcal I(\nabla X^L\circ\nabla X^N)=\mathcal I(\nabla X^{N+5}\circ\nabla X^N) 
$$
since $\mathscr F(X^L-X^{N+5})$ and $\mathscr F(X^N)$ have disjoint support. 
Moreover the same argument used in Theorem~\ref{th:stoc} allows to show that the limit of the sequence $\mathcal I(\nabla X^{N+5}\nabla X^N-c_N)$ exists in $C_{\kappa} \CC^{2\varrho}$ for all $\kappa>0$ and $\varrho<1/2$.  
Choosing  $c^{\<tree12>}_{N,L}=c_L^{\<tree12>}+(\der^2-2\der)c_N^{\<tree12>}$, the term $X^{N,L,\<tree12>}-\mathcal I(c^{\<tree12>}_{N,L})$ converges, as $L$ tends to $\infty$, to 
\small\begin{multline*}
X^{\<tree12>}+\der^2(X^{N,\<tree12>}-\mathcal I(c_N^{\<tree12>}))+2\der\mathcal I(\nabla (X-X^{N+5})\prec \nabla X^N\\
+\nabla (X-X^{N+5})\succ \nabla X^N+\nabla X^{N+5}\nabla X^N-c_N)
\end{multline*}\normalsize
in $C_T\CC^{2\varrho}$, and the latter converges in the $N\to\infty$-limit in the space $C_T\CC^{2\varrho}$. 

At this point we have proved the convergence of the first two terms of $\XX(\xi^{L}-\der\xi^N,c^{\<tree12>}_{N,L},c^{\<tree124>}_{N,L})$ and this is enough to conclude the proof in the two dimensional case thanks to Remark~\ref{rem:two}. 

By repeating essentially the same argument exploited in the proof of Theorem~\ref{th:stoc}, we see that there exists a constant $c^{\<tree124>}_{N,L}$ auch that $\XX(\xi^{L}-\der\xi^N,c^{\<tree12>}_{N,L},c^{\<tree124>}_{N,L})$ converges as $L$ goes to the infinity to some element $\XX^N\in\mathcal X^\rho$ for all $\rho<1/2$. 
For the same reason as before, we can take the limit in $N$. 
The conclusion now follows by the continuity of the map $\mathcal S_{rKPZ}$. Indeed, $h^{N,L}$ converges in the space $C_T\CC^\alpha$ and of course this in particular implies that 
$$
\lim_{N}\lim_{L}h^{N,L}(t,x)=h(t,x)
$$   
where $h=\mathcal S_{rKPZ}\big(\lim_N\lim_L\XX(\xi^{L}-\der\xi^N,c^{\<tree12>}_{N,L},c^{\<tree124>}_{N,L})\big)$ and the proof of the claim (and therefore of the theorem) is completed.


\begin{bibdiv}
 \begin{biblist}

\bib{AKQ14}{article}{
author={Alberts, T.}, 
author={Khanin, K.},
author={Quastel, J.},
title={The continuum directed random Polymer}
journal={J. Stat. Phys.},
volume={154},
number={1},
year={2014}, 
pages={119--141}
} 

\bib{BCD11}{book}{
author={H. Bahouri},
author={J-Y. Chemin},
author={R. Danchin},
title={Fourier analysis and nonlinear partial differential equations.},
 series={Grundlehren der Mathematischen Wissenschaften [Fundamental Principles of Mathematical Sciences].},
 volume={343},
 publisher={Springer},
 place={Heidelberg},
 date={2011}
 }

%\bib{bol}{article}{
%author={E. Bolthausen},
%author={U. Schmock},
%title={On self attracting d-dimensional random walks},
%journal={Annals of Probability}
%number={25},
%year={1997},
%pages={531-572}
%}

\bib{BC01}{article}{
author={Bass, R.},
author={Chen, Z. Q.},
title={Stochastic differential equations for Dirichlet processes.},
journal={Probab. Theory and Related Fields}, 
number={121},
year={2001},
pages={422--446}
}

\bib{bony}{article}{
author={J.-M. Bony}, 
title= {Calcul symbolique et propagation des singularit\'es pour les \'equations aux d\'eriv\'ees partielles non lin\'eaires}, 
journal={Ann. Sci. \'Ecole Norm. Sup.}, 
number={14}, 
year={1981},
pages={209--246},
}

%\bib{BS}{article}{
%author={D. Brydges },
%author={G. Slade},
%title={The diffusive phase of a model of self-interacting walks},
%journal={Probability Theory and Related fields}
%number={103},
%year={1995},
%pages={285-315}
%}
\bib{CC15}{unpublished}{
author = {Cannizzaro, G.},
author = {Chouk, K.},
    title = {SDEs with Distributional Drift and Polymer Measure with White Noise Potential},
  journal = {ArXiv e-prints},
   eprint = {??},
 keywords = {Mathematics - Probability, Mathematical Physics},
   adsurl = {http://adsabs.harvard.edu/abs/2013arXiv1310.6869C},
  adsnote = {Provided by the SAO/NASA Astrophysics Data System}
}


\bib{CC13}{unpublished}{
author = {Catellier, R.},
author = {Chouk, K.},
    title = {Paracontrolled Distributions and the 3-dimensional Stochastic Quantization Equation. ArXiv 1310.6869, 2013.},
  journal = {ArXiv e-prints},
   eprint = {1310.6869},
 keywords = {Mathematics - Probability, Mathematical Physics},
   adsurl = {http://adsabs.harvard.edu/abs/2013arXiv1310.6869C},
  adsnote = {Provided by the SAO/NASA Astrophysics Data System}
}

\bib{FR}{unpublished}{
author={F. Delarue}
author={R. Diel}
title={Rough paths and 1d sde with a time dependent distributional drift. application to polymers}
journal={Arxiv e-print}
eprint={1402.3662v2}
}

\bib{EK}{book}{
   author={Ethier, S. N.},
   author={Kurtz, T. G.},
   title={Markov Processes, Characterization and Convergence},
   series={Wiley Series in Probability and Mathematical Statistics},
   volume={120},
   %note={Theory and applications},
   publisher={John Wiley \& Sons, Inc.},
   place={New York},
   date={1986},
   %pages={xiv+656},
   %isbn={978-0-521-87607-0},
   %review={\MR{2604669 (2012e:60001)}},
}


\bib{FFF}{unpublished}{
author={F. Flandoli}
author={E. Issoglio} 
author={F. Russo}
title={Multidimensional stochastic differential equations with distributional drift}
journal={ArXiv e-prints}
eprint={1401.6010}
}

\bib{FRW03}{article}{
author={Flandoli, F.}, 
author={Russo, F.}, 
author={Wolf, J.},
title={Some SDEs with distributional drift. I. General calculus.},
journal={Osaka J. of Math.},  
number={40},
year={2003}, 
pages={493–-542}
}

\bib{FRW04}{article}{
author={Flandoli, F.}, 
author={Russo, F.}, 
author={Wolf, J.},
title={Some SDEs with distributional drift. II. Lyons-Zheng structure, It\^o’s formula and semimartingale characterization}
journal={Random Oper. Stochastic Equations},
number ={2},  
year={2004}, 
pages={145--184}
}



\bib{FH14}{book}{
    AUTHOR = {Friz, P. K. },
    AUTHOR = {Hairer, M.},
     TITLE = {A Course on Rough Paths: With an Introduction to Regularity Structures},
    SERIES = {Springer Universitext},
     PUBLISHER = {Springer},
     YEAR = {2014},
     PAGES = {252},
      ISBN = {978-3-319-08331-5},
}


\bib{FrizVictoir}{book}{
   author={Friz, P. K.},
   author={Victoir, N. B.},
   title={Multidimensional stochastic processes as rough paths},
   series={Cambridge Studies in Advanced Mathematics},
   volume={120},
   note={Theory and applications},
   publisher={Cambridge University Press},
   place={Cambridge},
   date={2010},
   %pages={xiv+656},
   %isbn={978-0-521-87607-0},
   %review={\MR{2604669 (2012e:60001)}},
}



\bib{GRR}{article}{
author={A. M. Garsia}, 
author={E. Rodemich}, 
author={H. Rumsey},
title={A real variable lemma and the continuity of paths of some Gaussian processes}, 
journal={Indiana Univ. Math. J.},
volume={20}
year={1970},
number={6}, 
pages={565--578},
}

\bib{Gub04}{article}{
author={M.~Gubinelli},
title={Controlling rough paths},
journal={Journal of Functional Analysis},
volume={216}, 
number={1}, 
year={2004}, 
pages={86--140}
}

\bib{GP}{article}{
	title = {KPZ Reloaded },
	journal = {In Preparation},
	author = {Gubinelli, M.},
	%author={Imkeller, P.},
	author={Perkowski, N.},
journal={ArXiv e-prints}
eprint={1508.03877}
	year = {2015},
}

\bib{GIP15}{article}{
	title = {Paracontrolled distributions and singular PDEs},
	author = {Gubinelli, M.},
	author={Imkeller, P.},
	author={Perkowski, N.},
	journal={Forum of Mathematics, Pi},
volume={3}
year={2015},
number={6}, 
pages={1--75}
}

\bib{gip2}{unpublished}{
	title = {A Fourier approach to pathwise stochastic integration},
	journal = {{arXiv} preprint {arXiv:1410.4006}},
	author = {Gubinelli, M.},
	author={Imkeller, P.},
	author={Perkowski, N.},
	year = {2014},
}

\bib{hairer_solving_2013}{article}{
	title = {Solving the {KPZ} equation},
	volume = {178},
	issn = {0003-{486X}},
	url = {http://annals.math.princeton.edu/2013/178-2/p04},
	doi = {10.4007/annals.2013.178.2.4},
	number = {2},
	Journal= {Annals of Mathematics},
	author = {Hairer, M.},
	year = {2013},
	pages = {559--664}
}

	
\bib{hairer_theory_2013}{article}{	
year={2014},
issn={0020-9910},
journal={Inventiones mathematicae},
doi={10.1007/s00222-014-0505-4},
title={A theory of regularity structures},
url={http://dx.doi.org/10.1007/s00222-014-0505-4},
publisher={Springer Berlin Heidelberg},
keywords={60H15; 81S20; 82C28},
author={Hairer, M.},
pages={1-236},
language={English}
}

\bib{HL1}{unpublished}{
title={A simple construction of the continuum
parabolic Anderson model on $R^2$},
author = {Hairer, M.},
author = {Labb\'e, C.},
year={2015}
journal={{arXiv} preprint {arXiv:1501.00692}}
}

\bib{HL2}{unpublished}{
title={Multiplicative stochastic heat equations on the whole space},
author = {Hairer, M.},
author = {Labb\'e, C.},
year={2015}
journal={{arXiv} preprint {arXiv:1504.07162}}
}


\bib{HW}{unpublished}{
	title = {Large deviations for white-noise driven, nonlinear stochastic PDEs in two and three dimensions},
	journal = {Annales de la Faculté des Sciences de Toulouse},
	author = {Hairer, M.},
	author = {Weber, H.},
volume={24(6)},
number={1},
	year = {2015},
pages={55--92}
	}


\bib{KPZ}{article}{
	title = {Dynamic scaling of growing interfaces},
	volume = {56},
	%issn = {0003-{486X}},
	%url = {http://annals.math.princeton.edu/2013/178-2/p04},
	doi = {10.4007/annals.2013.178.2.4},
	number = {9},
	Journal= {Phys. Rev. Lett.},
	author = {Kardar, M.},
	author = {Parisi, G.},
	author = {Zhang, Y.-C.},
	year = {1986},
	pages = {889--892}
}



 \bib{janson}{book}{,
	title = {Gaussian Hilbert Spaces},
	isbn = {9780521561280},
	language = {en},
	publisher = {Cambridge University Press},
	author = {Janson, S.},
	year = {1997},
}

\bib{KR05}{article}{
author={Krylov, N. V.},
author={R\"ockner, M.},
title={Strong solutions of stochastic equations with singular time dependent drift},
journal={Probab. Theory and Related Fields},
volume={131}, 
year={2005},
pages={154--196}
}  

\bib{L98}{article}{
    AUTHOR = {Lyons, Terry J.},
     TITLE = {Differential equations driven by rough signals},
   JOURNAL = {Rev. Mat. Iberoamericana},
  FJOURNAL = {Revista Matem\'atica Iberoamericana},
    VOLUME = {14},
      YEAR = {1998},
    NUMBER = {2},
     PAGES = {215--310},
}

\bib{MW15}{unpublished}{
author={Mourrat, J.-C.},
author={Weber, H.},
title={Global well-posedness of the dynamic $\Phi^4$ model in the plane}, 
year={2015}
journal={{arXiv} preprint {1501.06191v1}}
}


\end{biblist}
\end{bibdiv}
\end{document}